\documentclass[final,11pt,a4paper]{amsart} 

\usepackage{tikz,tikz-cd} 
\usetikzlibrary{shapes.geometric,positioning}
\usetikzlibrary{arrows,decorations.pathmorphing,decorations.pathreplacing}
\usetikzlibrary{matrix,arrows.meta}
\usepackage{tkz-graph}

\usetikzlibrary{calc,decorations.pathmorphing,shapes}

\newcounter{sarrow}

\tikzset{vertex/.style={circle,fill=black,inner sep=1pt,outer sep=2pt},
         tinyvertex/.style={font=\scriptsize,minimum size=6pt},
         smallvertex/.style={inner sep=1pt, font=\small},
         >=stealth',
         leadsto/.style={-angle 90,decorate,decoration=snake,very thick},
         cut/.style={decorate,decoration=saw,very thick}}
\tikzset{
    partial ellipse/.style args={#1:#2:#3}{
        insert path={+ (#1:#3) arc (#1:#2:#3)}
    }
}

\usepackage{graphicx}
\usepackage[all]{xy}
\usepackage{placeins}
\usepackage{enumitem}
\usepackage{amssymb}
\usepackage{latexsym}
\usepackage{amsmath}
\usepackage{mathrsfs}
\usepackage{array,booktabs}
\usepackage{verbatim}
\usepackage{fullpage}
\usepackage[notref, notcite]{showkeys}  
\usepackage{color}
\usepackage[normalem]{ulem}
\usepackage{float}

\usepackage{tabularx}
\newcolumntype{E}{>{\hsize=0.5cm \centering\arraybackslash}X}%
\newcolumntype{C}[1]{>{\hsize=#1\hsize \centering\arraybackslash}X}%

\usepackage{hyperref}
\hypersetup{
    colorlinks,
    linkcolor={red!50!black},
    citecolor={blue!50!black},
    urlcolor={blue!80!black}
}

\newcommand{\harxiv}[1]{ \href{http://arxiv.org/abs/#1}{\texttt{arXiv:#1}}}
\newcommand{\hyref}[2]{ \hyperref[#2]{#1~\ref*{#2}} }

\newcommand{\coloneqq}{\mathrel{\mathop:}=}

\newcommand{\Canakci}{\c{C}anak\c{c}\i}
\newcommand{\Ilke}{\.{I}lke }
\newcommand{\I}{\.{I}} 

\theoremstyle{plain}
\newtheorem{theorem}{Theorem}[section]

\newtheorem{lemma}[theorem]{Lemma}
\newtheorem{corollary}[theorem]{Corollary}
\newtheorem{proposition}[theorem]{Proposition}
\newtheorem{introtheorem}{Theorem}

\theoremstyle{definition}
\newtheorem{remark}[theorem]{Remark}

\newtheorem{example}[theorem]{Example}
\newtheorem{definition}[theorem]{Definition}
\newtheorem{definitions}[theorem]{Definitions}
\newtheorem{setup}[theorem]{Setup}
\newtheorem*{introremark}{Remark}

\newcommand{\sD}{\mathsf{D}} 
\newcommand{\sK}{\mathsf{K}} 

\newcommand{\ba}{\bar{a}}
\newcommand{\bb}{\bar{b}}
\newcommand{\bc}{\bar{c}}
\newcommand{\bg}{\bar{g}}
\newcommand{\bh}{\bar{h}}
\newcommand{\bk}{\bar{k}}
\newcommand{\bm}{\bar{m}}
\newcommand{\bn}{\bar{n}}

\newcommand{\Db}{\sD^b}
\newcommand{\KminusL}{\sK^{b,-}(\proj \Lambda)}

\newcommand{\kk}{{\mathbf{k}}}

\renewcommand{\emptyset}{\varnothing}

\DeclareMathOperator{\im}{\mathrm{im}}
\DeclareMathOperator{\Hom}{\mathrm{Hom}}
\DeclareMathOperator{\Ext}{\mathrm{Ext}}
\newcommand{\proj}[1]{\mathrm{proj}(#1)}
\newcommand{\rad}{{\rm rad}} 

\newcommand{\iQ}{\overline{Q}}

\newcommand{\inverse}{\mathrm{inv}}
\newcommand{\direct}{\mathrm{dir}}
\newcommand{\length}{\mathrm{length}}

\newcommand{\B}{B^\bullet}
\renewcommand{\P}{P^\bullet}
\newcommand{\Q}{Q^\bullet}
\newcommand{\E}{E^\bullet}

\newcommand{\f}{f^\bullet}
\newcommand{\g}{g^\bullet}
\newcommand{\h}{h^\bullet}

\newcommand{\too}{\longrightarrow}
\newcommand{\rightlabel}[1]{\stackrel{#1}{\longrightarrow}}
\newcommand{\wiggle}{\rightsquigarrow}
\newcommand{\onto}{\twoheadrightarrow}

\newcommand{\xydot}{{\bullet}}
\newcommand{\arr}{\ar@{-}[r]}

\newenvironment{pmat}{\left[ \begin{smallmatrix}}{\end{smallmatrix} \right]}

\renewcommand{\phi}{\varphi}
\renewcommand{\epsilon}{\varepsilon}

\begin{document}

\title[Extensions]{On extensions for gentle algebras}

\author{\Ilke \Canakci}
\address{Department of Mathematics, VU Amsterdam, Amsterdam 1081 HV, The Netherlands.}
\email{i.canakci@vu.nl}

\author{David Pauksztello} 
\address{Department of Mathematics and Statistics, Lancaster University, Lancaster, LA1 4YF, United Kingdom}
\email{d.pauksztello@lancaster.ac.uk}

\author{Sibylle Schroll}
\address{Department of Mathematics, University of Leicester, University Road, Leicester, LE1 7RH, United Kingdom.}
\email{schroll@leicester.ac.uk}



\begin{abstract}
We give a complete description of a basis of the extension spaces between indecomposable string and quasi-simple band modules in the module category of a gentle algebra.
\end{abstract}

\maketitle

\section*{Introduction}

The representation theory of finite-dimensional algebras plays an important role in many different areas of mathematics, such as in Lie theory, in number theory in connection with the  Langlands program and automorphic forms, in geometry ranging from invariant theory to non-commutative resolutions of singularities and as far afield  as  harmonic analysis where the representation theory of $S^1$ appears in the guise of Fourier analysis.

Most finite-dimensional algebras are of wild representation type, that is their representation theory is at least as complicated as that of the free associative algebra in two generators. An algebra that is not wild is of tame representation type.
One particular class of tame algebras, the so-called \emph{ gentle algebras} appear in a surprising number of different contexts. For example, in the context of Fukaya categories related to Kontsevich's homological mirror symmetry program \cite{HKK, LP, OPS},  of dimer models \cite{Bocklandt},  of the enveloping algebras of  Lie algebras  \cite{HK}, and in the context of
cluster theory as \sloppy \linebreak ($m$-)cluster tilted and $m$-Calabi Yau tilted algebras and also as Jacobian algebras associated to unpunctured surfaces \cite{ABCP, BCS,Garcia, LF}. Furthermore, the class of derived-discrete algebras consists of gentle algebras \cite{Vossieck}. 

But there are many other reasons why gentle algebras have been studied extensively. One of the main reasons being that   they are string algebras and that their indecomposable representations are classified by string and band modules \cite{WW}, see also \cite{BR}. The associated string combinatorics  governs the representation theory of gentle algebras, examples of this are  the classification of  morphisms between string and band modules \cite{CB, Krause} and a characterisation of almost split sequences  in terms of string combinatorics \cite{BR}.  Over last few years, interest in gentle algebras has intensified with many new results appearing,  an example of this is the recent work \cite{PPP},  where string combinatorics is used to classify support $\tau$-tilting modules.

Another reason for  the  extensive investigation of gentle algebras is the fact that they are derived tame and the indecomposable objects in the derived category of a gentle algebra have been classified. They are given by the so-called homotopy strings and bands \cite{BM}. In \cite{ALP} the morphisms between string and band complexes  
in the derived category of a gentle algebra were characterised in terms of homotopy string combinatorics and in \cite{CPS,Addendum} a graphical mapping cone calculus based on the morphisms described in \cite{ALP} was developed. 

Extensions between modules are one of the fundamental  cohomological tools. Not only do they play an essential role in the definition of, for example, the Yoneda algebra or Hochschild cohomology, they are also essential in many of the newer developments in  representation theory such as in  cluster tilting in cluster theory. 

The projective resolutions of indecomposable modules over gentle algebras are well understood, see, for example, \cite{Kalck}. So it is surprising that up to now, in general, no complete combinatorial description of the extensions between indecomposable modules over a gentle algebra is known.  A description of certain combinatorially defined extensions between string modules was given in \cite{Schroer}, we will refer to these extensions as \emph{arrow and overlap extensions}. In \cite{Zhang} it was shown that the existence of such extensions is a necessary and sufficient condition for the non-vanishing of the $\Ext^1$-space. 

However, it has remained an open problem for almost twenty years whether these extensions form a basis of the $\Ext^1$-space between string modules and what the extensions involving band modules are. In fact, it has become apparent that string combinatorics in the module category of a gentle algebra might not be enough to answer this question. This has further been confirmed by the recent results in \cite{CS} where based on arguments using the associated cluster category, it was shown that in the context of gentle Jacobian algebras of quivers with potential, the extensions between string modules described in \cite{Schroer} do indeed give a basis.

In this paper, we  answer this open question by giving, for any gentle algebra, a basis of the extension space between indecomposable modules.  More precisely, we explicitly determine the cohomology of the indecomposable objects in the bounded derived category of a gentle algebra given in terms of homotopy strings and bands.  Building on this we give a complete description of the extension space between string and quasi-simple band modules by giving a combinatorial description of a basis. We do this by working not in the module category of a gentle algebra, but we transfer the problem into the derived category, where we are able to use the graphical mapping cone calculus developed in \cite{CPS,Addendum}.
We now state our main results; 
for the relevant definitions and details on notation, see Section~\ref{sec:background} and Definition~\ref{def:extensions}. Throughout the following, $\kk$ will be an algebraically closed field and $\Lambda = \kk Q/I$ will be a gentle algebra.

\begin{introtheorem} \label{thm:strings}
Let $\Lambda$ be a gentle algebra and $v$ and $w$ strings with $M(v)$ and $M(w)$ the corresponding string modules over $\Lambda$. 
The collection of arrow and overlap extensions of $M(v)$ by $M(w)$ form a basis of $\Ext^1_\Lambda(M(v), M(w))$.
\end{introtheorem}

In the following we give a graphical presentation of the strings in an arrow and overlap extension.

\begin{figure}[H]
\[ \resizebox{0.9\textwidth}{!}{
\begin{tikzpicture}[node distance=1cm and 1.5cm]
\coordinate[label=left:{}] (1);
\coordinate[right=1cm of 1] (2);
\coordinate[above right=.8cm of 2,label=right:{}] (3);
\coordinate[right=1cm of 3] (4);
\coordinate[below right=.8cm of 4] (5);
\coordinate[right=1cm of 5] (6);
\coordinate[above right=2cm of 6] (7);
\coordinate[right=1cm of 7] (8);
\coordinate[below right=.8cm of 8] (9);
\coordinate[right=1cm of 9] (10);
\coordinate[below right=.8cm of 10] (11);
\coordinate[right=1cm of 11] (12);
\coordinate[right=11.3cm of 1] (13);
\coordinate[right=1cm of 13] (14);
\coordinate[below right=.8cm of 14] (15);
\coordinate[right=1cm of 15] (16);
\coordinate[above right=.8cm of 16] (17);
\coordinate[right=1cm of 17] (18);

\coordinate[right=.25cm of 6] (6'');
\coordinate[right=1cm of 6''] (7'');

\coordinate[right=2.5cm of 6,label=right:{$\oplus$}] (19);

\coordinate[below right=2cm of 6] (7');
\coordinate[right=1cm of 7'] (8');
\coordinate[above right=.8cm of 8'] (9');
\coordinate[right=1cm of 9'] (10');
\coordinate[above right=.8cm of 10'] (11');
\coordinate[right=1cm of 11'] (12');
\coordinate[above right=.8cm of 12'] (13');

\coordinate[right=5.8cm of 6] (12'');
\coordinate[right=1cm of 12''] (13'');

\coordinate[right=-.3cm of 1] (1');
\coordinate[right=-1cm of 1'] (2');
\coordinate[right=-.6cm of 2',label=right:{$0$}] (3');

\coordinate[right=.3cm of 18] (18');
\coordinate[right=1cm of 18'] (19');
\coordinate[right=.2cm of 19',label=right:{$0$}] (19''');

\draw[thick,->,blue] (2')--(1');
\draw[thick,->,blue] (6'')--(7'');
\draw[thick,->,blue] (12'')--(13'');
\draw[thick,->,blue] (18')--(19');

\draw[thick,decorate,decoration={snake,amplitude=.4mm,segment length=2mm}] 
(1)-- node [anchor=north,scale=.9]{$w_L$} (2) 
(5)-- node [anchor=north,scale=.9]{$w_R$} (6) 
(7)-- node [anchor=south,scale=.9]{$v_L$} (8) 
(7')-- node [anchor=north,scale=.9]{$w_L$} (8') 
(11)-- node [anchor=south,scale=.9]{$w_R$} (12) 
(11')-- node [anchor=north,scale=.9]{$v_R$} (12') 
(13)-- node [anchor=north,scale=.9]{$v_L$} (14) 
(17)-- node [anchor=north,scale=.9]{$v_R$} (18);

\draw[thick] 
(3)-- node [anchor=south,scale=.9]{$m$} (4)
(9)-- node [anchor=south,scale=.9]{$m$} (10)
(9')-- node [anchor=north,scale=.9]{$m$} (10')
(15)-- node [anchor=north,scale=.9]{$m$} (16);

\draw[thick,->] (3)--node [anchor=south east,scale=.9]{$D$}(2); 
\draw[thick,->] (4)--node [anchor=south west,scale=.9]{$\bar{C}$}(5);
\draw[thick,->] (8)--node [anchor=south west,scale=.9]{$\bar{B}$}(9);
\draw[thick,<-] (8')--node [anchor=north west,scale=.9]{$D$}(9');
\draw[thick,->] (10)--node [anchor=south west,scale=.9]{$\bar{C}$}(11);
\draw[thick,<-] (10')--node [anchor=north west,scale=.9]{$A$}(11');
\draw[thick,->] (14)--node [anchor=south west,scale=.9]{$\bar{B}$}(15);
\draw[thick,->] (17)--node [anchor=south east,scale=.9]{$A$}(16);
\end{tikzpicture} }
\]
\end{figure}

\begin{figure}[H]
\[ \resizebox{0.6\textwidth}{!}{
\begin{tikzpicture}[node distance=1cm and 1.5cm]
\coordinate[label=left:{}] (1);
\coordinate[right=1cm of 1] (2);

\coordinate[right=1.3cm of 1] (3);
\coordinate[right=1cm of 3] (4);

\coordinate[below right=.4cm of 4] (5);
\coordinate[right=1cm of 5] (6);

\coordinate[above right=1cm of 6] (7);
\coordinate[right=1cm of 7] (8);

\coordinate[right=3cm of 4] (9);
\coordinate[right=1cm of 9] (10);

\coordinate[right=.3cm of 10] (11);
\coordinate[right=1cm of 11] (12);

\coordinate[right=-.3cm of 1] (1');
\coordinate[right=-1cm of 1'] (2');
\coordinate[right=-.7cm of 2',label=right:{$0$}] (3');

\coordinate[right=.3cm of 12] (12');
\coordinate[right=1cm of 12'] (13');
\coordinate[right=.3cm of 13',label=right:{$0$}] (14');

\draw[thick,->,blue] (3)--(4);
\draw[thick,->,blue] (9)--(10);
\draw[thick,->,blue] (2')--(1');
\draw[thick,->,blue] (12')--(13');

\draw[thick,decorate,decoration={snake,amplitude=.4mm,segment length=2mm}] 
(1)-- node [anchor=north,scale=.9]{$w$} (2) 
(5)-- node [anchor=south,scale=.9]{$w$} (6) 
(7)-- node [anchor=south,scale=.9]{$v$} (8) 
(11)-- node [anchor=north,scale=.9]{$v$} (12);

\draw[thick,<-] 
(6)-- node [anchor=south,scale=.9]{$a$} (7);
\end{tikzpicture} }
\]
\caption{Presentation in terms of strings of an overlap extension (top picture) and an arrow extension (bottom picture) where for an arrow $a \in Q_1$ we denote its formal inverse by $\bar{a}$.} \label{fig:ext}
\end{figure}

We note that in concurrent work \cite{BDMTY}, which builds on \cite{MCC}, a basis for extensions between string modules over a gentle algebra is also given using different techniques.

When a band is involved there are no arrow extensions, only overlap extensions. An extension involving both a string module and a band module has only one indecomposable module as its middle term. An extensions involving two band modules can have as its middle term the direct sum of many indecomposable band modules. The following theorems describe the situation involving band modules more precisely. Given a band $b$ and a scalar $\mu \in \kk^*$, we denote the associated quasi-simple band module by $B(b,\mu)$. A useful comparison for the following statements is the corresponding statements for mapping cones of quasi-graph maps involving a band complex given in \cite{Addendum}.
In the following, for a band $b$, denote by ${}^\infty b^\infty$ (resp. ${}^\infty b$, resp. $b^\infty$) the string obtained from $b$ by repeatedly concatenating $b$ with itself both on the left and on the right (resp. on the left, resp. on the right).

\begin{introtheorem}\label{thm:band-string}
Let  $\Lambda$ be a gentle algebra, $v$ be a string and $(b,\mu)$  be a band with $\mu \in \kk^*$. Suppose that $v$ and ${}^\infty b^\infty$ admit decompositions 
\[
v  = v_L \bar{B} m A v_R \text{ and } {}^\infty b^\infty = {}^\infty b b_L D m \bar{C} b_R b^\infty,
\]
where $A,B,C,D \in Q_1$ with $C \neq \emptyset \neq D$ and $v_L$, $v_R$, $m$, $b_L$ and $b_R$ are (possibly trivial) strings satisfying the conditions of Definition~\ref{def:extensions}\ref{overlap}.
\begin{enumerate}[label=(\alph*)]
\item If, after suitable rotation of $b$, $m$ is a proper subword of $b$, then there is a non-split overlap extension 
\[ 
0 \to B(b,\mu) \to M(u) \to M(v) \to 0,
\]
where $u =  v_L \bar{B} m \bar{C} b_R b_L D m A v_R$ is a string.
\item If $b$ is a subword of $m$, then after suitable rotation of $b$ there is a decomposition $b = b_2 b_1$ such that $m = b^k b_2$ for some $k \geq 1$ and there is a non-split overlap extension
\[ 
0 \to B(b,\mu) \to M(u) \to M(v) \to 0,
\]
where $u = v_L \bar{B} b^{k+1} b_2 A v_R$ is a string.
\end{enumerate}
Moreover, the collection of such extensions forms a basis of $\Ext^1_\Lambda(M(v), B(b, \mu ))$.
\end{introtheorem}

\begin{introtheorem}\label{thm:string-band}
Let $\Lambda$ be a gentle algebra,  $(c,\lambda)$ be a band with $\lambda \in \kk^*$ and $w$ be a string. Suppose that ${}^\infty c^\infty$ and $w$ admit decompositions
\[
{}^\infty c^\infty = {}^\infty c c_L \bar{B} m A c_R c^\infty \text{ and } w  = w_L D m \bar{C} w_R, 
\]
where $A,B,C,D \in Q_1$ with $A \neq \emptyset \neq B$ and $c_L$, $c_R$, $m$, $w_L$ and $w_R$ are (possibly trivial) strings satisfying the conditions of Definition~\ref{def:extensions}\ref{overlap}.
\begin{enumerate}[label=(\alph*)]
\item If, after suitable rotation of $c$, $m$ is a proper subword of $c$, then there is a non-split overlap extension
\[ 
0 \to M(w) \to M(u) \to B(c,\lambda) \to 0,
\]
where $u =  w_L D m A c_R c_L \bar{B} m \bar{C} w_R$ is a string.
\item If $c$ is a subword of $m$, then after suitable rotation of $c$ there is a decomposition $c = c_2 c_1$ such that $m = c^\ell c_2$ for some $\ell \geq 1$ and there is a non-split overlap extension
\[ 
0 \to M(w) \to M(u) \to B(c,\lambda) \to 0,
\]
where $u = w_L D c^{\ell +1} c_2 \bar{C} w_R$ is a string.
\end{enumerate}
Moreover, the collection of such extensions forms a basis of $\Ext^1_\Lambda(B(c, \lambda),M(w))$.
\end{introtheorem}

\begin{introtheorem}\label{thm:band-band}
Let $\Lambda$ be a gentle algebra and $(b,\mu) \neq (c,\lambda)$ be bands with $\lambda,\mu \in \kk^*$. Suppose that ${}^\infty c^\infty$ and ${}^\infty b^\infty$ admit decompositions
\[
{}^\infty c^\infty = {}^\infty c c_L \bar{B} m A c_R c^\infty \text{ and } {}^\infty b^\infty = {}^\infty b b_L D m \bar{C} b_R b^\infty,
\]
where $A,B,C,D \in Q_1$ are each nonempty and $c_L$, $c_R$, $m$, $b_L$ and $b_R$ are (possibly trivial) strings satisfying the conditions of Definition~\ref{def:extensions}\ref{overlap}. 
Then, either
\begin{enumerate}[label=(\alph*)]
\item $m$ is a proper subword of $b$, i.e. after suitable rotation of $b$ there is a decomposition $b = m v$; or,
\item $b$ is a subword of $m$, i.e. after suitable rotation there is a decomposition $b = b_2 b_1$ such that $m = b^k b_2$,
\end{enumerate}
and, either,
\begin{enumerate}[label=(\alph*),resume]
\item $m$ is a proper subword of $c$, i.e. after suitable rotation of $c$ there is a decomposition $c = m w$; or,
\item $c$ is a subword of $m$, i.e. after suitable rotation there is a decomposition $c = c_2 c_1$ such that $m = c^\ell c_2$.
\end{enumerate}
Then, there is a band $d$ and an integer $t \geq 1$ such that 
\[
d^t = 
\begin{cases}
mvmw            & \text{ if (a) \& (c);} \\
mvc_2c_1        & \text{ if (a) \& (d);} \\
b_2 b_1 mw      & \text{ if (b) \& (c);} \\
b_2 b_1 c_2 c_1 & \text{ if (b) \& (d),}
\end{cases}
\]
and a non-split overlap extension
\[
0 \to B(b,\mu) \to \bigoplus_{i=1}^t B(d, \omega^i \sqrt[t]{\pm \lambda\mu^{-1}}) \to B(c,\lambda) \to 0,
\]
where $\omega$ is a primitive $t^{\it th}$ root of unity.
Moreover, the collection of such extensions forms a basis of $\Ext^1_\Lambda(B(c, \lambda),B(b,\mu))$.
\end{introtheorem}

\begin{introtheorem}\label{thm:same-band}
Let $\Lambda$ be a gentle algebra and $(b,\mu)$ be a band with $\mu \in \kk^*$. The collection of extensions in Theorem~\ref{thm:band-band} in which $m \neq b$ together with the Auslander--Reiten sequence,
\[
0 \to B(b,\mu) \to B(b,\mu,\kk^2) \to B(b,\mu) \to 0,
\] 
where $B(b,\mu,\kk^2)$ denotes the $2$-dimensional band module with Jordan block whose eigenvalue is $\mu$, form a basis of $\Ext^1_\Lambda(B(b, \mu),B(b,\mu))$.
\end{introtheorem}

\begin{introremark}
In Theorem~\ref{thm:band-band}, each of the words defining $d^t$ is, after suitable rotation of $b$ and $c$ just the concatenation of the two bands, $b c$. However, different possibilities for $d$ arise from the precise decompositions of $b$ and $c$: for different $m$, concatenations $bc$ with respect to different decompositions need not be equivalent up to inverting the word or cyclic permutation.
\end{introremark}

We now briefly outline the content of the paper, including  the general strategy of the proofs of Theorems~\ref{thm:strings}, \ref{thm:band-string}, \ref{thm:string-band}, \ref{thm:band-band} and \ref{thm:same-band}. Let $\Lambda$ be a gentle algebra. We begin by recalling  the basic notions of string and homotopy string combinatorics for gentle algebras in Section~\ref{sec:background}.  In  Section~\ref{sec:cohomology} we determine the homotopy string or band of the minimal projective resolution of a string or band module  over $\Lambda$ and the cohomology of a string or band complex in $\KminusL$.  

In order to describe the content of Sections~\ref{sec:determining extensions in module category} and \ref{sec:surjective} more precisely, fix the following notation. Let $v$ and $w$ be strings or bands and $M(v)$ and $M(w)$ the corresponding string or quasi-simple band modules. We denote the homotopy strings or bands of their projective resolutions by $\pi(v)$ and $\pi(w)$ and the corresponding string or band complexes by $\Q_{\pi(v)}$ and $\Q_{\pi(w)}$. 
The standard basis of homomorphisms between string and/or band complexes is recalled from \cite{ALP} in Section~\ref{sec:basis}, enabling us to give an explicit description of a basis of
$ \Hom_{\KminusL}(\Q_{\pi(v)}, \Sigma \Q_{\pi(w)})$. 

In the first step in the proof, we show in Section~\ref{sec:determining extensions in module category} that the image of every element of the standard basis under the canonical isomorphism
\begin{equation}\label{isomorphism}
 \Phi: \Hom_{\KminusL}(\Q_{\pi(v)}, \Sigma \Q_{\pi(w)}) \stackrel{\sim}{\to} \Ext^1_\Lambda (M(v),M(w))
 \end{equation}
is either an overlap or an arrow extension. In particular, this shows that the set of overlap and arrow extensions form a generating set for $\Ext^1_\Lambda (M(v),M(w))$.

The second step of the proof, comprising Section 4, shows that the set of overlap and arrow extensions forms a basis of $\Ext^1_\Lambda (M(v),M(w))$. To see this, we show that $\Phi$ restricts to a surjection from the standard basis of $\Hom_{\KminusL}(\Q_{\pi(v)}, \Sigma \Q_{\pi(w)})$ to the set of arrow and overlap extensions in $\Ext^1_\Lambda (M(v),M(w))$.

We emphasise that, with the exception of the case highlighted in the remark above, the methods apply equally to (homotopy or classical) strings and bands.
Furthermore, for ease of the already somewhat heavy notation, in the proofs in Section 3 and 4, whenever we have a map between two band complexes or an extension between two band modules, implicitly and without loss of generality we assume that the parameters of the corresponding band complexes or band modules are equal to one, see \cite[\S 2.3]{CPS} for more details on the placement of parameters with respect to mapping cones.

\subsection*{Acknowledgments}

The second author would like to thank Raquel Coelho Sim\~oes and Rosanna Laking for useful comments and corrections. The authors would also like to thank an anonymous referee for a thorough reading of the article and many useful comments that have significantly improved the exposition.
This work has been supported by the EPSRC through the grants EP/K026364/1, EP/K022490/1 and EP/N005457/1. The third author is supported by the EPSRC through an Early Career Fellowship EP/P016294/1.

\section{Background} \label{sec:background}

In this section we briefly recall the definition of gentle algebras, background on string and band modules, string and band complexes and the standard basis of the morphism spaces between string and band complexes that will be needed in the article.

\subsection{Gentle algebras}

Throughout, $\kk$ will be an algebraically closed field. We recall the following definition from \cite{AS}.

\begin{definition}
A finite-dimensional $\kk$-algebra $\Lambda$ is \emph{gentle} if it is Morita equivalent to a bound path algebra $\kk Q/I$, where $Q$ is a quiver and $I$ an admissible ideal in $\kk Q$ such that
\begin{enumerate}
\item for each vertex $i \in Q_0$ there are at most two arrows starting at $i$ and at most two arrows ending at $i$;
\item for each arrow $a \in Q_1$ there is at most one arrow $b$ with $e(a) = s(b)$ and such that $ba \notin I$ and at most one arrow $c$ with $e(c)= s(a)$ and such that $ac \notin I$;
\item for each arrow $a \in Q_1$ there is at most one arrow $b$ with $e(a) = s(b)$ and such that $ba \in I$ and at most one arrow $c$ with  $e(c)= s(a)$ and such that $ac \in I$;
\item the ideal $I$ is generated by length-two monomial relations.
\end{enumerate}
\end{definition}

From now on $\Lambda = \kk Q/I$ will be a gentle algebra.

\subsection{String and band modules} \label{sec:strings-and-bands}

We now describe strings and bands, which parametrise the indecomposable $\Lambda$-modules. The reference for this material is \cite{BR, WW}. Note that, in this paper all modules will be finitely generated left modules, and therefore paths in the quiver will be read from right to left.

For each arrow $a \in Q_1$ we introduce a formal inverse arrow $\overline{a} = a^{-1}$ with $s(\overline{a}) = e(a)$ and $e(\overline{a}) = s(a)$. We write $\iQ_1$ for the set of formal inverse arrows. Similarly for a path $p = a_n \cdots a_1$  the inverse path is $\overline{p} = \overline{a}_1 \cdots \overline{a}_n$.
Sometimes we shall assert the nonexistence of an arrow or inverse arrow $a$, and in this case we write $a = \emptyset$.

\begin{definitions}
We recall the following notions.
\begin{enumerate}
\item A \emph{walk} of length $l>0$ in $(Q,I)$ is a sequence $w = w_l \cdots w_1$ satisfying $s(w_{i+1}) = e(w_i)$, where each $w_i$ is either an arrow or an inverse arrow, and where the sequence does not contain any subsequence of the form $a \overline{a}$ or $\overline{a} a$ for an arrow $a \in Q_1$. 
We will call each arrow or inverse arrow $w_i$ in $w$ a \emph{letter} of $w$.
\item A \emph{string} is a walk that does not contain subwalks $v$ such that $v \in I$ or $\overline{v} \in I$. In addition, there are \emph{trivial strings} $1_x$ for each vertex $x \in Q_0$.
\item A \emph{band} is a string $b = b_n \cdots b_1$ such that $e(b_n) = s(b_1)$, $b_1 \neq \overline{b_n}$, $b_1 b_n$ is defined as a string, and $b \neq v^m$ for some substring $v$ and $m > 1$.
\end{enumerate}
\end{definitions}

Modulo the equivalence relation $w \sim \overline{w}$ the strings form an indexing set for the so-called \emph{string modules}. Given a string $w$, we write $M(w)$ for the corresponding string module. Note that if $w = 1_x$ is a trivial string $M(w) = S(x)$ is the simple module at $x$. We refer to \cite{BR, WW} for more details on how to construct string modules from strings.

Modulo the equivalence relation given by inversion and cyclic permutation (rotation), the bands together with scalars $\mu \in \kk^*$ form an indexing set for the so-called \emph{band modules},
$B(b,\mu)$,
where by convention we place $\mu$ on a direct arrow.
By abuse of notation, we will usually drop the scalar and write simply $B(b)$ for the corresponding band module. 
Again we refer to \cite{BR} for the actual construction of the band modules.

In order to deal with the word combinatorics involving bands effectively, we will need to consider infinite periodic words corresponding to bands. Let $b$ a band, we write
\begin{align*}
{}^\infty b^\infty & = \cdots \underbrace{b_n \cdots b_1}_b \underbrace{b_n \cdots b_1}_b \underbrace{b_n \cdots b_1}_b \cdots, \\ 
{}^\infty b & = \cdots \underbrace{b_n \cdots b_1}_b \underbrace{b_n \cdots b_1}_b, \quad \text{and,} \\
b^\infty & = \underbrace{b_n \cdots b_1}_b \underbrace{b_n \cdots b_1}_b \cdots.
\end{align*}
In particular, let $(b,\mu)$ and $(c,\lambda)$ be bands, then $(b,\mu) = (c,\lambda)$ if and only if ${}^\infty b^\infty = {}^\infty c^\infty$ or ${}^\infty b^\infty = {}^\infty (c^{-1})^\infty$ and $\lambda = \mu$ with both $\lambda$ and $\mu$ placed on a direct arrow in the infinite words that are equal.

By \cite[Prop. 2.3]{WW}, the string and band modules form a complete set of isomorphism classes of indecomposable $\Lambda$-modules.

The band modules given by representations in which each vertex is replaced by a $1$-dimensional vector space all lie at the mouth of homogeneous tubes and are referred to as quasi-simple (band) modules. They can be characterised as those band modules $B$ such that there exists an almost split sequence of the form $0 \to B \to E \to \tau^{-1} B \to 0$ where $E$ is indecomposable,  see for example \cite{SS}. In the following by abuse of notation, whenever we will use the term band module we will be  referring to a quasi-simple band module.

\subsection{String and band complexes}

We now describe homotopy strings and bands, which parametrise the indecomposable complexes in the derived category $\Db(\Lambda)$. We will use the notation and terminology employed in \cite{ALP,CPS} and the references therein. However, for the sake of brevity we drop some of the formality of \cite{ALP,CPS} regarding the degrees. 

\begin{definitions}
The original reference for the following definitions is \cite{BM}. 
\begin{enumerate}
\item A \emph{(finite) homotopy string} is a walk of finite length in $(Q,I)$. In addition, there are \emph{trivial homotopy strings} for each vertex $x \in Q_0$. 
\item A subwalk $p = w_j \cdots w_i$  of a homotopy string  $\sigma = w_l \cdots w_1$  is a \emph{homotopy letter} if 
\begin{enumerate}
\item $p$ or $\overline{p}$ is a path of length at least one in $(Q,I)$; and,
\item $w_i \in Q_1$ and $w_{i-1} \in \iQ_1$ or vice versa, or $w_i w_{i-1} \in I$, or $\overline{w_{i-1}} \overline{w_i} \in I$; and,
\item $w_j \in Q_1$ and $w_{j+1} \in \iQ_1$ or vice versa, or $w_{j+1} w_j \in I$, or $\overline{w_j} \overline{w_{j+1}} \in I$.
\end{enumerate}
We say that $p$ is a \emph{direct homotopy letter} if it is a path in $(Q,I)$ and an \emph{inverse homotopy letter} if $\overline{p}$ is a path in $(Q,I)$.
In this way we partition a homotopy string $\sigma$ into homotopy letters and write $\sigma = \sigma_n \cdots \sigma_1$ for this decomposition. A \emph{homotopy subletter} of $p$ is  a subwalk of $p$ of length at least one. 
\item A homotopy letter $p = w_l \cdots w_1$, with $w_i \in Q_1$ for $i = 1, \ldots, l$ or $\bar{w}_i \in Q_1$ for $i = 1, \ldots, l$, is said to have \emph{length} $l$ and we write $\length(p) = l$. The length can be zero, in which case $p = 1_x$ for some $x \in Q_0$ and $p$ is called a \emph{trivial homotopy letter}. Sometimes we shall assert the nonexistence of homotopy letters, and in this case we write $p = \emptyset$.

\item Let $\sigma = \sigma_n \cdots \sigma_1$ be a homotopy string decomposed into its homotopy letters. A subwalk $\tau = \sigma_j \cdots \sigma_i$ with $1 \leq i \leq j \leq n$ is called a \emph{homotopy substring} of $\sigma$.
\item A \emph{homotopy band} is a homotopy string $\sigma = \sigma_n \cdots \sigma_1$ with $s(\sigma) = e(\sigma)$, $\sigma_1 \neq \bar{\sigma}_n$, $\sigma \neq \tau^m$ for some homotopy substring $\tau$ and $m >1$, and $\sigma$ has equal numbers of direct and inverse homotopy letters. 
\end{enumerate}
\end{definitions}

\begin{remark}
Throughout the article, whenever we write a walk using Greek letters, such as $\sigma = \sigma_n \cdots \sigma_1$, we will always mean its decomposition into homotopy letters whereas, in general, we reserve Roman letters for (classical) strings and bands. 
\end{remark}

Modulo the equivalence relation $\sigma \sim \overline{\sigma}$ the homotopy strings form an indexing set for the so-called \emph{string complexes}. Given a homotopy string $\sigma$, we write $\P_\sigma$ for the corresponding string complex. Note that if $\sigma = 1_x$ is a trivial homotopy string $\P_\sigma = P(x)$ is the stalk complex of the projective module at $x$.
We refer to \cite{ALP,BM} for more details on how to construct string complexes from homotopy strings; for a sketch of the constructions, see Example~\ref{ex:string-complex} below.

Modulo the equivalence relation given by inversion and cyclic permutation, the homotopy bands together with scalars $\lambda \in \kk^*$ form an indexing set for the so-called \emph{band complexes} $\B_{\sigma,\lambda}$. 
Again we refer to \cite{ALP,BM} for the actual construction of the band complexes.

By \cite[Thm. 3]{BM}, the string and band complexes form a complete set of indecomposable perfect complexes in $\Db(\Lambda)$.  For the remaining objects of $\Db(\Lambda)$ we need some further terminology.

\begin{example}[{\cite[Running Example]{ALP}}] \label{ex:string-complex}
  Let $\Lambda = \kk Q/I$ be given by the following bound quiver:
\[
\begin{tikzpicture}
  \node (0) at (0,0) [smallvertex] {$0$};
  \node (1) at (-1,0.8) [smallvertex] {$1$};
  \node (2) at (-1,-0.8) [smallvertex] {$2$};
  \node (3) at (1,0.8) [smallvertex] {$3$};
  \node (4) at (1,-0.8) [smallvertex] {$4$};
  \draw [->] (0) -- node [above, tinyvertex] {$a$} (1);
  \draw [->] (1) -- node [left, tinyvertex] {$b$} (2);
  \draw [->] (2) -- node [below, tinyvertex] {$c$} (0);
  \draw [->] (0) -- node [below, tinyvertex] {$d$} (4);
  \draw [->] (4) -- node [right, tinyvertex] {$e$} (3);
  \draw [->] (3) -- node [above, tinyvertex] {$f$} (0);
  \draw[dotted,thick] (-0.3,0.2) arc (135:225:8pt);
  \draw[dotted,thick] (0.3,0.2) arc (45:-45:8pt);
  \draw[dotted,thick] (-1,0.35) arc (-90:-20:8pt);
  \draw[dotted,thick] (-1,-0.35) arc (90:20:8pt);
  \draw[dotted,thick] (1,0.35) arc (-90:-160:8pt);
  \draw[dotted,thick] (1,-0.35) arc (90:160:8pt);
\end{tikzpicture}
\] 
  
\medskip
\noindent
Consider the following indecomposable complex in $\Db(\Lambda)$, where we assume the left-most nonzero term is in cohomological degree zero.
\[
\xymatrix{
  0 \ar[r]
  & P(0) \ar[r]^-{\begin{pmat}c & f\end{pmat}}
  & P(2) \oplus P(3) \ar[r]^-{\begin{pmat}b & 0 \\ 0 & e\end{pmat}}
  & P(1) \oplus P(4) \ar[r]^-{\begin{pmat}af \\ 0\end{pmat}}
  & P(3) \ar[r] & 0. \\
}
\]
This complex can be `unfolded' to give the following diagram,
\[
\xymatrix@R=0.4pc{
  P(4) & P(3) \ar[l]_{\bar{e}} & P(0) \ar[l]_{\bar{f}} \ar[r]^{c} & P(2)
  \ar[r]^{b} & P(1) \ar[r]^{af}& P(3). \\
}
\]
The indecomposable projective modules appearing are uniquely determined by the endpoints of the maps, so all information in this complex is contained in the diagram
\begin{equation} \label{unfolded}
\xymatrix@R=0.4pc{
  \xydot & \xydot \ar[l]_{\bar{e}} & \xydot \ar[l]_{\bar{f}} \ar[r]^{c} & \xydot
  \ar[r]^{b} & \xydot \ar[r]^{af}& \xydot \, .
}
\end{equation}
Here the homotopy string $\sigma = \bar{e}\bar{f}cbaf$, and we refer to \eqref{unfolded} as the `unfolded diagram' of $\sigma$. For more details we refer the reader to \cite[\S 2]{ALP}. 
\end{example}

\begin{definitions}
In the following, walks may now be infinite (on both sides).
\begin{enumerate}
\item A walk $w$ is called a \emph{direct antipath} if it is direct and in its decomposition into homotopy letters, each homotopy letter has length $1$;
it is called an \emph{inverse antipath} if it is inverse and in its decomposition into homotopy letters, each homotopy letter has length $1$.
\item A left infinite walk $w = \cdots w_n \cdots w_2 w_1$ is a \emph{left infinite homotopy string} if there exists $m \geq 1$ such that $v = \cdots w_n \cdots w_{m+1} w_m$ is a direct antipath.
\item A right infinite walk $w = w_{-1} w_{-2} \cdots w_{-n} \cdots$ is a \emph{right infinite homotopy string} if there exists $m \geq 1$ such that $v = w_{-m} w_{-m-1} \cdots w_{-n} \cdots$ is an inverse antipath.
\item A two sided infinite walk $w = \cdots w_2 w_1 w_0 w_{-1} \cdots$ is called a \emph{two-sided infinite homotopy string} if there exist integers $n > m$ such that $\cdots v_{n+1} v_n$ is a direct antipath and $v_m v_{m-1} \cdots$ is an inverse antipath.
\item By a \emph{one-sided infinite homotopy string} we mean either a left infinite homotopy string or a right infinite homotopy string. 
\end{enumerate}
\end{definitions}

By \cite[Thm. 3]{BM} the indecomposable non-perfect complexes in $\Db(\Lambda)$ are parametrised by the one-sided and two-sided infinite homotopy strings; they are again called \emph{string complexes}. In the following, we write
\[
\Q_\sigma =
\left\{
\begin{array}{ll}
\P_\sigma                 & \text{if $\sigma$ is a (possibly infinite) homotopy string;} \\
\B_{\sigma,\lambda} & \text{if $\sigma$ is a homotopy band.}
\end{array}
\right.
\]
From now on, by abuse of terminology, we say homotopy string for a (possibly infinite) homotopy string.

\subsection{The standard basis} \label{sec:basis}

A basis for the morphism space between indecomposable complexes in $\Db(\Lambda)$ was determined in \cite{ALP}. 
Here we briefly recall this basis, which we shall refer to as the \emph{standard basis}. 
As observed in Example~\ref{ex:string-complex}, homotopy strings and bands correspond to an unfolding of the corresponding string and band complexes. Throughout the paper, we shall freely make use of the \emph{unfolded diagram} notation for string and band complexes from \cite{ALP,CPS}.

\begin{theorem}[{\cite[Theorem 3.15]{ALP}}] \label{thm:ALP}
Let $\sigma$ and $\tau$ be homotopy strings or bands. Then there is a canonical basis of $\Hom_{\Db(\Lambda)}(\Q_\sigma,\Q_\tau)$ given by:
\begin{itemize}
\item graph maps $\f \colon \Q_\sigma \to \Q_\tau$;
\item singleton single maps $\f \colon \Q_\sigma \to \Q_\tau$;
\item singleton double maps $\f \colon \Q_\sigma \to \Q_\tau$;
\item quasi-graph maps $\phi \colon \Q_\sigma \rightsquigarrow \Sigma^{-1} \Q_\tau$.
\end{itemize}
\end{theorem}

We note that a quasi-graph map is not a map, but in fact determines classes of homotopy equivalent single and double maps, which is why we denote it by $\rightsquigarrow$ and not $\to$. 
 
Throughout the following description of the maps listed above, $\sigma$ and $\tau$ will be homotopy strings or bands. 
 
\subsubsection{Graph maps} \label{sec:graph-maps}

Suppose $\sigma$ and $\tau$ are, up to inversion, of the form,
\begin{enumerate}
\item $\sigma = \beta \sigma_L \rho \sigma_R \alpha$ and $\tau = \delta \tau_L \rho \tau_R \gamma$; or
\item $\sigma = \rho \sigma_R \alpha$ and $\tau = \rho \tau_R \gamma$,
\end{enumerate}
where $\alpha, \beta, \gamma$ and $\delta$ are homotopy substrings, $\sigma_L, \sigma_R, \tau_L$ and $\tau_R$ are (possibly trivial) homotopy letters, and $\rho$ is a (possibly trivial) maximal common homotopy substring, and in the second case an infinite homotopy substring of $\sigma$ and $\tau$. We assume that $\rho$ occurs in the same cohomological degrees in both homotopy strings. Then the corresponding graph maps can be represented by the following unfolded diagrams: 
\begin{equation}\tag{1} \label{finite-graph-map}
\xymatrix@!R=5px{
\Q_\sigma \colon & \ar@{~}[r]^-{\beta} & \xydot \arr^-{\sigma_L} \ar[d]_-{f_L} \ar@{}[dr]|{(*)} & \xydot \arr^-{\rho_k} \ar@{=}[d] & \xydot \arr^-{\rho_{k-1}} \ar@{=}[d] & \cdots \arr^-{\rho_2}  & \xydot \arr^-{\rho_1} \ar@{=}[d] & \xydot \arr^-{\sigma_R} \ar@{=}[d] \ar@{}[dr]|{(**)} & \xydot \ar@{~}[r]^-{\alpha} \ar[d]^-{f_R} & \\
\Q_\tau \colon      & \ar@{~}[r]_-{\delta} & \xydot \arr_-{\tau_L}                                                & \xydot \arr_-{\rho_k}                  & \xydot \arr_-{\rho_{k-1}}                  & \cdots \arr_-{\rho_2}  & \xydot \arr_-{\rho_1}                  & \xydot \arr_-{\tau_R}                                                    & \xydot \ar@{~}[r]_-{\gamma}                 & 
}
\end{equation}
\begin{equation}\tag{2} \label{infinite-graph-map}
\xymatrix@!R=5px{
\P_\sigma \colon & \ar@{.}[r] & \xydot \arr^-{\rho_3} \ar@{=}[d] & \xydot \arr^-{\rho_2} \ar@{=}[d] & \xydot \arr^-{\rho_1} \ar@{=}[d] & \xydot \arr^-{\sigma_R} \ar@{=}[d] \ar@{}[dr]|{(**)} & \xydot \ar@{~}[r]^-{\alpha} \ar[d]^-{f_R} & \\
\P_\tau \colon      & \ar@{.}[r] & \xydot \arr_-{\rho_3}                 & \xydot \arr_-{\rho_2}                   & \xydot \arr_-{\rho_1}                  & \xydot \arr_-{\tau_R}                                                & \xydot \ar@{~}[r]_-{\gamma}                 & 
}
\end{equation}
where we require the squares marked $(*)$ and $(**)$ to commute; these are explicitly written down in \cite[\S 3.2]{ALP}. The maximality of $\rho$ as a common homotopy substring of $\sigma$ and $\tau$ necessarily means that $\sigma_L \neq \tau_L$ and $\sigma_R \neq \tau_R$. Note that in the case of \ref{sec:graph-maps}(2), $\rho$ is an antipath and  we say that the graph map $\f$ is \emph{incident with $\rho$}.

\subsubsection{Single maps} \label{sec:single}

The unfolded diagram of a single map $\f \colon \Q_\sigma \to \Q_\tau$ is given by 
\begin{equation} \label{single}
\xymatrix@!R=5px{
\Q_\sigma : \ar[d]_-{\f} & \ar@{~}[r]^-{\beta}  & \xydot \arr^-{\sigma_L} & \xydot \arr^-{\sigma_R} \ar[d]^-{f} & \xydot \ar@{~}[r]^-{\alpha} & \\
\Q_\tau      :                 & \ar@{~}[r]_-{\delta} & \xydot \arr_-{\tau_L}     & \xydot \arr_-{\tau_R}                      & \xydot \ar@{~}[r]_-{\gamma} &
}
\end{equation}
where $f$ is a nontrivial path in $(Q,I)$, and satisfying the following conditions:
\begin{itemize}
\item[(L1)] if $\sigma_L \neq \emptyset$ then $\sigma_L$ is either inverse or is direct and $\sigma_L f$ has a subpath in $I$.
\item[(L2)] if $\tau_L \neq \emptyset$ then $\tau_L$ is either direct or is inverse and $f \bar{\tau}_L$ has a subpath in $I$.
\item[(R1)] if $\sigma_R \neq \emptyset$ then $\sigma_R$ is either direct or is inverse and $\bar{\sigma}_R f$ has a subpath in $I$.
\item[(R2)] if $\tau_R \neq \emptyset$ then $\tau_R$ is either inverse or is direct and $f \tau_R$ has a subpath in $I$.
\end{itemize}

A single map $\f \colon \Q_\sigma \to \Q_\tau$ is called a \emph{singleton single map} if its unfolded diagram, up to inversion of one of the homotopy strings/bands, is
\begin{equation} \label{singleton-single}
\xymatrix@!R=5px{
\Q_\sigma : \ar[d]_-{\f} & \ar@{~}[r]^-{\beta}   & \xydot \arr^-{\sigma_L} & \xydot \ar[d]^{f} \ar[r]^-{\sigma_R = f f_R} & \xydot  \ar@{~}[r]^-{\alpha}                                                                      & \\
\Q_\tau :                      & \ar@{~}[r]_-{\delta}  & \xydot \arr_-{\tau_L}     & \xydot                                                        & \ar[l]^-{\tau_R = \overline{f} \overline{f_L}} \xydot \ar@{~}[r]_-{\gamma} &
                }
\end{equation}
where $\sigma_L$ and $\tau_L$ never contain $f$ as a subletter, and whenever $\sigma_L$ is inverse or $\tau_L$ is direct, $f$ does not contain $\sigma_L$ or $\tau_L$ as a subletter, and any of $\sigma_L$, $\sigma_R$, $\tau_L$ and $\tau_R$ are permitted to be the empty homotopy letter $\emptyset$. 

\subsubsection{Double maps}

The unfolded diagram of a double map $\f \colon \Q_\sigma \to \Q_\tau$ is 
\begin{equation} \label{double}
\xymatrix@!R=5px{
\Q_\sigma :  & \ar@{~}[r]^-{\beta}  & \xydot \arr^-{\sigma_L} & \xydot \ar[r]^-{\sigma_C} \ar[d]_-{f_L} & \xydot \arr^-{\sigma_R} \ar[d]^-{f_R} & \xydot \ar@{~}[r]^-{\alpha} &  \\
\Q_\tau      :  & \ar@{~}[r]_-{\delta} & \xydot \arr_-{\tau_L}     & \xydot \ar[r]_-{\tau_C}                         & \xydot \arr_-{\tau_R}                          & \xydot \ar@{~}[r]_-{\gamma} &
}
\end{equation}
where $f_L$ and $f_R$ are nontrivial paths in $(Q,I)$ such that $f_L \tau_C = \sigma_C f_R$ has no subpath in $I$, conditions (L1) and (L2) hold for $f_L$ and (R1) and (R2) hold for $f_R$.

A double map, as above, is called \emph{singleton} if there is a nontrivial path $f'$ in $(Q,I)$ such that $\sigma_C = f_L f'$ and $\tau_C = f' f_R$.

\subsubsection{Quasi-graph maps}\label{subsec:quasi-graph}

If, in the situation of Section~\ref{sec:graph-maps}, the squares marked $(*)$ and $(**)$ of diagrams \eqref{finite-graph-map} and \eqref{infinite-graph-map} do not commute, then such diagrams determine a \emph{quasi-graph map} $\phi \colon \Q_\sigma \rightsquigarrow \Q_\tau$. The non-commuting endpoint conditions are explicitly spelled out in \cite[\S 1.4.4]{CPS}. Note that, while a quasi-graph map $\Q_\sigma \rightsquigarrow \Q_\tau$ does not define a map, a quasi-graph map $\phi \colon \Q_\sigma \rightsquigarrow \Sigma^{-1} \Q_\tau$  determines a family of homotopy equivalent single and/or double maps. Indeed, all single and double maps that are not singleton arise in this way.

The following observation will be useful in the proofs in Section~\ref{sec:surjective}.

\begin{remark} \label{rem:quasi-graph}
Suppose, in the unfolded diagram \eqref{finite-graph-map} above, $\rho_1$ is not the start of both $\sigma$ and $\tau$ and $\rho_k$ is not the end of both $\sigma$ and $\tau$. In this case, the diagram defines a graph map $\f \colon \Q_\sigma \to \Q_\tau$ if and only if the same diagram, when read upside down, i.e. from bottom to top, defines a quasi-graph map $\phi \colon \Q_\tau \wiggle \Q_\sigma$.
Note that we do not read $f_L$ and $f_R$ upside down: they occur as homotopy subletters of $\sigma_L$ or $\tau_L$ (resp. $\sigma_R$ or $\tau_R$). For example, consider the following case:
\[
\xymatrix@!R=5px{
\xydot \ar@{=}[d] \ar[r]^-{\sigma_R}  & \xydot \ar[d]^-{f_R} \ar@{~}[r]^-{\alpha} & \\
\xydot \ar[r]_-{\tau_R = \sigma_R f_R} & \xydot \ar@{~}[r]_-{\gamma} & 
},
\quad \text{versus the same situation `upside down',} \quad
\xymatrix@!R=5px{
\xydot \ar@{=}[d] \ar[r]^-{\tau_R = \sigma_R f_R}  & \xydot \ar@{~}[r]^-{\gamma} & \\
\xydot \ar[r]_-{\sigma_R}                          & \xydot \ar@{~}[r]_-{\alpha} & 
}.
\]
Here the fact that $\tau_R = \sigma_R f_R$ is the obstacle to the commuting of the right endpoint when the `graph map' is read upside down, giving a (non-null-homotopic) quasi-graph map endpoint condition; see \cite[\S 1.4.4]{CPS}.
\end{remark}

\subsection{Morphisms vs. extensions}

For background on derived and homotopy categories we refer to \cite{Happel}. One of the powerful features of the derived category is that it reformulates extensions in the module category in terms of morphisms. In particular, for any algebra $\Lambda$, and any $\Lambda$-modules $M$ and $N$ we have

\smallskip

\begin{tabularx}{\linewidth}{C{1}  E  C{1.75}  E  C{1.75}}
$\Hom_{\sK}(\P_M, \Sigma \P_N)$ & $\simeq$    & $\Ext^1_{\sK}(\P_M, \P_N)$                                                         & $\simeq$    &  $\Ext^1_\Lambda (M,N)$, \\
\scalebox{0.9}{$\P_M \rightlabel{\f} \Sigma \P_N$}  & $\mapsto$ &  \scalebox{0.9}{$\P_N \too C^\bullet_{\f} \too \P_M \rightlabel{\f} \Sigma \P_N$}  & $\mapsto$ & \scalebox{0.9}{$0 \to N \to H^0(C^\bullet_{\f}) \to M \to 0$}
\end{tabularx}

\smallskip

\noindent
where $\sK = \KminusL$, $\P_M$ and $\P_N$ are projective resolutions of $M$ and $N$, respectively, and $C^\bullet_{\f}$ is the (negative shift of the) mapping cone of $\f$. In particular, computation of a basis of the Ext-space $\Ext^1_\Lambda (M,N)$ reduces to the computation of a basis of the Hom-space $\Hom_{\KminusL}(\P_M, \Sigma \P_N)$.

\section{Cohomology of string and band complexes}\label{sec:cohomology}

Throughout $\sigma$ will be a (possibly infinite) homotopy string or band, unless one is specified explicitly. When we wish to specify that $\sigma$ is finite on the right we will write $\sigma = \cdots \sigma_2 \sigma_1$, finite on the left: $\sigma = \sigma_n \sigma_{n-1} \cdots$, and finite on both sides: $\sigma = \sigma_n \cdots \sigma_1$.

Given a homotopy string or band $\sigma$ we will describe how to compute the cohomology of the string or band complex $\Q_\sigma$. The strategy is to divide $\sigma$ up into various homotopy substrings each corresponding to appropriately chosen two-term complexes.
We start with an important technical definition.

\begin{definition}
Let $\sigma$ be a homotopy string or band. A homotopy substring $\tau = \sigma_j \cdots \sigma_i$ with  $i < j$ is a \emph{maximal alternating homotopy substring} if 
\begin{enumerate}[label=(\roman*)]
\item for each $i \leq k < j$, if $\sigma_k$ is direct (resp., inverse) then $\sigma_{k+1}$ is inverse (resp., direct);
\item if $\sigma_i$ is direct (resp., inverse) then $\sigma_{i-1}$ is direct (resp., inverse) and $\sigma_i \sigma_{i-1}  \in I$ (resp., $\overline{\sigma_{i-1} \sigma_i} \in I$) or is  $\emptyset$; and,
\item if $\sigma_j$ is direct (resp., inverse) then $\sigma_{j+1}$ is direct (resp., inverse) and $\sigma_j \sigma_{j+1} \in I$ (resp., $\overline{\sigma_{j+1} \sigma_j } \in I$ or is  $\emptyset$.
\end{enumerate}
If only condition (i) holds, then $\tau$ is called an \emph{alternating homotopy substring}.
\end{definition}

\begin{remark}
Let $\sigma$ be a homotopy string or band and $\tau = \sigma_j \cdots \sigma_i$ with $i < j$ be a maximal alternating homotopy substring of $\sigma$.
\begin{enumerate}
\item The homotopy string $\tau$ has at least two homotopy letters.
\item The string complex $\P_\tau$ is concentrated in precisely two cohomological degrees, namely $\deg P(s(\sigma_i))$ and $\deg P(e(\sigma_i))$, i.e. it is a `two-term complex'.
\item A maximal alternating homotopy substring of a homotopy string or band cannot be infinite: all infinite homotopy strings have antipaths to the left and/or to the right.
\item Since no two consecutive homotopy letters of $\tau$ `pass through a relation', the underlying walk of $\tau$ also determines a string. In the case that $\sigma = \sigma_n \cdots \sigma_1$ is a homotopy band and $\tau = \sigma$, then the underlying walk of $\tau$ also determines a band.
\end{enumerate}
\end{remark}

\begin{lemma}[Maximal alternating homotopy substring rule] \label{lem:max-alternating}
Let $\sigma$ be a homotopy string or band. Suppose $\tau = \sigma_j \cdots \sigma_i$ is a maximal alternating homotopy substring. Decompose the homotopy letters $\sigma_j = b_l \cdots b_1$ and $\sigma_i = a_k \cdots a_1$ into paths or inverse paths in $(Q,I)$ and set 
\[
w \coloneqq 
\left\{
\begin{array}{ll}
  b_{l-1} \cdots b_1 \sigma_{j-1} \cdots \sigma_{i+1} a_k \cdots a_2 & \parbox[c]{0.56\textwidth}{ if $ \tau \neq \sigma$ or $\tau = \sigma$ and $\sigma$ is a homotopy string with $\sigma_1$ inverse and $\sigma_n$ direct;} \\
\sigma                                                                                               & \text{if } \tau = \sigma \text{ and } \sigma \text{ is a homotopy band.}
\end{array}
\right.
\]
Then the string module $M(w)$ (resp., band module $B(w)$) is an indecomposable summand of the cohomology module $H^d(\Q_\sigma)$, where $d = \max\{ \deg P(s(\sigma_i)), \deg P(e(\sigma_i))\}$.
\end{lemma}

\begin{proof}
Suppose $\sigma$ is a homotopy string and $(\Q_\sigma, \partial^\bullet)$ is the corresponding string complex.
We treat the case that the maximal alternating homotopy substring $\tau$ has unfolded diagram of the form below; the other cases, and the case that $\sigma$ is a homotopy band, are similar.
\[
\xymatrix{
\ar@{.}[r] & \xydot  \ar[r]^-{\sigma_{j+1}} & \xydot \ar[r]^-{\sigma_j} & \xydot & \ar[l]_-{\sigma_{j-1}} \xydot \ar@{.}[r] & \xydot & \ar[l]_-{\sigma_{i+1}} \xydot \ar[r]^-{\sigma_i} & \xydot \ar[r]^-{\sigma_{i-1}} & \xydot \ar@{.}[r] & 
}
\]
Note that, in this case $d = \deg P(s(\sigma_i))$ and the homotopy letters $\sigma_i, \ldots, \sigma_j$ are components of the differential $\partial^{d-1}$. In particular, we can wrap $\tau$  back up into a complex:
\[
\xymatrix@!R=3mm{
P(e(\sigma_{j+1})) \ar[r]^-{\sigma_{j+1}} & P(e(\sigma_j)) \ar[r]^-{\sigma_j} \ar@{}[d]|\oplus                                          & P(e(\sigma_{j-1})) \ar@{}[d]|\oplus      & \\
                                                                & P(e(\sigma_{j-2})) \ar[ur]^-{\sigma_{j-1}} \ar[r]^-{\sigma_{j-2}} \ar@{}[d]|\oplus  & P(e(\sigma_{j-3})) \ar@{}[d]|\oplus   & \\
                                                                & P(e(\sigma_{j-4})) \ar[ur]^-{\sigma_{j-3}} \ar[r]^-{\sigma_{j-4}} \ar@{}[d]|\oplus  & \, \vdots \ar@{}[d]|\oplus   & \\
                                                                & \vdots \, \ar[r]^-{\sigma_{i+2}} \ar@{}[d]|\oplus                                                   & P(e(\sigma_{i+1})) \ar@{}[d]|\oplus      & \\
                                                                & P(e(\sigma_i)) \ar[r]^-{\sigma_i} \ar[ur]^-{\sigma_{i+1}}                                & P(s(\sigma_i)) \ar[r]^-{\sigma_{i-2}} & P(s(\sigma_{i-1}))  
}
\]
The other components of the differentials $\partial^{d-2}$, $\partial^{d-1}$ and $\partial^d$ are disconnected from the components of $\partial^{d-1}$ indicated above. The components above therefore contribute a summand, $M$ say, of the cohomology module $H^d(\Q_\sigma)$; the other summands of $H^d(\Q_\sigma)$ are contributed by other parts of  $\sigma$. We claim that $M \cong M(w)$, where $w$ is the string defined in the statement.

The projective modules $P(e(\sigma_{i+1})),P(e(\sigma_{i+3})), \ldots, P(e(\sigma_{j-1})) \subset \ker(\partial^d)$.
Consider the following components of the differential $\partial^{d-1}$,
\[
\xymatrix@!R=2mm{
                                                                                             & P(e(\sigma_{m+1})) \ar@{}[d]|\oplus \\
P(e(\sigma_m)) \ar[ur]^-{\sigma_{m+1}} \ar[r]_-{\sigma_m}  & P(s(\sigma_m)) 
}
\]
which map diagonally into submodules of $P(e(\sigma_{m+1}))$ and $P(s(\sigma_m))$ with simple top $S(e(\sigma_m))$. Thus, in the quotient $\ker(\partial^d)/\im(\partial^{d-1})$ the action of $\sigma_m$ on the basis vector at $s(\sigma_{m+1})$ is the same as the action of $\sigma_{m-1}$ on the basis vector at $s(\sigma_{m-1})$, as indicated in 
Figure~\ref{fig:coh}.
At the left-hand endpoint of $\tau$, i.e. at the homotopy letter $\sigma_j$, in the quotient $\ker(\partial^d)/\im(\partial^{d-1})$ the arrow $b_l$ acts on the basis vector supported at $s(b_l)$ by sending it to $0$ (because the basis vector at $e(b_l)$ is an element of $\im(\partial^{d-1})$. Hence, the arrow $b_l$ is removed from the string describing this indecomposable summand of $H^d(\P_\sigma)$.
Similarly, at the right-hand endpoint of $\tau$, i.e. at $\sigma_i$, the basis vector supported at $s(a_1) = x_{i-1}$ is not an element of $\ker(\partial^d)$. Hence, the arrow $a_1$ is removed from the string describing the indecomposable summand of $H^d(\P_\sigma)$.
\begin{figure}
\scalebox{0.8}{\begin{tikzpicture}
\node (0) at (0,3){$\begin{tikzpicture} 
\node (a) at (0,0){$x_{j}$};
\node (b) at (-1,-1){};
\node (c) at (1,-1){};
\draw[-stealth,
decoration={snake, 
    amplitude = .4mm,
    segment length = 2mm,
    post length=0.9mm},decorate] (a) --  node [anchor=south east,scale=.7] {$\alpha_{j}$}  (b);
\draw[-stealth,
decoration={snake, 
    amplitude = .4mm,
    segment length = 2mm,
    post length=0.9mm},decorate] (a) --  node [anchor=south west,scale=.7] {$\beta_{j}$}  (c);
\end{tikzpicture}$};
\node (1) at (0,0){$\begin{tikzpicture} 
\node (a) at (0,0){$x_{j-2}$};
\node (b) at (-1,-1){};
\node (c) at (1,-1){};
\draw[-stealth,
decoration={snake, 
    amplitude = .4mm,
    segment length = 2mm,
    post length=0.9mm},decorate] (a) --  node [anchor=south east,scale=.7] {$\alpha_{j-2}$}  (b);
\draw[-stealth,
decoration={snake, 
    amplitude = .4mm,
    segment length = 2mm,
    post length=0.9mm},decorate] (a) --  node [anchor=south west,scale=.7] {$\beta_{j-2}$}  (c);
\end{tikzpicture}$};
\node (2) at (6,4){$\begin{tikzpicture} 
\node (a) at (0,0){$x_{j-1}$};
\node (b1) at (-.5,-.5){};
\node (b2) at (-1,-1){};
\node (b3) at (-1.5,-1.5){};
\node (b4) at (-2,-2){$x_{j}$};
\node (b) at (-3,-3){};
\node[circle,fill=black!20,draw, scale=.8](c) at (1,-1){$x_{j-2}$};
\node (d) at (2,-2){};
\draw [line width=35pt,opacity=0.1,black,line cap=round,rounded corners] (b.center) -- (a.center) -- (d.center);
\draw [line width=25pt,opacity=0.15,black,line cap=round,rounded corners] (c.center) -- (d.center);
\draw [line width=25pt,opacity=0.15,black,line cap=round,rounded corners] (b4.center) -- (b.center);
\draw[->] (a) -- node [anchor=south east,scale=.7] {$b_1$} (b1.center);
\draw[dotted] (b1) -- (b2.center);
\draw[->] (b2) -- node [anchor=south east,scale=.6] {$b_{l-1}$} (b3);
\draw[->] (b3.center) -- node [anchor=south east,scale=.7] {$b_{l}$} (b4);
\draw[-stealth,
decoration={snake, 
    amplitude = .4mm,
    segment length = 2mm,
    post length=0.9mm},decorate] (b4) --  node [anchor=south east,scale=.7] {$\alpha_{j}$}  (b);
 \draw[-stealth,
decoration={snake, 
    amplitude = .4mm,
    segment length = 2mm,
    post length=0.9mm},decorate] (a) --  node [anchor=south west,scale=.7] {$\sigma_{j-1}$}  (c);
\draw[-stealth,
decoration={snake, 
    amplitude = .4mm,
    segment length = 2mm,
    post length=0.9mm},decorate]  (c) --  node [anchor=south west,scale=.7] {$\beta_{j-2}$}  (d);
\end{tikzpicture}$};
\node (3) at (6.5,.5){$\begin{tikzpicture} 
\node (a) at (0,0){$x_{j-3}$};
\node[circle,fill=black!20,draw, scale=.8] (b) at (-1,-1){$x_{j-2}$};
\node (c) at (-2,-2){};
\node[circle,fill=black!20,draw, scale=.8] (d) at (1,-1){$x_{j-4}$};
\node (e) at (2,-2){};
\draw [line width=35pt,opacity=0.1,black,line cap=round,rounded corners] (c.center) -- (a.center) -- (e.center);
\draw [line width=25pt,opacity=0.15,black,line cap=round,rounded corners] (b.center) -- (c.center) (d.center) -- (e.center) ;
\draw[-stealth,
decoration={snake, 
    amplitude = .4mm,
    segment length = 2mm,
    post length=0.9mm},decorate] (a) --  node [anchor=south east,scale=.7] {$\sigma_{j-2}$}  (b);
\draw[-stealth,
decoration={snake, 
    amplitude = .4mm,
    segment length = 2mm,
    post length=0.9mm},decorate]  (b) --  node [anchor=south east,scale=.7] {$\alpha_{j-2}$}  (c);
 \draw[-stealth,
decoration={snake, 
    amplitude = .4mm,
    segment length = 2mm,
    post length=0.9mm},decorate] (a) --  node [anchor=south west,scale=.7] {$\sigma_{j-3}$}  (d);
\draw[-stealth,
decoration={snake, 
    amplitude = .4mm,
    segment length = 2mm,
    post length=0.9mm},decorate] (d) --  node [anchor=south west,scale=.7] {$\beta_{j-4}$}  (e);

\end{tikzpicture}$};
\draw[-stealth] (0)--node [anchor=south,scale=.7] {$\sigma_{j}=b_l\dots b_1$} (3,3);
\draw[-stealth] (1)--node [anchor=south,scale=.7] {$\sigma_{j-1}$} (3,2);
\draw[-stealth] (1)--node [anchor=north,scale=.7] {$\sigma_{j-2}$} (3.2,0);
\node at (6.5,3){$\oplus$};
\path[-, ultra thick, black!30] (7.08,4.4) edge [out=210, in=120] (5.2,.8);
\path[-, ultra thick, black!30] (7.12,.3) edge [out=210, in=80] (6.2,-.8);
\path[dotted, ultra thick, black!30] (6.2,-.8) edge [out=90, in=80] (6.2,-1.2);
\end{tikzpicture}}
\quad
\scalebox{0.8}{\begin{tikzpicture}
\node (1) at (0,0){$\begin{tikzpicture} 
\node (a) at (0,0){$x_m$};
\node (b) at (-1,-1){};
\node (c) at (1,-1){};
\draw[-stealth,
decoration={snake, 
    amplitude = .4mm,
    segment length = 2mm,
    post length=0.9mm},decorate] (a) --  node [anchor=south east,scale=.7] {$\alpha_m$}  (b);
\draw[-stealth,
decoration={snake, 
    amplitude = .4mm,
    segment length = 2mm,
    post length=0.9mm},decorate] (a) --  node [anchor=south west,scale=.7] {$\beta_m$}  (c);
\end{tikzpicture}$};
\node (2) at (5,1){$\begin{tikzpicture} 
\node (a) at (0,0){$x_{m+1}$};
\node (b) at (-1,-1){};
\node[circle,fill=black!20,draw, scale=.8](c) at (1,-1){$x_m$};
\node (d) at (2,-2){};
\draw [line width=35pt,opacity=0.1,black,line cap=round,rounded corners] (b.center) -- (a.center) -- (d.center);
\draw [line width=25pt,opacity=0.15,black,line cap=round,rounded corners] (c.center) -- (d.center);
\draw[-stealth,
decoration={snake, 
    amplitude = .4mm,
    segment length = 2mm,
    post length=0.9mm},decorate] (a) --  node [anchor=south east,scale=.7] {$\alpha_{m+1}$}  (b);
 \draw[-stealth,
decoration={snake, 
    amplitude = .4mm,
    segment length = 2mm,
    post length=0.9mm},decorate] (a) --  node [anchor=south west,scale=.7] {$\sigma_{m+1}$}  (c);
\draw[-stealth,
decoration={snake, 
    amplitude = .4mm,
    segment length = 2mm,
    post length=0.9mm},decorate]  (c) --  node [anchor=south west,scale=.7] {$\beta_{m}$}  (d);
\end{tikzpicture}$};
\node (3) at (4,-1.5){$\begin{tikzpicture} 
\node (a) at (0,0){$x_{m-1}$};
\node[circle,fill=black!20,draw, scale=.8] (b) at (-1,-1){$x_m$};
\node (c) at (-2,-2){};
\node (d) at (1,-1){};
\draw [line width=35pt,opacity=0.1,black,line cap=round,rounded corners] (c.center) -- (a.center) -- (d.center);
\draw [line width=25pt,opacity=0.15,black,line cap=round,rounded corners] (b.center) -- (c.center);
\draw[-stealth,
decoration={snake, 
    amplitude = .4mm,
    segment length = 2mm,
    post length=0.9mm},decorate] (a) --  node [anchor=south east,scale=.7] {$\sigma_{m}$}  (b);
\draw[-stealth,
decoration={snake, 
    amplitude = .4mm,
    segment length = 2mm,
    post length=0.9mm},decorate]  (b) --  node [anchor=south east,scale=.7] {$\alpha_{m}$}  (c);
 \draw[-stealth,
decoration={snake, 
    amplitude = .4mm,
    segment length = 2mm,
    post length=0.9mm},decorate] (a) --  node [anchor=south west,scale=.7] {$\beta_{m-1}$}  (d);
\end{tikzpicture}$};
\draw[-stealth] (1)--node [anchor=south,scale=.7] {$\sigma_{m+1}$} (2.7,.8);
\draw[-stealth] (1)--node [anchor=north,scale=.7] {$\sigma_{m}$} (2.7,-.8);
\node at (4.5,0.4){$\oplus$};
\path[-, ultra thick, black!30] (5.2,1) edge [out=180, in=90] (3.2,-1.4);
\end{tikzpicture}} \\
\scalebox{0.8}{\begin{tikzpicture}
\node (1) at (0,0){$\begin{tikzpicture} 
\node (a) at (0,0){$x_{i}$};
\node (b) at (-1,-1){};
\node (c) at (1,-1){};
\draw[-stealth,
decoration={snake, 
    amplitude = .4mm,
    segment length = 2mm,
    post length=0.9mm},decorate] (a) --  node [anchor=south east,scale=.7] {$\alpha_{i}$}  (b);
\draw[-stealth,
decoration={snake, 
    amplitude = .4mm,
    segment length = 2mm,
    post length=0.9mm},decorate] (a) --  node [anchor=south west,scale=.7] {$\beta_{i}$}  (c);
\end{tikzpicture}$};
\node (2) at (6,3){$\begin{tikzpicture} 
\node (a) at (0,0){$x_{i+1}$};
\node (b) at (-1,-1){$x_{i+2}$};
\node (c) at (-2,-2){};
\node[circle,fill=black!20,draw, scale=.8] (d) at (1,-1){$x_{i}$};
\node (e) at (2,-2){};
\draw [line width=35pt,opacity=0.1,black,line cap=round,rounded corners] (c.center) -- (a.center) -- (e.center);
\draw [line width=25pt,opacity=0.15,black,line cap=round,rounded corners] (d.center) -- (e.center) ;
\draw[-stealth,
decoration={snake, 
    amplitude = .4mm,
    segment length = 2mm,
    post length=0.9mm},decorate] (a) --  node [anchor=south east,scale=.7] {$\sigma_{i+2}$}  (b);
\draw[-stealth,
decoration={snake, 
    amplitude = .4mm,
    segment length = 2mm,
    post length=0.9mm},decorate]  (b) --  node [anchor=south east,scale=.7] {$\beta_{i+2}$}  (c);
 \draw[-stealth,
decoration={snake, 
    amplitude = .4mm,
    segment length = 2mm,
    post length=0.9mm},decorate] (a) --  node [anchor=south west,scale=.7] {$\sigma_{i+1}$}  (d);
\draw[-stealth,
decoration={snake, 
    amplitude = .4mm,
    segment length = 2mm,
    post length=0.9mm},decorate] (d) --  node [anchor=south west,scale=.7] {$\beta_{i}$}  (e);

\end{tikzpicture}$};
\node (3) at (4.4,-1)
{$\begin{tikzpicture} 
\node (a) at (0,0){$x_{i-1}$};
\node (b1) at (-.5,-.5){};
\node (b2) at (-1,-1){};
\node (b3) at (-1.5,-1.5){};
\node[circle,fill=black!20,draw, scale=.8] (b4) at (-2,-2){$x_{i}$};
\node (b) at (-3,-3){};
\node (c) at (1,-1){};

\draw [line width=35pt,opacity=0.1,black,line cap=round,rounded corners] (b.center) -- (b2.center);
\draw [line width=25pt,opacity=0.15,black,line cap=round,rounded corners] (b4.center) -- (b.center);
\draw[->] (a) -- node [anchor=south east,scale=.7] {$a_1$} (b1.center);
\draw[->] (b1) -- node [anchor=south east,scale=.6] {$a_2$} (b2.center);
\draw[dotted] (b2) -- (b3);
\draw[->] (b3.center) -- node [anchor=south east,scale=.7] {$a_k$} (b4);
\draw[-stealth,
decoration={snake, 
    amplitude = .4mm,
    segment length = 2mm,
    post length=0.9mm},decorate] (b4) --  node [anchor=south east,scale=.7] {$\alpha_{i}$}  (b);
 \draw[-stealth,
decoration={snake, 
    amplitude = .4mm,
    segment length = 2mm,
    post length=0.9mm},decorate] (a) --  node [anchor=south west,scale=.7] {$\beta_{i-1}$}  (c);
\end{tikzpicture}$};
\node (4) at (10.5,-.5){$\begin{tikzpicture} 
\node (a) at (0,0){$x_{i-2}$};
\node (b) at (-1,-1){};
\node (d) at (1,-1){$x_{i-1}$};
\node (e) at (2,-2){};
\draw[-stealth,
decoration={snake, 
    amplitude = .4mm,
    segment length = 2mm,
    post length=0.9mm},decorate] (a) --  node [anchor=south east,scale=.7] {$\alpha_{i-2}$}  (b);
\draw[-stealth,
decoration={snake, 
    amplitude = .4mm,
    segment length = 2mm,
    post length=0.9mm},decorate] (a) --  node [anchor=south west,scale=.7] {$\sigma_{i-1}$}  (d);
\draw[-stealth,
decoration={snake, 
    amplitude = .4mm,
    segment length = 2mm,
    post length=0.9mm},decorate] (d) --  node [anchor=south west,scale=.7] {$\beta_{i-1}$}  (e);

\end{tikzpicture}$};
\draw[-stealth] (1)--node [anchor=south,scale=.7] {$\sigma_{i+1}$} (3,1.5);
\draw[-stealth] (1)--node [anchor=north,scale=.7] {$\sigma_{i}=a_k\dots a_1$} (3,0);
\draw[-stealth] (7.3,0)--node [anchor=south,scale=.7] {$\sigma_{i-1}$} (8.3,0);
\path[-, ultra thick, black!30] (6.8,2.8) edge [out=250, in=140] (3.4,-1.1);
\end{tikzpicture}}
\caption{Schematic showing the computation of the cohomology: $\ker \partial$ is shown in light grey, $\im \partial$ in mid-grey, and basis vectors identified by quotienting by the image of the diagonal map shown in dark grey and joined by a dark grey line. Here $x_m = e(\sigma_m)$, i.e. is the end of the homotopy letter $\sigma_m$.
Top left: illustration of the situation at the left end of the maximal alternating homotopy string $\tau$. Top right: a generic situation midway in $\tau$. Bottom: the situation at the right hand end of $\tau$.} \label{fig:coh}
\end{figure}
It follows that the summand $M$ of $\ker(\partial^d)/\im(\partial^{d-1})$ has the following form. 
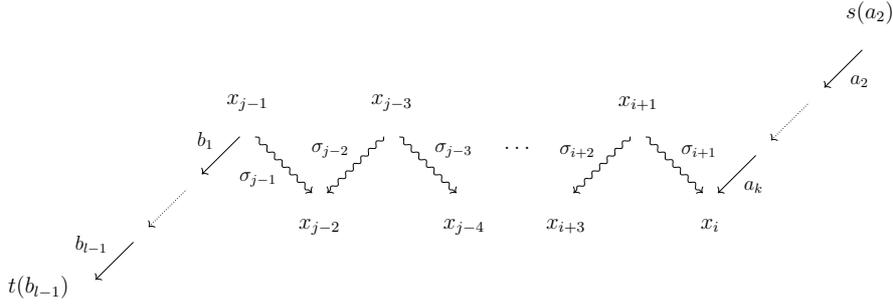
\[ \resizebox{0.75\textwidth}{!}{
\begin{tikzpicture}
\coordinate[circle,label=left:{$t(b_{l-1})$}] (1);
\coordinate[circle,above right=1cm of 1,label=center:{$\quad$}] (2);
\coordinate[circle,above right=1cm of 2,label=center:{$\quad$}] (3);
\coordinate[circle,above right=1cm of 3,label=above:{$x_{j-1}$}] (4);
\coordinate[circle,below right=1.5cm of 4,label=below:{$x_{j-2}$}] (5);
\coordinate[circle,above right=1.5cm of 5,label=above:{$x_{j-3}$}] (6);
\coordinate[circle,below right=1.5cm of 6,label=below:{$x_{j-4}$}] (7);
\coordinate[circle,above right=1cm of 7,label=center:{$\ldots$}] (8);
\coordinate[circle,right=1.5cm of 7,label=below:{$x_{i+3}$}] (9);
\coordinate[circle,above right=1.5cm of 9,label=above:{$x_{i+1}$}] (10);
\coordinate[circle,below right=1.5cm of 10,label=below:{$x_i$}] (11);
\coordinate[circle,above right=1cm of 11,label=center:{$\quad$}] (12);
\coordinate[circle,above right=1cm of 12,label=center:{$\quad$}] (13);
\coordinate[circle,above right=1cm of 13,label=above:{$s(a_2)$}] (14);

\draw[<-]  (1)--node [anchor=south east,scale=.9]{$b_{l-1}$}(2); 
\draw[densely dotted,<-] (2)--node [anchor=south east,scale=.9]{}(3); 
\draw[<-] (3)--node [anchor=south east,scale=.9]{$b_1$}(4); 

\draw[<-]  (11)--node [anchor=north west,scale=.9]{$a_{k}$}(12); 
\draw[densely dotted,<-] (12)--node [anchor=north west,scale=.9]{}(13); 
\draw[<-] (13)--node [anchor=north west,scale=.9]{$a_2$}(14); 

\draw[->,decorate,decoration={snake,amplitude=.4mm,segment length=2mm}] 
(4)-- node [anchor=north east,scale=.9]{$\sigma_{j-1}$} (5);
\draw[->,decorate,decoration={snake,amplitude=.4mm,segment length=2mm}] 
(6)-- node [anchor=south east,scale=.9]{$\sigma_{j-2}$} (5);
\draw[->,decorate,decoration={snake,amplitude=.4mm,segment length=2mm}] 
(6)-- node [anchor=south west,scale=.9]{$\sigma_{j-3}$} (7);
\draw[->,decorate,decoration={snake,amplitude=.4mm,segment length=2mm}] 
(10)-- node [anchor=south east,scale=.9]{$\sigma_{i+2}$} (9);
\draw[->,decorate,decoration={snake,amplitude=.4mm,segment length=2mm}] 
(10)-- node [anchor=south west,scale=.9]{$\sigma_{i+1}$} (11);
\end{tikzpicture}}
\]
that is, corresponds to the string
$w = b_{l-1} \cdots b_1 \sigma_{j-1} \cdots \sigma_{i+1} a_k \cdots a_2$.
\end{proof}

The following lemmas are computations analogous to that in Lemma~\ref{lem:max-alternating} above. Thus we provide only their statements and leave the proofs to the reader.

\begin{lemma}[Cokernel rule] \label{lem:cokernel}
Let $\sigma = \cdots \sigma_2 \sigma_1$ be a homotopy string in which $\sigma_1 = a_k \cdots a_1$ and $\sigma_2$ are direct homotopy letters.
If there exists $c$ with $c \in \overline{Q}_1$ such that $\sigma_1 c$ is defined as a string, then take $u = c_m \cdots c_1$ to be the maximal inverse string ending with $c_m = c$. Set
\[
w \coloneqq 
\left \{
\begin{array}{ll}
a_{k-1} \cdots a_1 u & \text{if there is such a } c; \\
a_{k-1} \cdots a_1 & \text{otherwise}.
\end{array}
\right.
\]
Then the string module $M(w)$ is an indecomposable summand of the cohomology module $H^d(\P_\sigma)$, where $d = \deg P(s(\sigma_1))$.
\end{lemma}

For a homotopy string $\sigma = \cdots \sigma_2 \sigma_1$ with $\sigma_1$ direct, it is possible that $\sigma_2$ is not direct, in which case $\sigma_1$ is part of a maximal alternating homotopy substring $\tau = \sigma_j \cdots \sigma_1$. 
In this case, we combine the cokernel rule with the maximal alternating homotopy substring rule.

\begin{lemma}[Combined rule] \label{lem:combined}
Let $\sigma = \cdots \sigma_2 \sigma_1$ be a homotopy string in which $\sigma_1 = a_k \cdots a_1$ is a direct homotopy letter and $\tau = \sigma_j \cdots \sigma_1$ is a maximal alternating homotopy substring. Decompose the homotopy letter $\sigma_j = b_l \cdots b_1$ into a path or inverse path in $(Q,I)$ and set
\[
w \coloneqq 
\left \{
\begin{array}{ll}
b_{l-1} \cdots b_1 \sigma_{j-1} \cdots \sigma_1 u & \text{if there exist $c$ and $u$ as in Lemma~\ref{lem:cokernel};} \\
b_{l-1} \cdots b_1 \sigma_{j-1} \cdots \sigma_2 a_k \cdots a_1 & \text{otherwise.}
\end{array}
\right.
\]
Then the string module $M(w)$ is an indecomposable summand of the cohomology module $H^d(\P_\sigma)$, where $d = \deg P(s(\sigma_1))$.
\end{lemma}

There are obvious dual statements if $\sigma = \sigma_n \sigma_{n-1} \cdots$ with $\sigma_n$ inverse. If $\tau = \sigma$ with $\sigma_n$ inverse and $\sigma_1$ direct then we must combine Lemma~\ref{lem:combined} with its dual.

\begin{lemma}[Kernel rule] \label{lem:kernel}
Let $\sigma = \sigma_n \sigma_{n-1} \cdots$ be a homotopy string in which $\sigma_n = b_l \cdots b_1$ is a direct homotopy letter.
If there exists $c \in Q_1$ and $c b_l = 0$  then take $c_m \cdots c_1$ to be the maximal direct string starting with $c_1 = c$. Set
\[
v \coloneqq 
\left \{
\begin{array}{ll}
c_m \cdots c_2 & \text{if there exists such a $c$;} \\
\emptyset & \text{otherwise.}
\end{array}
\right .
\]
Then the string module $M(v)$ is an indecomposable summand of the cohomology module $H^d(\P_\sigma)$, where $d=\deg P(e(\sigma_n))$. If $m =1$ then $v = 1_{e(c)}$ is the trivial string corresponding to the simple module $S(e(c))$.
\end{lemma}

Note that if $\sigma = \sigma_n \sigma_{n-1} \cdots$ is a homotopy string and $\tau = \sigma_n \cdots \sigma_i$ is a maximal alternating homotopy substring with $\sigma_n$ direct, then Lemmas~\ref{lem:max-alternating} and \ref{lem:kernel} do not need to be combined. In particular, the string module $M(v)$ is an indecomposable summand of $H^d(\P_\sigma)$ and the string module $M(w)$ is an indecomposable summand of $H^{d+1}(\P_\sigma)$, where $v$ is defined as in Lemma~\ref{lem:kernel}, $w$ is defined as in Lemma~\ref{lem:max-alternating}, and $d = \deg P(e(\sigma_n))$.

\begin{lemma}[Nontrivial homotopy letter rule] \label{lem:other-case}
Let $\sigma$ be a homotopy string or band in which $\sigma_i$ is a direct homotopy letter and $\sigma_{i+1} \sigma_i \sigma_{i-1}$ is a non-alternating homotopy substring with $\sigma_{i+1}$ possibly empty. Let $d = \deg P(s(\sigma_i))$.
\begin{enumerate}
\item If $\sigma_i = a_k \cdots a_1$ with $a_j \in Q_1$ and $k > 1$ then set $w = a_{k-1} \cdots a_2$. The string module $M(w)$ is an indecomposable summand of the cohomology module $H^d(\Q_\sigma)$. If $k= 2$ then $w = 1_{e(a_1)} = 1_{s(a_2)}$ and $M(w) = S(e(a_1)) = S(s(a_2))$.
\item If $\sigma_i = a$ for some $a \in Q_1$ then the map $\sigma_i \colon P(e(\sigma_i)) \to P(s(\sigma_i))$ contributes the zero submodule to the cohomology module of $H^d(\Q_\sigma)$.
\end{enumerate}
\end{lemma}

Lemmas~\ref{lem:cokernel}, \ref{lem:combined}, \ref{lem:kernel} and \ref{lem:other-case} admit obvious dual statements. When referring to these lemmas we shall freely include those dual statements. We summarise this section with the following theorem and illustrate with an example.

\begin{theorem} \label{thm:cohomology}
Let $\sigma$ be a homotopy string or band. Lemmas~\ref{lem:max-alternating}, \ref{lem:cokernel}, \ref{lem:combined}, \ref{lem:kernel} and \ref{lem:other-case} and their duals provide a complete description of the cohomology complex $H^\bullet(\Q_\sigma)$.
\end{theorem}

\begin{remark}
Note that in computing the cohomology Lemmas~\ref{lem:max-alternating}, \ref{lem:cokernel}, \ref{lem:combined}, \ref{lem:kernel}, \ref{lem:other-case} and their duals can be applied independently and therefore in any order. The only exception is that the combined rule Lemma~\ref{lem:combined} should always be applied instead of Lemma~\ref{lem:max-alternating} whenever the homotopy string has the appropriate form.
\end{remark}

\begin{example}
We consider the gentle algebra with the following quiver where the (length 2) relations are indicated by dotted lines. 

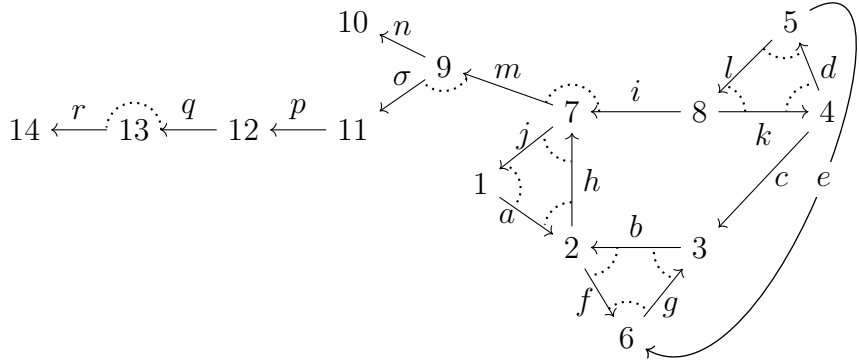
\[
\begin{tikzpicture}[scale=1]
\node (A14) at (.8,1.8){$14$};
\node (A13) at (2,1.8){$13$};
\node (A12) at (3.2,1.8){$12$};
\node (A11) at (4.4,1.8){$11$};

\node (A4) at (9.6,2) {$4$};
\node (A8) at (8.2,2) {$8$};
\node (A7) at (6.8,2) {$7$};
\node (A9) at (5.4,2.5) {$9$};
\node (A10) at (4.4,3) {$10$};

\node (A3) at (8.2,.5) {$3$};
\node (A2) at (6.8,.5) {$2$};

\node (A1) at (5.8,1.2) {$1$};

\node (A5) at (9.2,3) {$5$};

\node (A6) at (7.4,-.5) {$6$};

\draw[->] (A6) -- node[pos=.2, rotate=0,right] {$g$} (A3);
\draw[->] (A3) -- node[pos=.5, rotate=0,above] {$b$} (A2);
\draw[->] (A2) -- node[pos=.7, rotate=0,left] {$f$} (A6);
\draw[->] (A4) -- node[pos=.5, rotate=0,right] {$c$} (A3);
\draw[->] (A4) -- node[pos=.5, rotate=0,right] {$d$} (A5);
\draw[->] (A5) -- node[pos=.5, rotate=0,left] {$l$} (A8);
\draw[->] (A8) -- node[pos=.5, rotate=0,above] {$i$} (A7);
\draw[->] (A7) -- node[pos=.5, rotate=0,above] {$m$} (A9);
\draw[->] (A9) -- node[pos=.5, rotate=0,above] {$n$} (A10);
\draw[->] (A7) -- node[pos=.2, rotate=0,left] {$j$} (A1);
\draw[->] (A2) -- node[pos=.5, rotate=0,right] {$h$} (A7);
\draw[->] (A1) -- node[pos=.5, rotate=0,left] {$a$} (A2);
\draw[->] (A8) -- node[pos=.5, rotate=0,below] {$k$} (A4);

\draw[->] (A9) -- node[pos=.5, rotate=0,below] {$o$} (A11);
\draw[->] (A11) -- node[pos=.5, rotate=0,above] {$p$} (A12);
\draw[->] (A12) -- node[pos=.5, rotate=0,above] {$q$} (A13);
\draw[->] (A13) -- node[pos=.5, rotate=0,above] {$r$} (A14);

\draw[thick,dotted] (2.3,1.8) arc (0:175:.3cm)
(7.1,2) arc (0:160:.3cm) 
(5.2,2.3) arc (230:330:.3cm)
(8.7,2) arc (0:70:.3cm)
(9.4,2.3) arc (100:180:.3cm)
(8.9,2.7) arc (220:310:.3cm)
(6.2,1) arc (330:420:.3cm)
(6.8,1) arc (90:170:.3cm)
(6.5,1.7) arc (190:270:.3cm)
(7.05,.2) arc (280:360:.3cm)
(7.6,-.17) arc (50:140:.3cm)
(7.7,.45) arc (180:250:.3cm)
;

\tikzset{EdgeStyle/.style={->,relative=false,in=-30,out=30,line width=.5pt}}
  \Edge[label=$e$](A5)(A6) 
  \end{tikzpicture}
\]

\noindent
Consider the following homotopy strings where the top line indicates the cohomological degree of the corresponding projective indecomposable:
\[
\xymatrix{ & 0 & 1& 0 & -1 & -2 & -3 & -2 & -1 \\
\sigma \colon &
7 \ar[r]^-{i} &
8                &
\ar[l]_-{\bk \bc \bb} 2 &
\ar[l]_-{\bf} 6 &
\ar[l]_-{\bg} 3 &
\ar[l]_-{\bb \bh} 7 \ar[r]^-{il} &
5 \ar[r]^-{d} &
4
}
\]
and
\[
\xymatrix{ & -1 & 0 & 1 & 2 \\
\tau \colon &
14 \ar[r]^-{r} &
13        \ar[r]^-{qpo}        &
9 \ar[r]^-{m} &
7.
}
\]
Examining the homotopy string $\sigma$, we see that there are four indecomposable summands of $H^\bullet(\P_\sigma)$. We list them below in order of ascending cohomological degree.
\begin{itemize}
\item We have $H^{-2}(\P_\sigma) = M(w_1)$, where $w_1 = \bh i$ coming from the maximal alternating homotopy substring rule (Lemma~\ref{lem:max-alternating}) applied to $3 \stackrel{\bb \bh}{\longleftarrow} 7 \stackrel{il}{\too} 5$.
\item We have $H^{-1}(\P_\sigma) = M(w_2)$, where $w_2 = \bc \bb \bh \bm \bn$ coming from the cokernel rule (Lemma~\ref{lem:cokernel}) applied to $5 \stackrel{d}{\too} 4$.
\item We have $H^0(\P_\sigma) = M(w_3)$, where $w_3 = n$ coming from the kernel rule (Lemma~\ref{lem:kernel}) applied to $7 \stackrel{i}{\too} 8$.
\item We have $H^1(\P_\sigma) = M(w_4)$, where $w_4 = \bk \bc$ coming from the maximal alternating homotopy substring rule (Lemma~\ref{lem:max-alternating}) applied to $7 \stackrel{i}{\too} 8 \stackrel{\bk \bc \bb}{\longleftarrow} 2$.
\item By Lemma~\ref{lem:other-case}, all remaining parts of the homotopy string $\sigma$ contribute zero to the cohomology $H^\bullet(\P_\sigma)$.
\end{itemize}

Examining the homotopy string $\tau$, in a similar fashion we obtain the following for  $H^\bullet(\P_\tau)$.
\begin{itemize}
\item We have $H^1(P_\tau ) = M(w_1)$ where $w_1 = p$ coming from the nontrivial homotopy letter rule (Lemma~\ref{lem:other-case}) applied to $13 \stackrel{qpo}{\too} 9 \stackrel{m}{\too} 7$.
\item We have $H^2(P_\tau) = M(w_2)$ where $w_2 = j$ coming from  the cokernel rule (Lemma~\ref{lem:cokernel}) applied to $9 \stackrel{m}{\too} 7$. 
\item There is no non-zero contribution to the cohomology coming from the nontrivial homotopy letter rule (Lemma~\ref{lem:other-case}) applied to $14 \stackrel{r}{\too} 13 \stackrel{qpo}{\too} 9 $  or from the kernel rule (Lemma~\ref{lem:kernel}) applied to $14 \stackrel{r}{\too} 13$. 
\end{itemize}
\end{example}

We end this section by giving the homotopy string or band of the minimal projective resolution of a string or quasi-simple band module, which will be heavily used in the next sections. We note that  gentle algebras are string algebras and that  there is a large body of work on string algebras. In particular, projective resolutions and syzygies, have been  considered before, see for example  \cite{HZS, Kalck}. 
In \cite{WW}, minimal projective presentations of string and band modules were given in terms of string combinatorics, which in the case of gentle algebras can be formulated in terms of homotopy string combinatorics. These projective presentations correspond to maximal alternating homotopy substrings sitting between degrees $-1$ and $0$.
Before stating the result, we set up some notation.

\begin{definition}
Let $a$ and $b$ be such that $\ba, b \in Q_1$. Define
\begin{itemize}
\item $\inverse(a) \coloneqq \sigma_{-1} \sigma_{-2} \cdots $ to be maximal inverse antipath ending with $\sigma_{-1} = a$;
\item $\direct(b) \coloneqq \cdots \sigma_2 \sigma_1$ to be the maximal direct antipath starting with $\sigma_1 = b$. 
\end{itemize}
\end{definition}

\begin{corollary} \label{cor:projective-resolution}
Let $w = w_n \cdots w_1$ be a string. Define a homotopy string $\pi(w)$ as follows:
\begin{enumerate}
\item  $\pi(w) = \direct(b) \, w' \, \inverse(a)$ if there are $a$ and $b$ such that $\ba, b \in Q_1$ and $b w a$ is defined as a string and where $w' = w$.
\item $\pi(w) = w' \, \inverse(a)$ if there is an $a$ with $\ba \in Q_1$ such that $w a$ is defined as a string but no $b \in Q_1$ with $b w$ defined as a string, where $w' = w_j \cdots w_1$ after removing a maximal direct substring $w_n \cdots w_{j+1}$ of $w$.
\item $\pi(w) = \direct(b) \, w'$ if there is $b \in Q_1$ with $b w$ defined as a string but no $a$ with $\ba \in Q_1$ such that $w a$ is defined as a string, where $w' = w_n \cdots w_i$ after removing a maximal inverse substring $w_{i-1} \cdots w_1$ of $w$.
\item $\pi(w) = w'$ if there are no $a$ and $b$ such that $\ba, b \in Q_1$ and $b w a$ is defined as a string, where $w' = w_j \cdots w_i$ after removing a maximal direct substring $w_n \cdots w_{j+1}$ and a maximal inverse substring $w_{i-1} \cdots w_1$.
\item $\pi(w) = w$ if $w$ is a band.
\end{enumerate}
Then $\P_{\pi(w)}$ (resp., $\B_{\pi(w)}$ when $w$ is a band) is a projective resolution of $M(w)$ (resp., $B(w)$). 
\end{corollary}

\begin{proof}
The computation of the cohomology of $\P_{\pi(w)}$ (resp., $\B_{\pi(w)}$) in Theorem~\ref{thm:cohomology} gives $M(w)$ (resp., $B(w)$) in cohomological degree zero and zero in all other degrees.
\end{proof}

\begin{corollary} \label{cor:band-pd-one}
Let $A$ be a gentle algebra. Then any quasi-simple band module has projective dimension one.
\end{corollary}

The maximal direct substring $w_n \cdots w_{j+1}$ removed from $w$ in Corollary~\ref{cor:projective-resolution}(2) will be called a \emph{maximal direct suffix}. Likewise, the maximal inverse substring $w_{i-1} \cdots w_1$ removed from $w$ in Corollary~\ref{cor:projective-resolution}(3) will be called a \emph{maximal inverse prefix}.

\begin{definition}
For the homotopy string $\sigma = \pi(w)$ defined in Corollary~\ref{cor:projective-resolution} above we call the homotopy substrings $\inverse(a)$ and $\direct(b)$ the \emph{antipath part} of $\pi(w)$. By abuse of notation we write $\inverse(w) = \inverse(a)$ and $\direct(w) = \direct(b)$. In the notation of Corollary~\ref{cor:projective-resolution}, we will call $w'$ the \emph{module part} of $\pi(w)$.

An inverse homotopy letter $\sigma_i = \bar{a}_1 \cdots \bar{a}_k$ of $\sigma$ is \emph{incident with $\inverse(a)$} if $\bar{a}_k = \bar{a}$. Likewise, a direct homotopy letter $\sigma_j = b_l \cdots b_1$ of $\sigma$ is \emph{incident with $\direct(b)$} if $b_1 = b$.
\end{definition}

In the following, as usual, we write $\Q_{\pi(w)}$ when we do not wish to specify whether $w$ is a string or a band.

\begin{remark} \label{rem:observations}
We make the following straightforward observations regarding the forms of the homotopy strings occurring in Corollary~\ref{cor:projective-resolution}. 
\begin{enumerate}
\item If there is no $a$ such that $wa$ is defined as a string then the homotopy string $\pi(w)$ starts with a direct homotopy letter whose target lies in degree $0$. \label{starts}
\item If there is no $b$ such that $bw$ is defined as a string 
then the homotopy string $\pi(w)$ ends with an inverse homotopy letter whose target lies in degree $0$. \label{ends}
\item If $\sigma_{i+1} \sigma_i$ are consecutive homotopy letters with the same orientation then at least one of them lies in the antipath part and the other either lies in the antipath part or else is incident with $\direct(w)$ or $\inverse(w)$.
\item Owing to being a projective resolution of a module, the string/band complex $\Q_{\pi(w)}$ attains its maximal cohomological degree in degree $0$. Moreover, homotopy letters occurring in the module part of $\pi(w)$ provide components of the differential in $\Q_{\pi(w)}$ from degree $-1$ to degree $0$. 
Indeed, together with those homotopy letters incident with $\direct(w)$ and $\inverse(w)$ these provide all components of the differential in $\Q_{\pi(w)}$ from degree $-1$ to degree $0$. \label{degrees}
\item Suppose $\sigma_k$ is a homotopy letter of $\pi(w)$. If $\length(\sigma_k) > 1$ then $\deg(P(e(\sigma_k))) \in \{0,-1\}$ and $\deg(P(s(\sigma_k))) \in \{-1,0\}$, where $\deg(P(x))$ denotes the cohomological degree in which $P(x)$ occurs. \label{lengths}
\end{enumerate}
\end{remark}

\bigskip

\section{Determining extensions in the module category}\label{sec:determining extensions in module category}

Recall that in \cite{Schroer} extensions for string modules are given in terms of string combinatorics. 

\begin{definition} \label{def:extensions}
Let $v$ be a string or band and $w$ be a string or band.
\begin{enumerate}[label=(\arabic*)]
\item \label{arrow} \textbf{(Arrow extension)}
If there exists $a \in Q_1$ such that $u = w a v$ is a string then there is a non-split short exact sequence
\[ 0 \to M(w) \to M(u) \to M(v) \to 0. \] 
\item \label{overlap} \textbf{(Overlap extension)} 
Suppose that $v = v_L \bar{B} m A v_R$ and  $w =     w_L D m \bar{C} w_R$ with $A,B,C,D \in Q_1$ and $m, v_L, v_R, w_L, w_R$ (possibly trivial) strings such that  
\begin{enumerate}[label=(\roman*)]
\item if $A = \emptyset$ then $C \neq \emptyset$;
\item if $B = \emptyset$ then $D \neq \emptyset$; and,
\item \label{compatibly-oriented} if $m = 1_x$ for some $x \in Q_0$, i.e. a trivial string, then $CA \in I$ and $BD \in I$ (whenever they exist, subject to the constraints above). 
\end{enumerate} 
Then there is a non-split short exact sequence
\[ 0 \to M(w) \to M( u     ) \oplus M(u') \to M(v) \to 0\]
where $u = w_L D m A v_R$ and $u' = v_L \bar{B} m \bar{C} w_R$.
\end{enumerate}
\end{definition}

\begin{remark}
Condition~\ref{compatibly-oriented} is a `compatibly oriented' condition corresponding to \cite[Def 2.1]{CPS}.
We remark that this condition is missing in the definition of overlap extension in \cite{CS}, but it is used implicitly in the arguments therein.
\end{remark}

Recall the canonical isomorphism \eqref{isomorphism} from the introduction,
\[
\Phi: \Hom_{\KminusL}(\Q_{\pi(v)}, \Sigma \Q_{\pi(w)}) \stackrel{\sim}{\to} \Ext^1_\Lambda (M(v),M(w)).
\]
 
\begin{theorem}\label{thm induced extensions} 
With the notation above, let $M(v)$ and $M(w)$ be indecomposable $\Lambda$-modules with strings or bands $w$ and $v$ respectively and let $\Q_\sigma$ and $\Q_\tau$ with $\sigma = \pi(v)$ and $\tau = \pi(w)$ be their  projective resolutions. Then 
for any standard basis element $\f$ in $\ \Hom_{\KminusL}(\Q_\sigma, \Sigma \Q_\tau) $ the corresponding extension $\Phi(\f)$ in $\Ext^1_\Lambda (M(v),M(w))$ is given by an arrow or an overlap extension. In particular, the set of overlap and arrow extensions form a generating set for $\Ext^1_\Lambda (M(v),M(w))$.
\end{theorem}

In the rest of this section we prove Theorem~\ref{thm induced extensions} by considering each type of map of the standard basis of $\ \Hom_{\KminusL}(\Q_\sigma, \Sigma \Q_\tau) $  as defined in \cite{ALP}. We start by showing that Theorem~\ref{thm induced extensions} holds for graph maps. 

\subsection{Graph maps}

Throughout this subsection we fix the following setup. 

\begin{setup} \label{setup:quasi}
Let $v$ and $w$ be strings or bands and $M(v)$ and $M(w)$ be the corresponding string or band modules. 
Let $\sigma = \pi(v)$ and $\tau = \pi(w)$ be the homotopy strings or bands corresponding to the projective resolutions $\Q_\sigma$ and $\Q_\tau$ of $M(v)$ and $M(w)$ as given in Corollary~\ref{cor:projective-resolution}, respectively.
\end{setup}

\begin{lemma}\label{Graph map in antipaths}
Let $\f \colon \Q_{\sigma} \to \Sigma \Q_{\tau}$ be a graph map incident with an antipath in $ \Q_{\sigma}$ and an antipath in $ \Sigma \Q_{\tau}$. Then $\Phi(\f)$ is an arrow extension in $\Ext^1_\Lambda (M(v),M(w))$. 
\end{lemma}
 
\begin{proof} 
Recall  that $\pi(v) = \sigma$ and $\pi(w) = \tau$. 
We treat the case $\pi(v) = \direct(a) v \phi$ and  $\pi(w) = \phi' w \, \inverse(a')$ in detail, where  $a, \bar{a}' \in Q_1$, $\phi$ is an inverse antipath and $\phi'$ is a direct antipath, putting $\pi(v)$ and $\pi(w)$ both in case $(1)$ of Corollary~\ref{cor:projective-resolution}. If one of $\phi = \emptyset$ or $\phi' = \emptyset$, i.e. if $\pi(v) = \direct(a)v'$ or $\pi(w) = w' \inverse(a')$, then the calculations below remain essentially the same, even in the case that $v$ is inverse and $w$ is direct.

To simplify the notation, set $\direct(a) = \theta = \cdots \theta_2 \theta_1 $ and $\inverse(a') = \theta' = \theta'_1 \theta'_2 \cdots$ with $\theta_i, \bar{\theta}'_i \in Q_1$ and $\theta_1 = a$ and $\theta'_1 = a'$. 
Suppose that $\f$ induces an isomorphism of projective modules lying in $\theta$ and $\theta'$ and suppose this isomorphism is in degree $-n$. Then as homotopy letters in antipaths are of length 1 and since $\Lambda$ is gentle, there exists an isomorphism $\theta_n \simeq \bar{\theta}'_{n-1}$ and we obtain an isomorphism in degree $-n-1$. We now continue inductively to the left and right; we only need to take care about what happens in degree $-1$, which we analyse in cases below.
Write $ v = v_k \cdots v_1$ and $w = w_l \cdots w_1$. 

\smallskip
\noindent \textbf{Case 1:} \textit{$v_k$ is inverse and $w_1$ is direct.}
\smallskip

We have the following unfolded diagram
\[ 
\xymatrix@!R=5px{
\Q_\sigma \colon \ar[d]_-{\f} & \ar@{.}[r] & \xydot \ar[r]^-{\theta_3} \ar@{=}[d] & \xydot \ar[r]^-{\theta_ 2} \ar@{=}[d] & \xydot \ar[r]^-{a} \ar@{=}[d]  & \xydot   & \ar[l]_-{v_k \cdots v_i} \xydot \ar[r]  & \xydot       \ar@{~}[r]   & \\
\Sigma \Q_{\bar{\tau}} \colon      & \ar@{.}[r] & \xydot \ar[r]_-{\bar{\theta}'_2}                 & \xydot \ar[r]_-{\bar{a}' = \theta_2}                   & \xydot                  & \xydot \ar[l]^-{\bar{w}_1 \cdots \bar{w}_j}                          \ar[r]                     & \xydot       \ar@{~}[r]        & 
}
\]
 where $1 \leq i \leq k$ and $1 \leq j  \leq l$.
Then by \cite[Thm. 2.2]{CPS} the homotopy string of the (shift of the) mapping cone of $\f$ is given by $\alpha= \phi' w a v \phi $. By the form of $\pi(w)$ and $\pi(v)$, 
it follows from Corollary~\ref{cor:projective-resolution} that there exist $\bar{b}, b' \in Q_1$ such that $\phi =  b \rho$ and $\phi' = \rho' b'$ and $b' w a v b$ is a string. Then by  Lemma~\ref{lem:max-alternating}, $M(w a v)$ is the cohomology (in cohomological degree zero) of $\Q_\alpha$.

The shift of the mapping cone, $\Q_\alpha$, by definition sits in a distinguished triangle 
\begin{equation} \label{triangle}
\Q_\tau \rightlabel{\h} \Q_\alpha \rightlabel{\g} \Q_\sigma \rightlabel{\f} \Sigma \Q_\tau .
\end{equation}
We now observe that $H^0(\g) \coloneqq g \colon M(w a v) \to M(v)$ is the canonical map in the arrow extension, showing that the corresponding graph map does indeed induce the claimed arrow extension.

Decompose $v = \nu_n \cdots \nu_1$ and $w = \mu_m \cdots \mu_1$ into homotopy letters so that $\sigma = \theta \nu_n \cdots \nu_1 b \rho$ and $\alpha = \phi' \mu_m \cdots \mu_1 a \nu_n \cdots \nu_1 b \rho$. We assume that $\nu_1$ is direct so that $b$ is a homotopy letter; the case $\nu_1$ is inverse is similar. The map $\g \colon \Q_\alpha \to \Q_\sigma$ is given by the following unfolded diagram
\[
\xymatrix@!R=5px{
\Q_\alpha \colon \ar[d]_-{\g}                  & \ar@{~}[r]                     & \xydot                            & \ar[l] \xydot \ar[r]^-{\mu_1 a} \ar[d]_-{\mu_1}  & \xydot  \ar@{=}[d] & \ar[l]_-{\nu_n} \xydot \ar[r] \ar@{=}[d] & \cdots  \ar@{-}[r]   & \xydot \ar[r]^-{\nu_1} \ar@{=}[d] & \xydot \ar@{=}[d] & \ar[l]_-{b} \xydot \ar@{=}[d] & \ar[l]_-{\phi_2} \xydot \ar@{=}[d]  \ar@{.}[r] & \\
\Q_\sigma \colon                   \ar@{.}[r] & \xydot \ar[r]_-{\theta_3} & \xydot \ar[r]_-{\theta_ 2} & \xydot \ar[r]_-{a}                                                                                                       & \xydot                   & \ar[l]^-{\nu_n} \xydot \ar[r]                   & \cdots      \ar@{-}[r]   & \xydot \ar[r]_-{\nu_1} & \xydot & \ar[l]^-{b} \xydot & \ar[l]^-{\phi_2} \xydot \ar@{.}[r] &
}
\]
which is supported in cohomological degree $-1$ at the left endpoint. Wrapping $\alpha$ and $\sigma$ back up into complexes as in the proof of Lemma~\ref{lem:max-alternating}, where we have taken a `mirror image' of $\sigma$ in order to more easily match up the cohomological degree $0$ parts, we get the following diagram.
\begin{equation} \label{eq:cohomology-map}
\scalebox{0.9}{\xymatrix@!R=3mm{
                                                            & P(e(\mu_1 a)) \ar[r]^-{\omega a} \ar@{}[d]|\oplus                 & P(e(\nu_n)) \ar@{}[d]|\oplus \ar@{=}[r]       & P(e(\nu_n)) \ar@{}[d]|\oplus & \ar[l]_-{a} P(e(a)) \ar@{}[d]|\oplus  & \\
                                                             & P(s(\nu_{n})) \ar[ur]^-{\nu_n} \ar[r]^-{\nu_{n-1}} \ar@{}[d]|\oplus        & P(s(\nu_{n-1})) \ar@{}[d]|\oplus  \ar@{=}[r] & P(s(\nu_{n-1})) \ar@{}[d]|\oplus & P(s(\nu_{n})) \ar[ul]_-{\nu_n} \ar[l]_-{\nu_{n-1}} \ar@{}[d]|\oplus & \\
                                                             & P(s(\nu_{n-2})) \ar[ur]^-{\nu_{n-2}} \ar[r]^-{\nu_{n-3}} \ar@{}[d]|\oplus  & \, \vdots \ar@{}[d]|\oplus  \ar@{}[r]|\vdots     &  \vdots \, \ar@{}[d]|\oplus & P(s(\nu_{n-2})) \ar[ul]_-{\nu_{n-2}} \ar[l]_-{\nu_{n-3}} \ar@{}[d]|\oplus & \\
                                                             & \vdots \, \ar[r]^-{\nu_1} \ar@{}[d]|\oplus                                                 & P(s(\nu_1))  \ar@{=}[r]                                  & P(s(\nu_1))                & \, \vdots \ar[l]_-{\nu_1} \ar@{}[d]|\oplus  & \\
P(s(\phi_2)) \ar[r]^-{\phi_2} & P(s(b)) \ar[ur]^-{b}                                                                &                                                                       &  & P(s(b)) \ar[ul]_-{b} & P(s(\phi_2)) \ar[l]_-{\phi_2}
}}
\end{equation}
In Figure~\ref{fig:cohomology-map}, we re-write diagram \eqref{eq:cohomology-map} as in Figure~\ref{fig:coh}. Here we see immediately that $H^0(\g)$ is the canonical factor map $M(wav) \onto M(v)$.
\begin{figure} 
\scalebox{0.55}{\begin{tikzpicture}

\coordinate[circle,fill=black!20,draw, scale=1.2,label=left:{$\quad$},label=center:{$\bullet$}] (a1);
\coordinate[circle,below right=1cm of a1,label=right:{$\quad$}] (a1');
\coordinate[circle,below left=1cm of a1,label=right:{$\quad$}] (a1'');
\coordinate[circle,above left=2cm of a1,label=above:{$s(\nu_{n-1})$},label=center:{$\bullet$}] (a2);
\coordinate[circle,fill=black!20,draw, scale=1.2,below left=2cm of a2,label=above:{$\quad$},label=center:{$\bullet$}] (a3);
\coordinate[circle,below right=1cm of a3,label=below:{$\quad$}] (a3');
\coordinate[circle,below left=1cm of a3,label=below:{$\quad$}] (a3'');
\coordinate[circle,above left=2cm of a3,label=above:{$e(\nu_n)$},label=center:{$\bullet$}] (a4);
\coordinate[circle,below left=2cm of a4,label=below:{$\quad$},label=center:{$\bullet$}] (a5);
\coordinate[circle,below left=2cm of a5,label=below:{$\quad$},label=center:{$\bullet$}] (a6);
\coordinate[circle,below right=1cm of a6,label=below:{$\quad$}] (a6');
\coordinate[circle,below left=1cm of a6,label=below:{$\quad$}] (a6'');
\coordinate[circle,above left=2cm of a6,label=above:{$s(\mu_3)$},label=center:{$\bullet$}] (a7);
\coordinate[circle,below left=2cm of a7,label=below:{$\quad$},label=center:{$\bullet$}] (a8);
\coordinate[circle,below right=1cm of a8,label=below:{$\quad$}] (a8');
\coordinate[circle,below left=1cm of a8,label=below:{$\quad$}] (a8'');

\draw[->,decorate,decoration={snake,amplitude=.4mm,segment length=2mm}] 
(a1)-- node [anchor=south west,scale=.9]{$\alpha_{n-2}$} (a1');
\draw[->,decorate,decoration={snake,amplitude=.4mm,segment length=2mm}] 
(a1)-- node [anchor=south east,scale=.9]{$\beta_{n-2}$} (a1'');
\draw[->,decorate,decoration={snake,amplitude=.4mm,segment length=2mm}] 
(a2)-- node [anchor=south west,scale=.9]{$\nu_{n-2}$} (a1);
\draw[->,decorate,decoration={snake,amplitude=.4mm,segment length=2mm}] 
(a2)-- node [anchor=south east,scale=.9]{$\nu_{n-1}$} (a3);
\draw[->,decorate,decoration={snake,amplitude=.4mm,segment length=2mm}] 
(a3)-- node [anchor=south west,scale=.9]{$\alpha_{n}$} (a3');
\draw[->,decorate,decoration={snake,amplitude=.4mm,segment length=2mm}] 
(a3)-- node [anchor=south east,scale=.9]{$\beta_{n}$} (a3'');
\draw[->,decorate,decoration={snake,amplitude=.4mm,segment length=2mm}] 
(a4)-- node [anchor=south west,scale=.9]{$\nu_{n}$} (a3);
\draw[<-]  (a5)--node [anchor=south east,scale=.9]{$a$}(a4); 
\draw[->,decorate,decoration={snake,amplitude=.4mm,segment length=2mm}] 
(a5)-- node [anchor=south east,scale=.9]{$\mu_{1}$} (a6);
\draw[->,decorate,decoration={snake,amplitude=.4mm,segment length=2mm}] 
(a6)-- node [anchor=south west,scale=.9]{$\gamma_{2}$} (a6');
\draw[->,decorate,decoration={snake,amplitude=.4mm,segment length=2mm}] 
(a6)-- node [anchor=south east,scale=.9]{$\delta_{2}$} (a6'');
\draw[->,decorate,decoration={snake,amplitude=.4mm,segment length=2mm}] 
(a7)-- node [anchor=south west,scale=.9]{$\mu_{2}$} (a6);
\draw[->,decorate,decoration={snake,amplitude=.4mm,segment length=2mm}] 
(a7)-- node [anchor=south east,scale=.9]{$\mu_{3}$} (a8);
\draw[->,decorate,decoration={snake,amplitude=.4mm,segment length=2mm}] 
(a8)-- node [anchor=south west,scale=.9]{$\gamma_{4}$} (a8');
\draw[->,decorate,decoration={snake,amplitude=.4mm,segment length=2mm}] 
(a8)-- node [anchor=south east,scale=.9]{$\delta_{2}$} (a8'');

\draw [line width=15pt,opacity=0.15,black,line cap=round,rounded corners] (a1.center) -- (a2.center) -- (a3.center)--(a4.center)--(a5.center)--(a6.center)--(a7.center)--(a8.center);


\coordinate[circle,above right=.5cm of a1, label=above:{$\quad$}] (c');
\coordinate[circle,right=1cm of c', label=above:{$\cdots$}] (c);
\coordinate[circle,below right=.5cm of c, label=above:{$\quad$}] (c'');

\coordinate[circle,fill=black!20,draw, scale=1.2,right=1cm of c'',label=center:{$\bullet$}] (a11);
\coordinate[circle,below right=1cm of a11,label=right:{$\quad$}] (a11');
\coordinate[circle,below left=1cm of a11,label=right:{$\quad$}] (a11'');
\coordinate[circle,above right=2cm of a11,label=above:{$s(\nu_{5})$},label=center:{$\bullet$}] (a12);
\coordinate[circle,fill=black!20,draw, scale=1.2,below right=2cm of a12,label=above:{$\quad$},label=center:{$\bullet$}] (a13);
\coordinate[circle,below right=1cm of a13,label=below:{$\quad$}] (a13');
\coordinate[circle,below left=1cm of a13,label=below:{$\quad$}] (a13'');
\coordinate[circle,above right=2cm of a13,label=above:{$s(\nu_{3})$},label=center:{$\bullet$}] (a14);
\coordinate[circle,fill=black!20,draw, scale=1.2,below right=2cm of a14,label=above:{$\quad$},label=center:{$\bullet$}] (a15);
\coordinate[circle,below right=1cm of a15,label=below:{$\quad$}] (a15');
\coordinate[circle,below left=1cm of a15,label=below:{$\quad$}] (a15'');
\coordinate[circle,above right=2cm of a15,label=above:{$s(\nu_{1})$},label=center:{$\bullet$}] (a16);
\coordinate[circle,below right=2cm of a16,label=above:{$\quad$},label=center:{$\bullet$}] (a17);
\coordinate[circle,below right=1cm of a17,label=below:{$\quad$}] (a17');

\draw[->,decorate,decoration={snake,amplitude=.4mm,segment length=2mm}] 
(a11)-- node [anchor=south west,scale=.9]{$\alpha_{6}$} (a11');
\draw[->,decorate,decoration={snake,amplitude=.4mm,segment length=2mm}] 
(a11)-- node [anchor=south east,scale=.9]{$\beta_{6}$} (a11'');
\draw[->,decorate,decoration={snake,amplitude=.4mm,segment length=2mm}] 
(a12)-- node [anchor=south east,scale=.9]{$\nu_{5}$} (a11);
\draw[->,decorate,decoration={snake,amplitude=.4mm,segment length=2mm}] 
(a12)-- node [anchor=south west,scale=.9]{$\nu_{4}$} (a13);
\draw[->,decorate,decoration={snake,amplitude=.4mm,segment length=2mm}] 
(a13)-- node [anchor=south west,scale=.9]{$\alpha_{4}$} (a13');
\draw[->,decorate,decoration={snake,amplitude=.4mm,segment length=2mm}] 
(a13)-- node [anchor=south east,scale=.9]{$\beta_{4}$} (a13'');
\draw[->,decorate,decoration={snake,amplitude=.4mm,segment length=2mm}] 
(a14)-- node [anchor=south east,scale=.9]{$\nu_{3}$} (a13);
\draw[->,decorate,decoration={snake,amplitude=.4mm,segment length=2mm}] 
(a14)-- node [anchor=south west,scale=.9]{$\nu_{2}$} (a15);
\draw[->,decorate,decoration={snake,amplitude=.4mm,segment length=2mm}] 
(a15)-- node [anchor=south west,scale=.9]{$\alpha_{2}$} (a15');
\draw[->,decorate,decoration={snake,amplitude=.4mm,segment length=2mm}] 
(a15)-- node [anchor=south east,scale=.9]{$\beta_{2}$} (a15'');
\draw[->,decorate,decoration={snake,amplitude=.4mm,segment length=2mm}] 
(a15)-- node [anchor=south east,scale=.9]{$\nu_{1}$} (a16);
\draw[->] (a16)-- node [anchor=south west,scale=.9]{$\overline{b}$} (a17);
\draw[->,decorate,decoration={snake,amplitude=.4mm,segment length=2mm}] 
(a17)-- node [anchor=south west,scale=.9]{$\alpha_{0}$} (a17');

\draw [line width=15pt,opacity=0.15,black,line cap=round,rounded corners] (a11.center) -- (a12.center) -- (a13.center)--(a14.center)--(a15.center)--(a16.center);



\coordinate[circle,fill=black!20,draw, scale=1.2,below=4cm of a1,label=center:{$\bullet$}] (b1);
\coordinate[circle,below right=1cm of b1,label=right:{$\quad$}] (b1');
\coordinate[circle,below left=1cm of b1,label=right:{$\quad$}] (b1'');
\coordinate[circle,above left=2cm of b1,label=above:{$\quad$},label=center:{$\bullet$}] (b2);
\coordinate[circle,fill=black!20,draw, scale=1.2,below left=2cm of b2,label=above:{$\quad$},label=center:{$\bullet$}] (b3);
\coordinate[circle,below right=1cm of b3,label=below:{$\quad$}] (b3');
\coordinate[circle,below left=1cm of b3,label=below:{$\quad$}] (b3'');
\coordinate[circle,above left=2cm of b3,label=above:{$\quad$},label=center:{$\bullet$}] (b4);
\coordinate[circle,below left=2cm of b4,label=below:{$\quad$},label=center:{$\bullet$}] (b5);
\coordinate[circle,below left=2cm of b5,label=below:{$\quad$},label=center:{$\bullet$}] (b6);
\coordinate[circle,below left=1cm of b6,label=below:{$\quad$}] (b6'');

\draw[->,decorate,decoration={snake,amplitude=.4mm,segment length=2mm}] 
(b1)-- node [anchor=south west,scale=.9]{$\alpha_{n-2}$} (b1');
\draw[->,decorate,decoration={snake,amplitude=.4mm,segment length=2mm}] 
(b1)-- node [anchor=south east,scale=.9]{$\beta_{n-2}$} (b1'');
\draw[->,decorate,decoration={snake,amplitude=.4mm,segment length=2mm}] 
(b2)-- node [anchor=south west,scale=.9]{$\nu_{n-2}$} (b1);
\draw[->,decorate,decoration={snake,amplitude=.4mm,segment length=2mm}] 
(b2)-- node [anchor=south east,scale=.9]{$\nu_{n-1}$} (b3);
\draw[->,decorate,decoration={snake,amplitude=.4mm,segment length=2mm}] 
(b3)-- node [anchor=south west,scale=.9]{$\alpha_{n}$} (b3');
\draw[->,decorate,decoration={snake,amplitude=.4mm,segment length=2mm}] 
(b3)-- node [anchor=south east,scale=.9]{$\beta_{n}$} (b3'');
\draw[->,decorate,decoration={snake,amplitude=.4mm,segment length=2mm}] 
(b4)-- node [anchor=south west,scale=.9]{$\nu_{n}$} (b3);
\draw[<-]  (b5)--node [anchor=south east,scale=.9]{$a$}(b4); 
\draw[->,decorate,decoration={snake,amplitude=.4mm,segment length=2mm}] 
(b5)-- node [anchor=south east,scale=.9]{$\mu_{1}$} (b6);
\draw[->,decorate,decoration={snake,amplitude=.4mm,segment length=2mm}] 
(b6)-- node [anchor=south east,scale=.9]{$\delta_{2}$} (b6'');

\draw [line width=15pt,opacity=0.15,black,line cap=round,rounded corners] (b1.center) -- (b2.center) -- (b3.center)--(b4.center);


\coordinate[circle,above right=.5cm of b1, label=above:{$\quad$}] (bc');
\coordinate[circle,right=1cm of bc', label=above:{$\cdots$}] (bc);
\coordinate[circle,below right=.5cm of bc, label=above:{$\quad$}] (bc'');

\coordinate[circle,fill=black!20,draw, scale=1.2,right=1cm of bc'',label=center:{$\bullet$}] (b11);
\coordinate[circle,below right=1cm of b11,label=right:{$\quad$}] (b11');
\coordinate[circle,below left=1cm of b11,label=right:{$\quad$}] (b11'');
\coordinate[circle,above right=2cm of b11,label=above:{$s(\nu_{5})$},label=center:{$\bullet$}] (b12);
\coordinate[circle,fill=black!20,draw, scale=1.2,below right=2cm of b12,label=above:{$\quad$},label=center:{$\bullet$}] (b13);
\coordinate[circle,below right=1cm of b13,label=below:{$\quad$}] (b13');
\coordinate[circle,below left=1cm of b13,label=below:{$\quad$}] (b13'');
\coordinate[circle,above right=2cm of b13,label=above:{$\quad$},label=center:{$\bullet$}] (b14);
\coordinate[circle,fill=black!20,draw, scale=1.2,below right=2cm of b14,label=above:{$\quad$},label=center:{$\bullet$}] (b15);
\coordinate[circle,below right=1cm of b15,label=below:{$\quad$}] (b15');
\coordinate[circle,below left=1cm of b15,label=below:{$\quad$}] (b15'');
\coordinate[circle,above right=2cm of b15,label=above:{$\quad$},label=center:{$\bullet$}] (b16);
\coordinate[circle,below right=2cm of b16,label=above:{$\quad$},label=center:{$\bullet$}] (b17);
\coordinate[circle,below right=1cm of b17,label=below:{$\quad$}] (b17');

\draw[->,decorate,decoration={snake,amplitude=.4mm,segment length=2mm}] 
(b11)-- node [anchor=south west,scale=.9]{$\alpha_{6}$} (b11');
\draw[->,decorate,decoration={snake,amplitude=.4mm,segment length=2mm}] 
(b11)-- node [anchor=south east,scale=.9]{$\beta_{6}$} (b11'');
\draw[->,decorate,decoration={snake,amplitude=.4mm,segment length=2mm}] 
(b12)-- node [anchor=south east,scale=.9]{$\nu_{5}$} (b11);
\draw[->,decorate,decoration={snake,amplitude=.4mm,segment length=2mm}] 
(b12)-- node [anchor=south west,scale=.9]{$\nu_{4}$} (b13);
\draw[->,decorate,decoration={snake,amplitude=.4mm,segment length=2mm}] 
(b13)-- node [anchor=south west,scale=.9]{$\alpha_{4}$} (b13');
\draw[->,decorate,decoration={snake,amplitude=.4mm,segment length=2mm}] 
(b13)-- node [anchor=south east,scale=.9]{$\beta_{4}$} (b13'');
\draw[->,decorate,decoration={snake,amplitude=.4mm,segment length=2mm}] 
(b14)-- node [anchor=south east,scale=.9]{$\nu_{3}$} (b13);
\draw[->,decorate,decoration={snake,amplitude=.4mm,segment length=2mm}] 
(b14)-- node [anchor=south west,scale=.9]{$\nu_{2}$} (b15);
\draw[->,decorate,decoration={snake,amplitude=.4mm,segment length=2mm}] 
(b15)-- node [anchor=south west,scale=.9]{$\alpha_{2}$} (b15');
\draw[->,decorate,decoration={snake,amplitude=.4mm,segment length=2mm}] 
(b15)-- node [anchor=south east,scale=.9]{$\beta_{2}$} (b15'');
\draw[->,decorate,decoration={snake,amplitude=.4mm,segment length=2mm}] 
(b15)-- node [anchor=south east,scale=.9]{$\nu_{1}$} (b16);
\draw[->] (b16)-- node [anchor=south west,scale=.9]{$\overline{b}$} (b17);
\draw[->,decorate,decoration={snake,amplitude=.4mm,segment length=2mm}] 
(b17)-- node [anchor=south west,scale=.9]{$\alpha_{0}$} (b17');

\draw [line width=15pt,opacity=0.15,black,line cap=round,rounded corners] (b11.center) -- (b12.center) -- (b13.center)--(b14.center)--(b15.center)--(b16.center);


\draw[transform canvas={xshift=-1.5pt},color=red] (a1) -- (b1);
\draw[transform canvas={xshift=1.5pt},color=teal] (a1) -- (b1);

\draw[transform canvas={xshift=-1.5pt},color=teal] (a1'.center) -- (b1'.center);
\draw[transform canvas={xshift=1.5pt},color=teal] (a1'.center) -- (b1'.center);

\draw[transform canvas={xshift=-1.5pt},color=teal] (a2.center) -- (b2.center);
\draw[transform canvas={xshift=1.5pt},color=teal] (a2.center) -- (b2.center);

\draw[transform canvas={xshift=-1.5pt},color=red] (a3) -- (b3);
\draw[transform canvas={xshift=1.5pt},color=teal] (a3) -- (b3);

\draw[transform canvas={xshift=-1.5pt},color=red] (a3'.center) -- (b3'.center);
\draw[transform canvas={xshift=1.5pt},color=red] (a3'.center) -- (b3'.center);

\draw[transform canvas={xshift=-1.5pt},color=teal] (a3''.center) -- (b3''.center);
\draw[transform canvas={xshift=1.5pt},color=teal] (a3''.center) -- (b3''.center);

\draw[transform canvas={xshift=-1.5pt},color=red] (a4.center) -- (b4.center);
\draw[transform canvas={xshift=1.5pt},color=red] (a4.center) -- (b4.center);

\draw[transform canvas={xshift=-1.5pt},color=red] (a5.center) -- (b5.center);
\draw[transform canvas={xshift=1.5pt},color=red] (a5.center) -- (b5.center);

\draw[transform canvas={xshift=-1.5pt},color=red] (a6.center) -- (b6.center);
\draw[transform canvas={xshift=1.5pt},color=red] (a6.center) -- (b6.center);

\draw[transform canvas={xshift=-1.5pt},color=red] (a6''.center) -- (b6''.center);
\draw[transform canvas={xshift=1.5pt},color=red] (a6''.center) -- (b6''.center);

\draw[transform canvas={xshift=-1.5pt},color=red] (a11) -- (b11);
\draw[transform canvas={xshift=1.5pt},color=teal] (a11) -- (b11);

\draw[transform canvas={xshift=-1.5pt},color=red] (a11''.center) -- (b11''.center);
\draw[transform canvas={xshift=1.5pt},color=red] (a11''.center) -- (b11''.center);

\draw[transform canvas={xshift=-1.5pt},color=red] (a12.center) -- (b12.center);
\draw[transform canvas={xshift=1.5pt},color=red] (a12.center) -- (b12.center);

\draw[transform canvas={xshift=-1.5pt},color=red] (a13) -- (b13);
\draw[transform canvas={xshift=1.5pt},color=teal] (a13) -- (b13);

\draw[transform canvas={xshift=-1.5pt},color=red] (a13'.center) -- (b13'.center);
\draw[transform canvas={xshift=1.5pt},color=red] (a13'.center) -- (b13'.center);

\draw[transform canvas={xshift=-1.5pt},color=teal] (a13''.center) -- (b13''.center);
\draw[transform canvas={xshift=1.5pt},color=teal] (a13''.center) -- (b13''.center);

\draw[transform canvas={xshift=-1.5pt},color=teal] (a14.center) -- (b14.center);
\draw[transform canvas={xshift=1.5pt},color=teal] (a14.center) -- (b14.center);

\draw[transform canvas={xshift=-1.5pt},color=red] (a15''.center) -- (b15''.center);
\draw[transform canvas={xshift=1.5pt},color=red] (a15''.center) -- (b15''.center);

\draw[transform canvas={xshift=-1.5pt},color=red] (a15.center) -- (b15.center);
\draw[transform canvas={xshift=1.5pt},color=teal] (a15.center) -- (b15.center);

\draw[transform canvas={xshift=-1.5pt},color=teal] (a15'.center) -- (b15'.center);
\draw[transform canvas={xshift=1.5pt},color=teal] (a15'.center) -- (b15'.center);

\draw[transform canvas={xshift=-1.5pt},color=red] (a16.center) -- (b16.center);
\draw[transform canvas={xshift=1.5pt},color=red] (a16.center) -- (b16.center);

\draw[transform canvas={xshift=-1.5pt},color=red] (a17.center) -- (b17.center);
\draw[transform canvas={xshift=1.5pt},color=red] (a17.center) -- (b17.center);

\draw[transform canvas={xshift=-1.5pt},color=red] (a17'.center) -- (b17'.center);
\draw[transform canvas={xshift=1.5pt},color=red] (a17'.center) -- (b17'.center);

\end{tikzpicture}}
\caption{The strings shaded in grey are the cohomology of $\Q_\alpha$ and $\Q_\sigma$ in degree $0$. The identity maps between projective modules are indicated in alternating red and green colour. The induced map between cohomology modules is the canonical factor map $M(wav) \onto M(v)$.}
\label{fig:cohomology-map}
\end{figure}
Taking the long exact cohomology sequence associated to the triangle \eqref{triangle} gives a short exact sequence
\[
0 \too M(w) \rightlabel{H^0(\h)} M(w a v) \rightlabel{H^0(\g)} M(v) \too 0,
\]
in which $H^0(\g)$ is the canonical map, whence it follows immediately that $H^0(\h)$ is also the canonical map associated to the obvious substring/factor string decomposition. 

It now follows that $\f$ induces an arrow extension corresponding to the arrow induced by $a$, where the  middle term  of the extension is given by  the string module $M(w a v)$.

\smallskip
\noindent \textbf{Case 2:} \textit{Both $v_k$ and $w_1$ are inverse.}
\smallskip
 
We have the following unfolded diagram
\[ 
\xymatrix@!R=5px{
\Q_\sigma \colon \ar[d]_-{\f} & \ar@{.}[r] & \xydot \ar[r]^-{\theta_3} \ar@{=}[d] & \xydot \ar[r]^-{\theta_2} \ar@{=}[d] & \xydot \ar[r]^-{a} \ar[d]^-{\bar{w}_1 \cdots \bar{w}_j}  & \xydot   & \ar[l]_-{v_k \cdots v_i} \xydot \ar[r]  & \xydot       \ar@{~}[r]   & \\
\Sigma \Q_{\bar{\tau}} \colon      & \ar@{.}[r] & \xydot \ar[r]_-{\bar{\theta}'_2 = \theta_3}                 & \xydot \ar[r]_-{\bar{a}' \bar{w}_1 \cdots \bar{w}_j}                   & \xydot                                                       & \xydot     \ar[l]        \ar@{~}[r]        & 
}
\]
where $1 \leq i \leq k$ and $1 \leq j \leq l$. Then $a'=\theta_1$ and
by \cite[Thm. 2.2]{CPS} the homotopy string of the mapping cone of $\f$ is given by $\alpha= \rho' b'  w a v b \rho$ where $ \phi = b \rho  $ and $\phi' = \rho' b'$ with $\bar{b}, b' \in Q_1$ such that $b' w a v b$ is a string.
As in Case 1 above, one can check that the map $H^0(\g) \colon M(wav) \to M(v)$ is the canonical map given by the obvious substring/factor string decomposition. It then follows that, taking cohomology, $\f$ induces an arrow extension, corresponding to the arrow $a$, whose middle term is $ M(wav)$.

\smallskip
\noindent \textbf{Case 3:} \textit{Both $v_k$ and $w_1$ are direct.}
\smallskip

This case is similar to case 1. We have the following unfolded diagram
\[ 
\xymatrix@!R=5px{
\Q_\sigma \colon \ar[d]_-{\f} & \ar@{.}[r] & \xydot \ar[r]^-{\theta_3} \ar@{=}[d] & \xydot \ar[r]^-{\theta_3} \ar@{=}[d] & \xydot \ar[r]^-{av_k \cdots v_i } \ar@{=}[d]  & \xydot   & \ar[l]  \xydot       \ar@{~}[r]   & \\
\Sigma \Q_{\bar{\tau}} \colon      & \ar@{.}[r] & \xydot \ar[r]_-{\bar{\theta}'_2 = \theta_3}                & \xydot \ar[r]_-{\bar{a}' = \theta_2}                   & \xydot                  & \xydot \ar[l]^-{\bar{w}_1 \cdots \bar{w}_j}                            \ar[r]                     & \xydot       \ar@{~}[r]        & 
}
\]
where $1 \leq i \leq k$ and $1 \leq j  \leq l$. Then as above the cohomology of the mapping cone induces an arrow extension corresponding to the arrow $a$.

\smallskip
\noindent \textbf{Case 4:} \textit{$v_k$ is direct and $w_1$ is inverse.}
\smallskip

This case is similar to case 2. We have the following unfolded diagram
\[ 
\xymatrix@!R=5px{
\Q_\sigma \colon \ar[d]_-{\f} & \ar@{.}[r] & \xydot \ar[r]^-{\theta_3} \ar@{=}[d] & \xydot \ar[r]^-{\theta_2} \ar@{=}[d] & \xydot \ar[r]^-{av_k \cdots v_i } \ar[d]^-{\bar{w}_1 \cdots \bar{w}_j}  & \xydot   & \ar[l]  \xydot       \ar@{~}[r]   & \\
\Sigma \Q_{\bar{\tau}} \colon      & \ar@{.}[r] & \xydot \ar[r]_-{\bar{\theta}'_1 = \theta_3}                 & \xydot \ar[r]_-{\bar{a}' \bar{w}_1 \cdots \bar{w}_j}                  & \xydot                                                       & \xydot     \ar[l]        \ar@{~}[r]        & 
}
\]
where $1 \leq i \leq k$ and $1 \leq j  \leq l$. Then as above the cohomology of the mapping cone induces an arrow extension corresponding to the arrow $a$. 

\smallskip
\noindent \textbf{Case 5:} \textit{$v$ or $w$ or both are trivial.}
\smallskip

If $v$ is trivial but $w$ is not, this is a degenerate case of Case 1 or 2. If $v$ is not trivial but $w$ is, this is a degenerate case of Case 1 or 3. If both $v$ and $w$ are trivial, this is a degenerate case of Case 1.
\end{proof}

\begin{lemma}\label{lem:graph map single degree}
Let $\f \colon \Q_{\sigma} \to \Sigma \Q_{\tau}$ be a graph map  and let $\nu$ in $ \sigma$ and $\omega$ in $ \tau$ be  the maximal alternating homotopy substrings corresponding to the module parts of $\sigma$ and $\tau$ respectively. Suppose that $\f$ is supported in projective modules lying in $\nu$  and $\omega$.  Then $\f$ is supported in a single indecomposable projective $\Lambda$-module $P$ in degree $-1$ unless it is incident with antipaths in both $\Q_\sigma$ and $\Sigma \Q_\tau$. 

Furthermore, $\Phi(\f)$ gives rise to either an arrow extension or an overlap extension where the overlap is given by the  simple $\Lambda$-module $P / \rad(P)$.
\end{lemma}

\begin{proof}
There are three cases to be considered. 

\smallskip
\noindent \textbf{Case 1:} \textit{$\f$ is supported in $\nu$, and $\f$ is not incident with any antipath of $\sigma$.}
\smallskip

In this case, we must have at least one isomorphism between projective modules in degree $-1$ as follows
\begin{equation}\label{eq:graph map single support}
\xymatrix@!R=5px{
 & & \text{\footnotesize 0} &  \text{\footnotesize -1} &  \text{\footnotesize 0} & \\
\Q_{\sigma} \colon \ar[d]_-{\f} &   \ar@{~}[r]             & \xydot            & x    \ar[l]_-{\nu_i}     \ar[r]^-{\nu_{i-1}}     \ar@{=}[d]                                                & \xydot                      \ar@{~}[r]  & \\
\Sigma \Q_{\tau} \colon      &   \ar@{~}[r]             & \xydot \ar[r]_-{\omega_j}                & x                                                      & \xydot     \ar[l]^-{\omega_{j-1}}         \ar@{~}[r]        & \\
& & \text{\footnotesize -2} & \text{\footnotesize -1} & \text{\footnotesize -2} & 
}
\end{equation}
where $x \in Q_0$.
Since the projectives in $\nu$ as a substring of  $\sigma$ are in cohomological degrees $0$ and $-1$ and the projectives in $\omega$ as a substring of the homotopy string corresponding to $\Sigma \Q_{\tau}$  are in degrees $-1$ and $-2$, the graph map $\f$ can only be supported in a single degree, as shown. By reversing the orientation on $\tau$ if necessary, we may assume that $\sigma$ and $\tau$ are compatibly oriented in the sense of \cite[Def 2.1]{CPS}.

Now, the homotopy letters $\nu_{i-1}$, $\nu_i$, $\omega_{j-1}$ and $\omega_j$ have the form $\nu_{i-1} = A\nu'_{i-1}$, $\nu_i = \nu'_i \bar{B}$, $\omega_{j-1} = \bar{D}\omega'_{j-1}$ and $\omega_j = \omega'_j C$, where $A, B, C, D \in Q_1$ and the primed symbols are homotopy subletters.  
Then $v= v_L \bar{B} A v_R$ and $w = w_L D \bar{C} w_R$  where  $v_L, v_R$ are (possibly trivial) subwords of $v$ and $w_L, w_R$ are (possibly trivial) subwords of $w$.
Set $x = e(A) \,  (= e(\bar{B}) = s(\bar{C}) = s(D))$. 
We wish to verify we have an overlap extension in which $m = 1_x$ in the sense of Definition~\ref{def:extensions}\ref{overlap}.

First observe that whenever $A,B,C,D$ exist, the fact that diagram~\eqref{eq:graph map single support} is compatibly oriented means that $CA \in I$ and $BD \in I$. We now need to check that if $A = \emptyset$ then $C \neq \emptyset$ and if $B = \emptyset$ then $D \neq \emptyset$. We check the first condition, the second is analogous.

Suppose that $C = \emptyset$, i.e. $w = w_L D $. If $A = \emptyset$ then $v = v_L \bar{B}$. There must be an arrow $a \in Q_1$ such that $v \bar{a}$ is defined as a string, otherwise by Corollary~\ref{cor:projective-resolution}, $\bar{B}$ would be removed as a maximal inverse prefix and the situation depicted in the unfolded diagram~\eqref{eq:graph map single support} would not occur. However, in this case the homotopy letters $\nu_i = \nu'_i \bar{B} \bar{a}$ and $\nu_{i-1}$ must be inverse or empty, again taking us outside the situation occurring in diagram~\eqref{eq:graph map single support}. Hence, we must have $A \neq \emptyset$, as required. A similar argument shows that if $B = \emptyset$ then $D \neq \emptyset$. We have thus verified that the conditions for an overlap extension in Definition~\ref{def:extensions}\ref{overlap} hold.

Finally, by \cite[\S 2]{CPS} the mapping cone of $\f$ is a direct sum of the projective resolutions of the $\Lambda$-modules $M(u)$ and $M(u')$ where $u = w_L D  e_x A v_R$ and $u' = v_L \bar{B} e_x \bar{C} w_R$. Taking cohomology and checking the maps in the corresponding triangle as in the proof of Lemma~\ref{Graph map in antipaths} then shows that $\f$ gives rise to an overlap extension in the simple $\Lambda$-module $S(x)$. 

\smallskip
\noindent \textbf{Case 2:} \textit{ $\f$ is incident with an antipath in $\Q_{\sigma}$ and the module part in $\Sigma \Q_{\tau}$.}
\smallskip

In this case we obtain the following diagram for $\f$.
\[ 
\xymatrix@!R=5px{
 & & \text{\footnotesize -3} &  \text{\footnotesize -2} &  \text{\footnotesize -1} & \\
\Q_{\sigma} \colon \ar[d]_-{\f} &   \ar@{~}[r]            & \xydot \ar[r]^-{\theta_2} & \xydot \ar[r]^-{\theta_1}               & \xydot         \ar[r]^-{a}    \ar@{=}[d]                                                & \xydot                  \ar@{~}[r]  & \\
\Sigma \Q_{\tau} \colon      &   \ar@{~}[r]         & \xydot \ar@{-}[r]_-{\phi'_1}        & \xydot \ar[r]_-{\omega_j}                & \xydot                                                       & \xydot     \ar[l]^-{\omega_{j-1}}         \ar@{~}[r]        & \\
& & & \text{\footnotesize -2} & \text{\footnotesize -1} & 
}
\]
Since $\theta_1$ is a homotopy letter of length 1, in order to obtain a graph map supported in more than one degree, we must have  $\theta_1 = \omega_j$. If $\phi'_1$ is inverse or empty then we reach a non-commuting endpoint condition as defined in Section~\ref{subsec:quasi-graph}, see also \cite[\S 1.4.4]{CPS} for more details on non-commuting endpoint conditions. This contradicts the fact that $\f$ is a graph map. Thus $\phi'_1$ must be direct and $\phi'_1 = \theta_{2}$. We are therefore in the setup of Lemma~\ref{Graph map in antipaths} and the corresponding extension in the module category is an arrow extension. 

\smallskip
\noindent \textbf{Case 3:} \textit{$\f$ is incident with the module part in $\Q_{\sigma}$ and an antipath in $\Sigma \Q_{\tau}$}
\smallskip

This case cannot happen for degree reasons.
\end{proof}

\begin{remark}
Suppose one of $v$ or $w$ is a band, in which case $\sigma$ or $\tau$ is a homotopy band. A priori, one might expect that there may be a graph map $\f \colon \Q_{\sigma} \to \Sigma \Q_{\tau}$ determined by an overlap which is longer than $\sigma$ or $\tau$. However, Lemmas~\ref{Graph map in antipaths} and \ref{lem:graph map single degree} ensure that this situation can never occur. This makes sense: in  \cite[\S 1]{Addendum} such a situation corresponds to a `shortening' of the homotopy string or homotopy band determining the mapping cone relative to those occurring in the domain and target of the map. This would correspond to having a middle term of an extension of smaller dimension than the sum of the outer terms. 
\end{remark}

\subsection{Quasi-graph maps}\label{sec:quasi-graph maps}

In this section we consider a quasi-graph map $\phi \colon \Q_\sigma \wiggle \Q_\tau$, corresponding to a homotopic family of single and double maps in the basis of $\Hom_{\KminusL}(\Q_\sigma, \Sigma \Q_\tau)$; see \cite[Def. 3.12]{ALP}.

We start by placing a restriction on the cohomological degrees in which a quasi-graph map $\phi \colon \Q_\sigma \wiggle \Q_\tau$ can be supported.

\begin{lemma} \label{lem:quasi-degrees}
Under the hypotheses of Setup~\ref{setup:quasi}, a quasi-graph map $\phi \colon \Q_\sigma \rightsquigarrow \Q_\tau$ is supported in cohomological degrees $-1$ and $0$ only.
\end{lemma}

\begin{proof}
If one of $\Q_\sigma$ or $\Q_\tau$ is a band complex then, by Corollary~\ref{cor:band-pd-one}, it is supported in cohomological degrees $-1$ and $0$ only, and therefore any quasi-graph map $\phi \colon \Q_\sigma \wiggle \Q_\tau$ is trivially supported in only those cohomological degrees. Therefore we assume that $\Q_\sigma = \P_\sigma$ and $\Q_\tau = \P_\tau$ are string complexes.

Suppose, for a contradiction, that $\phi \colon \P_\sigma \wiggle \P_\tau$ is supported in cohomological degree $-k \leq -2$.  
By Remark~\ref{rem:observations}\eqref{degrees}, any component of $\phi$ supported in degrees $-k \leq -2$ occurs in antipath parts of $\P_\sigma$ and $\P_\tau$. 
Without loss of generality, we may assume, up to inversion if necessary, that $\sigma = \direct(b) w' \sigma_R$ and $\tau = \direct(d) v' \tau_R$, where $\sigma_R$ is either an inverse antipath or empty; likewise for $\tau_R$. Thus, the antipath parts have the form
\[
\direct(b) = \cdots \theta_n \cdots \theta_2 \theta_1 
\quad \text{and} \quad
\direct(c) = \cdots \psi_n \cdots \psi_2 \psi_1,
\]
where $\theta_1 = b$ and $\psi_1 = d$ and $b, d \in Q_1$ are such that $bw$ and $dv$ are defined as strings. 

Since $\phi \colon \P_\sigma \wiggle \P_\tau$ is supported in cohomological degree $-k \leq -2$, we have the following subdiagram of the unfolded diagram for $\phi$.
\[
\xymatrix@!R=3px{
\ar@{..}[r] & \xydot \ar[r]^-{\theta_{k+1}} & \xydot \ar[r]^-{\theta_k} \ar@{=}[d] & \xydot \ar@{..}[r] & \\
\ar@{..}[r] & \xydot \ar[r]_-{\psi_{k+1}}    & \xydot \ar[r]_-{\psi_k}                      & \xydot \ar@{..}[r] & 
}
\]
We first show that $\phi$ is supported in degrees $-k-1$ and $-k+1$. Suppose that $\phi$ was not supported in cohomological degree $-k+1$, then inverting $\tau$ the unfolded diagram of $\phi$ would have the form,
\[
\xymatrix@!R=3px{
\ar@{..}[r] & \xydot \ar[r]^-{\theta_{k+1}} & \xydot \ar[r]^-{\theta_k} \ar@{=}[d] \ar@{}[dr]|{(*)}& \xydot \ar@{..}[r] & \\
\ar@{..}[r] & \xydot                                  & \ar[l]^-{\psi_k} \xydot                      & \ar[l]^-{\psi_{k+1}} \xydot \ar@{..}[r] & 
}
\]
where $(*)$ corresponds to the graph map right endpoint condition (RG3) in \cite[\S 1.4.1]{CPS}, whence by \cite[Rem. 4.9]{ALP} corresponds to a family of null-homotopic maps. Similarly, one can show that $\phi$ is supported in cohomological degree $-k-1$. This means that we can extend the subdiagram of the unfolded diagram of $\phi$ to the following,
\[
\xymatrix@!R=3px{
\ar@{..}[r] & \xydot \ar[r]^-{\theta_{k+1}} \ar@{=}[d] & \xydot \ar[r]^-{\theta_k} \ar@{=}[d] & \xydot \ar@{..}[r] \ar@{=}[d] & \\
\ar@{..}[r] & \xydot \ar[r]_-{\psi_{k+1}}                      & \xydot \ar[r]_-{\psi_k}                      & \xydot \ar@{..}[r]                 & 
}
\]
showing that $\theta_k = \psi_k$ for each $k \geq 2$. But this means that the unfolded diagram of $\phi$ satisfies (LG3) or (LG$\infty$) (cf. \cite[\S 1.4.1]{CPS}) and, therefore, invoking \cite[Rem. 4.9]{ALP} again, we see that $\phi$ corresponds to a null-homotopic family of single and double maps. This contradicts our assumption that $\phi$ is a quasi-graph map, therefore $\phi$ cannot be supported in cohomological degrees smaller than $-2$, as claimed.
\end{proof}

We now consider the endpoints of a quasi-graph map $\phi \colon \Q_\sigma \wiggle \Q_\tau$. Lemma~\ref{lem:quasi-degrees} says that they must occur in degrees $-1$ or $0$.
Recall the definition of homotopy strings or bands $\sigma$ and $\tau$ being \emph{compatibly oriented} for a quasi-graph map $\phi$ from \cite[Def. 3.1]{CPS}; note that if a quasi-graph map is supported in more than one degree it is automatically compatibly oriented in its unfolded form. 

\begin{lemma} \label{lem:quasi-right-0}
Suppose the quasi-graph map $\phi \colon \Q_\sigma \wiggle \Q_\tau$ has right endpoint in degree 0.
\begin{enumerate}[label=(\arabic*)]
\item The compatibly oriented unfolded diagram for $\phi$ has the following form at the right endpoint of $\phi$:
\[
\xymatrix@!R=3px{
\ar@{~}[r] & \xydot \ar[r]^-{\sigma_s} & x \ar@{<-}[r]^-{\sigma_R} \ar@{=}[d] & \xydot \ar@{~}[r]^-{\alpha} & \\
\ar@{~}[r] & \xydot \ar[r]_-{\tau_t}     & x \ar@{<-}[r]_-{\tau_R}                       & \xydot \ar@{~}[r]                &
}
\]
such that $\sigma_s , \sigma_R \neq \emptyset$, $\tau_t = \emptyset$ or $\tau_t = \sigma_s' \sigma_s$ for some (possibly nontrivial) $\sigma_s'$, and $\tau_R = \emptyset$ or $\tau_R = \sigma_R \sigma_R'$ for some nontrivial $\sigma'_R$.
\item Write $\sigma_R = \bar{a}_k \cdots \bar{a}_1$ and $\sigma_s = b_l \cdots b_1$ for $k,l \geq 1$ and $a_i,b_j \in Q_1$. Then
\begin{enumerate}[label=(\roman*)]
\item $v$ has a substring of the form
\[
\widetilde{v} =
\left\{
\begin{array}{ll}
b_{l-1} \cdots b_1 \bar{a}_k \cdots \bar{a}_2      & \text{if $\sigma_R$ is incident with $\inverse(v)$,} \\
b_{l-1} \cdots b_1 \bar{a}_k \cdots \bar{a}_1 a   & \text{for some $a \in Q_1$ otherwise;}
\end{array}
\right.
\]
\item $w$ has a substring of the form
\[
\widetilde{w} =
\left\{
\begin{array}{ll}
b_{l-1} \cdots b_1 \bar{a}_k \cdots \bar{a}_1 \bar{a}' & \text{for some $a' \in Q_1$ if $\tau_R \neq \emptyset$,} \\
b_{l-1} \cdots b_1 \bar{a}_k \cdots \bar{a}_1             & \text{otherwise.}
\end{array}
\right.
\]
\end{enumerate}
\end{enumerate}
\end{lemma}

\begin{proof}
$(1)$ Since $P(x)$ sits in degree zero it must be a sink for any differential incident with it because $\Q_\sigma$ and $\Q_\tau$ are projective resolutions. If $\sigma_s = \emptyset$ or $\sigma_R = \emptyset$, then the diagram indicates a graph map endpoint and $\phi \colon \Q_\sigma \wiggle \Q_\tau$ is not a quasi-graph map. Therefore, $\sigma_s, \sigma_R \neq \emptyset$. If $\tau_R \neq \emptyset$, the orientation of the differentials means that $\phi$ must satisfy the quasi-graph map right endpoint condition (RQ2), whence $\tau_R = \sigma_R \sigma'_R$ for some nontrivial $\sigma'_R$. The statement regarding $\tau_t$ just lists the possible cases that may occur with the given orientation.

$(2) (i)$ First note that $b_{l-1} \cdots b_1$ is a substring of $v$ by Corollary~\ref{cor:projective-resolution}.
If $\sigma_R$ is incident with $\inverse(v)$ then the first statement is clear. 
By Remark~\ref{rem:observations}\eqref{starts},  $\sigma$ cannot start with the inverse homotopy letter $\sigma_R$ unless it is incident with $\inverse(v)$. Thus, if $\sigma_R$ is not incident with $\inverse(v)$ then $\alpha$ must end with a direct homotopy letter, whose last arrow we denote by $a \in Q_1$, say, giving the required form for $\widetilde{v}$.

$(2) (ii)$ We treat this in cases. Firstly, if $\tau_t, \tau_R = \emptyset$, then $\Q_\tau$ is the stalk complex $P(x)$ concentrated in degree zero. Using the form of $\sigma_s$ and $\sigma_R$ we see that $P(x) \cong M(u)$ for some string $u = q b_l \cdots b_1 \bar{a}_k \cdots \bar{a}_1 \bar{p}$, where $q$ is a maximal direct string and $\bar{p}$ a maximal inverse string composable with $b_l$ and $\bar{a}_k$, respectively, as strings. The claim is now clear in this case.

Now assume that $\tau_t \neq \emptyset$ and $\tau_R = \emptyset$. By $(1)$, $\tau_t = \sigma'_s \sigma_s$, where $\sigma'_s$ is possibly trivial. Since $\tau_R = \emptyset$, $w$ either starts with $b_1$ (a direct arrow) or else $w$ has had a maximal inverse prefix removed. The former case cannot occur because $b_1 \bar{a}_k$ is defined as a string, which by Corollary~\ref{cor:projective-resolution} would make $\tau_R \neq \emptyset$. Thus, by gentleness, $w = u \bar{a}_1 w_R$ for some (possibly trivial) inverse string $w_R$. If $\bar{a}_k \cdots \bar{a}_1$ is a (possibly equal) substring of $\bar{a}_k w_R$ then $w$ contains the substring $\widetilde{w}$ as claimed. So suppose $\bar{a}_k w_R = \bar{a}_k \cdots \bar{a}_i$ for some $1 < i \leq k$. Then, $\bar{a}_k w_R \bar{a}_{i-1}$ is defined as a string, again rendering $\tau_R \neq \emptyset$ by Corollary~\ref{cor:projective-resolution}; a contradiction. 

Suppose now that $\tau_t = \emptyset$ and $\tau_R \neq \emptyset$. Since $\tau_t = \emptyset$, $w$ ends with a direct substring which has been removed 
by Remark~\ref{rem:observations}\eqref{ends}.
By gentleness, the maximal direct suffix that has been removed is $p b_l \cdots b_1$, where again $p$ is the maximal direct path composable with $b_l$ as a string. Now since $\tau_R = \sigma_R \sigma'_R$ is a strictly longer inverse homotopy letter than $\sigma_R$, it follows that $\widetilde{w}$ is a substring of $w$, where $\sigma'_R = \bar{a}' \sigma''_R$ for some $a' \in Q_1$ and $\sigma''_R$ is possibly trivial.

Finally, if $\tau_t, \tau_R \neq \emptyset$, then arguing as above shows that $\widetilde{w}$ is a substring of $w$.
\end{proof}

\begin{lemma} \label{lem:quasi-right-minus-1}
Suppose the quasi-graph map $\phi \colon \Q_\sigma \wiggle \Q_\tau$ has right endpoint in degree $-1$.
\begin{enumerate}[label=(\arabic*)]
\item The compatibly oriented unfolded diagram for $\phi$ has the following form at the right endpoint of $\phi$:
\[
(a) \
\xymatrix@!R=3px{
\ar@{~}[r] & \xydot \ar@{<-}[r]^-{\sigma_s} & x \ar[r]^-{\sigma_R} \ar@{=}[d] & \xydot \ar@{~}[r] & \\
\ar@{~}[r] & \xydot \ar@{<-}[r]_-{\tau_t}     & x \ar[r]_-{\tau_R}                       & \xydot \ar@{~}[r]  &
}
\quad (b) \
\xymatrix@!R=3px{
\ar@{~}[r] & \xydot \ar@{<-}[r]^-{\sigma_s} & x \ar@{<-}[r]^-{\sigma_R} \ar@{=}[d] & \xydot \ar@{~}[r] & \\
\ar@{~}[r] & \xydot \ar@{<-}[r]_-{\tau_t}     & x \ar[r]_-{\tau_R}                                & \xydot \ar@{~}[r]  &
},
\]
where $\tau_t \neq \emptyset$. In case (a), $\sigma_s = \emptyset$ or $\sigma_s = \tau_t' \tau_t$ for some $\tau_t'$ and we require $\tau_R \neq \emptyset$ and $\sigma_R = \emptyset$ or else $\sigma_R = \tau_R \tau_R'$ for some nontrivial $\tau_R'$. In case (b) $\sigma_s = \tau_t' \tau_t$ for some $\tau_t'$ and we require one of $\tau_R \neq \emptyset$ or $\sigma_R \neq \emptyset$ and if both are not empty letters then $\bar{\sigma_R} \tau_R \neq 0$. In both cases $\tau_t'$ may be trivial.
\item Write $\tau_t = \bar{d}_q \cdots \bar{d}_1$ and $\tau_R = c_p \cdots c_1$ for $k,l \geq 1$ and $c_i,d_j \in Q_1$. Then
\begin{enumerate}[label=(\roman*)]
\item $v$ has a substring of the form
\[
\widetilde{v} =
\left\{
\begin{array}{ll}
\bar{d}_q \cdots \bar{d}_2                             & \text{if $\sigma_s$ is incident with $\inverse(v)$,} \\
c_{p-1} \cdots c_1 c                                      & \text{for some $c \in Q_1$ if $\sigma_s = \emptyset$ and $\sigma_R$ is incident with $\direct(v)$,} \\
\bar{d}_q \cdots \bar{d}_1 c_p \cdots c_1 c  & \text{for some $c \in Q_1$ if $\sigma_s \neq \emptyset$ and $\sigma_R \neq \emptyset$ is direct;}
\end{array}
\right.
\]
\item $w$ has a substring of the form
\[
\widetilde{w} =
\left\{
\begin{array}{ll}
\bar{d}_q \cdots \bar{d}_1 c_p \cdots c_1 & \text{$\tau_R \neq \emptyset$,} \\
\bar{d}_q \cdots \bar{d}_2                         & \text{otherwise.}
\end{array}
\right.
\]
\end{enumerate}
\end{enumerate}
\end{lemma}

\begin{proof}
$(1)$ There are three possible orientations for the homotopy strings $\sigma$ and $\tau$ with right endpoint in degree $-1$, where in the following diagrams $x$ sits in degree $-1$:
\[
(I) \
\xymatrix{
\xydot & \ar[l] x \ar[r] & \xydot
}
\quad (II) \
\xymatrix{
\xydot & \ar[l] x & \ar[l] \xydot
}
\quad (III) \
\xymatrix{
\xydot \ar[r] & x \ar[r] & \xydot
}.
\]
Note that the fourth possible orientation does not occur because then the corresponding string or band complex would have maximal cohomological degree $-1$, contradicting Remark~\ref{rem:observations}\eqref{degrees}.  
One can check that if $\sigma$ has orientation (I) then so does $\tau$: the other orientations produce graph map endpoint conditions (and hence null-homotopies; see \cite[Rem. 4.9]{ALP}), this gives case $(a)$ above. Observe that in case (a), $\tau_t \neq \emptyset$ and $\tau_R \neq \emptyset$, for otherwise we would have a graph map endpoint condition.

If $\sigma$ has orientation (II) then $\tau$ cannot have orientation (III) because this again gives a graph map endpoint condition. If $\tau$ has orientation (II) then we may assume $\tau_R \neq \emptyset$ (the case $\tau_R = \emptyset$ is trivial can be considered as a subcase of $\tau$ having orientation (I)), in which case $\length(\tau_R) \geq 1$. However, for degree reasons, it must be incident with $\inverse(w)$ and hence $\length(\tau_R) = 1$. Therefore $\tau$ cannot have orientation (II). This gives us case $(b)$. 
Note in this case that since $x$ sits in degree $-1$, $\sigma_s \neq \emptyset$ by Remark~\ref{rem:observations}\eqref{ends}; as above, $\tau_s \neq \emptyset$ otherwise we have a graph map endpoint condition.

When $\sigma$ has orientation (III), the unfolded diagrams are those for the dual left endpoint conditions, and 
can be properly stated in the dual of this lemma. 

$(2)(i)$  First observe that, in both cases, either $\sigma_s \neq \emptyset$ or $\sigma_R \neq \emptyset$ (or both) for degree reasons: if both were empty homotopy letters, $\Q_\sigma$ would be a stalk complex concentrated in degree $-1$, contradicting Remark~\ref{rem:observations}\eqref{degrees}.

Suppose we are in case $(a)$ of part $(1)$.
Suppose $\sigma_s = \emptyset$ but $\sigma_R \neq \emptyset$. Then $\sigma_R = \tau_R \tau_R'$ for some nontrivial $\tau_R'$ by the (RQ1) endpoint condition.  
By Remark~\ref{rem:observations}\eqref{ends}, $\sigma$ cannot end with a direct homotopy letter unless it is incident with $\direct(v)$.
Let $c \in Q_1$ be the final arrow of the homotopy (sub)letter $\tau_R'$. Then since $\sigma_R$ is incident with $\direct(v)$, we have that $\widetilde{v} = c_{p-1} \cdots c_1 c$ is a substring of $v$.

If $\sigma_s \neq \emptyset$ but $\sigma_R = \emptyset$, then Remark~\ref{rem:observations}\eqref{starts} shows that $\sigma_s$ is incident with $\inverse(v)$, giving $\widetilde{v} = \bar{d}_q \cdots \bar{d}_2$ as a substring of $v$.

If $\sigma_1, \sigma_R \neq \emptyset$, then neither is incident with $\direct(v)$ or $\inverse(v)$, in which case $\widetilde{v} = \bar{d}_q \cdots \bar{d}_1 c_p \cdots c_1 c$, where $c \in Q_1$ is as above, is a substring of $v$.

Now suppose we are in case $(b)$ of part $(1)$. 
If $\sigma_R = \emptyset$ then using Remark~\ref{rem:observations}\eqref{starts} again we have $\sigma_s$ is incident with $\inverse(v)$ and $\widetilde{v} = \bar{d}_q \cdots \bar{d}_2$ is a substring of $v$. If $\sigma_R \neq \emptyset$, then by Remark~\ref{rem:observations}\eqref{lengths}, $\length(\sigma_R) = 1$ and $\sigma_R$ is incident with $\inverse(v)$, in which case $\widetilde{v} = \bar{d}_q \cdots \bar{d}_2$ is again a substring of $v$.

$(2)(ii)$ Suppose we are in case $(a)$ of part $(1)$. Since $\tau_t , \tau_R \neq \emptyset$, the homotopy substring $\tau_1 \tau_R$ cannot be incident with $\direct(w)$ nor $\inverse(w)$ for degree reasons. Thus, $\widetilde{w} = \bar{d}_q \cdots \bar{d}_1 c_p \cdots c_1$ is a substring of $w$.

Finally, suppose we are in case $(b)$ of $(1)$. 
If $\tau_R = \emptyset$, then Remark~\ref{rem:observations}\eqref{starts} shows that $\tau_1$ is incident with $\inverse(w)$, giving $\widetilde{w} = \bar{d}_q \cdots \bar{d}_2$ as a substring of $w$. If $\tau_R \neq \emptyset$, then as above the homotopy substring $\tau_1 \tau_R$ cannot be incident with $\direct(w)$ nor $\inverse(w)$. Thus, $\widetilde{w} = \bar{d}_q \cdots \bar{d}_1 c_p \cdots c_1$ is a substring of $w$.
\end{proof}

Lemmas~\ref{lem:quasi-right-0} and \ref{lem:quasi-right-minus-1} admit obvious duals for the left endpoints of quasi-graph maps.

Now applying the graphical calculus for the mapping cones of the homotopy set determined by a quasi-graph map \cite[Prop. 5.2]{CPS} and \cite[\S 2]{Addendum} determines the middle term of the extension $\Q_\tau \to \E \to \Q_\sigma \to \Sigma \Q_\tau$ in $\KminusL$. Lemmas~\ref{lem:quasi-right-0} and \ref{lem:quasi-right-minus-1} and their duals, Theorem~\ref{thm:cohomology}, together with a calculation as in the proof of Lemma~\ref{Graph map in antipaths} allows us to take cohomology to determine the extension  $0 \to M(w) \to H^0(\E) \to M(v) \to 0$. 
We summarise this computation in the next proposition. 

\begin{proposition} \label{prop:quasi-overlaps}
Suppose $\phi \colon \Q_\sigma \wiggle \Q_\tau$ is a quasi-graph map with the following unfolded diagram, with $t \geq 0$ and, when $t=0$ we mean a quasi-graph map supported in precisely one degree and we replace $\rho_1$ by $\sigma_L$ and $\tau_L$ as appropriate.
\[
\xymatrix@!R=5px{
\text{deg:}                                           &                               &                                       & h'                                              &                                                        &                                  &                                                         &  h                                                                                           &                                             & \\
\Q_\sigma \colon \ar@{~>}[d]_-{\phi} & \ar@{~}[r]^-{\beta} & \xydot \arr^-{\sigma_L}  & \xydot \arr^-{\rho_t} \ar@{=}[d] & \xydot \arr^-{\rho_{t-1}} \ar@{=}[d] & \cdots \arr^-{\rho_2}  & \xydot \arr^-{\rho_1} \ar@{=}[d] & \xydot \arr^-{\sigma_R} \ar@{=}[d]  & \xydot \ar@{~}[r]^-{\alpha}  & \\
\Q_\tau \colon                                    & \ar@{~}[r]_-{\delta} & \xydot \arr_-{\tau_L}      & \xydot \arr_-{\rho_t}                  & \xydot \arr_-{\rho_{t-1}}                  & \cdots \arr_-{\rho_2}  & \xydot \arr_-{\rho_1}                  & \xydot \arr_-{\tau_R}             & \xydot \ar@{~}[r]_-{\gamma} & 
}
\]
Let $\f \colon \Q_\sigma \to \Sigma \Q_\tau$ be any representative of $\phi$, then $\Phi(\f)$ is an overlap extension with overlap $m = m_L \rho_{t-1} \cdots \rho_2 m_R$, where
\begin{align*}
m_R & = 
\left\{
\begin{array}{ll}
\widetilde{\rho_1} \bar{a}_k \cdots \bar{a}_2   & \text{if $h=0$ and $\sigma_R = \bar{a}_k \cdots \bar{a}_2$ is incident with $\inverse(v)$;} \\
\widetilde{\rho_1} \bar{a}_k \cdots \bar{a}_1   & \text{if $h=0$ and $\sigma_R = \bar{a}_k \cdots \bar{a}_1$ is not incident with $\inverse(v)$;} \\
\bar{d}_q \cdots \bar{d}_2                                & \text{if $h=-1$ and $\rho_1 \neq \emptyset$ is incident with $\inverse(v)$;} \\
\widetilde{\rho_1} c_p \cdots c_2                     & \text{if $h=-1$, $\rho_1 \neq \emptyset$ and $\sigma_R = c_p \cdots c_1$ with $p > 0$;}
\end{array}
\right. \\
m & = c_{p-1} \cdots c_1 \hspace{10mm} \text{if $\rho_1 = \emptyset$ and $\sigma_R = c_p \cdots c_1$ is incident with $\direct(v)$,}
\end{align*}
where $a_i, d_i, c_i \in Q_1$, $m_L$ is defined dually, and 
\[
\widetilde{\rho_1} = 
\begin{cases} 
\rho_1                                                     & \text{if $t \geq 1$,} \\
\text{the last homotopy letter of $m_L$} & \text{if $t= 0$.}
\end{cases}
\]
\end{proposition}

\begin{remark}
The analysis concerning quasi-graph maps above leading to Proposition~\ref{prop:quasi-overlaps} only concerns the endpoints of the overlap defining a quasi-graph map in the unfolded diagram. As such, the length of the overlap is not relevant for the argument. In particular, this means that when one (or both) of $v$ or $w$ is a band, and thus $\sigma$ or $\tau$ is a homotopy band, we are able to get quasi-graph maps whose overlaps are longer than at least one of the bands, but this does not affect the computation carried out in Proposition~\ref{prop:quasi-overlaps}.
\end{remark}

\subsection{Singleton maps} \label{sec:singleton}

As before, throughout this subsection $\sigma = \pi(v)$ and $\tau = \pi(w)$ for some strings or bands $v$ and $w$.
We now examine the kinds of extensions that arise from singleton (single and double) maps $\f \colon \Q_\sigma \to \Sigma \Q_\tau$. We first note that singleton double maps never occur as morphisms between projective resolutions of modules. 

\begin{lemma} \label{lem:no-singleton-doubles}
There are no singleton double maps $\f \colon \Q_\sigma \to \Sigma \Q_\tau$.
\end{lemma}

\begin{proof}
By definition, the unfolded diagram of a singleton double map has the form
\[
\xymatrix@!R=5px{
\Q_\sigma :          & \ar@{~}[r]^-{\beta}  & \xydot \arr^-{\sigma_L} & \xydot \ar[r]^-{\sigma_C = f_L f'} \ar[d]_-{f_L} & \xydot \arr^-{\sigma_R} \ar[d]^-{f_R} & \xydot \ar@{~}[r]^-{\alpha} &  \\
\Sigma \Q_\tau :  & \ar@{~}[r]_-{\delta} & \xydot \arr_-{\tau_L}     & \xydot \ar[r]_-{\tau_C = f' f_R}                         & \xydot \arr_-{\tau_R}                          & \xydot \ar@{~}[r]_-{\gamma} &
}
\]
where $f_L$, $f'$ and $f_R$ are nontrivial. By Remark~\ref{rem:observations}\eqref{lengths}, $\length(\sigma_C) > 1$ and $\length(\tau_C) > 1$. In particular, since $\sigma$ is a homotopy string or band corresponding to a projective resolution, $\sigma_C$ is a homotopy letter occurring between degrees $-1$ and $0$. On the other hand, $\tau$ is also a homotopy string or band corresponding to a projective resolution, but $\Sigma \Q_\tau$ has been shifted, whence $\tau_C$ must be a homotopy letter occurring between degrees $-2$ and $-1$. Hence there are no such maps.
\end{proof}

Recall the notation and unfolded diagram for a singleton single map $\f \colon \Q_\sigma \to \Sigma \Q_\tau$ from Section~\ref{sec:single}\eqref{singleton-single}. Throughout this section, whenever $\sigma_R \neq \emptyset$ or $\tau_R \neq \emptyset$ in \eqref{singleton-single} then since $f_R$ and $f_L$ are direct strings, we can assume, without loss of generality, that $f_R \in Q_1$ and $f_L \in Q_1$, respectively.

\begin{lemma} \label{lem:degree-minus-one}
Suppose $\f \colon \Q_\sigma \to \Sigma \Q_\tau$ as a singleton single map with single component $f = f_n \cdots f_1$. Then the component $f$ occurs in cohomological degree $-1$.
\end{lemma}

\begin{proof}
Suppose $\f$ is supported in cohomological degree $d$. Since $\Q_\tau$ is a projective resolution, $\Sigma \Q_\tau$ attains its maximal degree in degree $-1$, thus $d \leq -1$.
By Remark~\ref{rem:observations}\eqref{lengths}, if in \eqref{singleton-single} either $\sigma_R \neq \emptyset$ or $\tau_R \neq \emptyset$ then $d = -1$. So assume $\sigma_R, \tau_R = \emptyset$ and $d < -1$. 
By Corollary~\ref{cor:projective-resolution}, since $\tau_L$ is the endpoint of a homotopy string occurring in degree $d$ it must be inverse (otherwise there would be nontrivial cohomology in degree $d$, contradicting the fact that $\Sigma \Q_\tau$ is a (shifted) projective resolution). Moreover, for degree reasons, $\tau_L$ must be the first homotopy letter of $\inverse(w)$. Writing $\tau_L = \bar{b}_l \cdots \bar{b}_1$ for some $b_i \in Q_1$, $i = 1, \ldots, l$, the definition of single maps gives us that $\bar{b}_1 \bar{f}_1 = 0$. This contradicts the fact that $\inverse(w)$ is the longest inverse antipath incident with $w'$ (see Corollary~\ref{cor:projective-resolution}). Therefore, $d = -1$, as claimed.
\end{proof}

\begin{corollary} \label{cor:tau-is-inverse}
Suppose $\f \colon \Q_\sigma \to \Sigma \Q_\tau$ is a singleton single map. In the unfolded diagram \eqref{singleton-single} in Section~\ref{sec:single}, $\tau_L$ must be a direct homotopy letter or $\tau_L = \emptyset$.
\end{corollary}

\begin{proof}
Since $\Sigma \Q_\tau$ attains its maximal cohomological degree in degree $-1$ and $\f$ is supported in degree $-1$ by Lemma~\ref{lem:degree-minus-one}, $\tau_L$ cannot be inverse.
\end{proof}

Corollary~\ref{cor:tau-is-inverse} allows us to further specialise the setup in Section~\ref{sec:single}\eqref{singleton-single} in the statement of the next proposition.

\begin{proposition}\label{prop:singleton-single-to-extensions}
Suppose $\f \colon \Q_\sigma \to \Sigma \Q_\tau$ is a singleton single map with single component $f = f_n \cdots f_1$. Write $\tau_L = b_l \cdots b_1$ with $b_i \in Q_1$ for $i = 1,\ldots,l$. 
Whenever $\sigma_L$ is an inverse homotopy letter we shall write $\sigma_L = \bar{a}_k \cdots \bar{a}_1$, where $a_i \in Q_1$ for $i = 1, \cdots, k$ with $k \geq 1$.
\begin{enumerate}
\item If $\sigma_R = \emptyset$ then $\tau_R = \emptyset$ and $\sigma_L$ is inverse. 
The corresponding extension $\Phi(\f) \in \Ext^1_{\Lambda}(M(\bar{v}),M(w))$ is an arrow extension given by $a_1$, i.e. $\Phi(\f)$ gives rise to an extension of $M(w)$ by $M(\bar{v})$ with middle term  $ M(u)$ where $u = w a_1 \bar{v}$.
\end{enumerate}
Suppose $\sigma_R \neq \emptyset$. If $\sigma_R$ is not incident with $\direct(v)$ then $\sigma_L$ is inverse and we have:
\begin{enumerate}[resume]
\item If $\tau_R = \emptyset$ then the corresponding extension $\Phi(\f) \in \Ext^1_{\Lambda}(M(\bar{v}),M(w))$ is an overlap extension whose middle term is given by
\begin{align*}
& m = \bar{f}_1 \cdots \bar{f}_{n-1} , \
A = \emptyset , \
B = f_R , \
C = f_n  \text{ and }
D = b_1 \ & &
\text{ if } \sigma_R \text{ is incident with } \direct(v); \\
& m = \bar{f}_1 \cdots \bar{f}_n , \
A = a_1 , \
B = f_R , \
C = \emptyset \text{ and }
D = b_1 & &
\text{ otherwise.}
\end{align*}
\item If $\tau_R \neq \emptyset$ then the corresponding extension $\Phi(\f) \in \Ext^1_{\Lambda}(M(\bar{v}),M(w))$ is an overlap extension which, when $\sigma_R$ is incident with $\direct(v)$, has its middle term given by,
\begin{align*}
& m = \bar{f}_1 \cdots \bar{f}_{n-1} , \quad
A = \emptyset , \quad
B = f_R , \quad
C = f_n \quad \text{and} \quad
D = c_1 & & \text{ if } \tau_L = \emptyset ; \\
& m = \bar{f}_1 \cdots \bar{f}_{n-1} , \quad
A = \emptyset , \quad
B = f_R , \quad
C = f_n \quad \text{and} \quad
D = b_1 & & \text{ if } \tau_L \neq \emptyset ,
\end{align*}
and is an overlap extension which, when $\sigma_R$ is not incident with $\direct(v)$, has its middle term given by, 
\begin{align*}
& m = \bar{f}_1 \cdots \bar{f}_n , \quad
A = a_1 , \quad
B = f_R , \quad
C = \emptyset \quad \text{and} \quad
D = c_1 & & \text{ if } \tau_L = \emptyset ; \\
& m = \bar{f}_1 \cdots \bar{f}_n , \quad
A = a_1 , \quad
B = f_R , \quad
C = \emptyset \quad \text{and} \quad
D = b_1 & & \text{ if } \tau_L \neq \emptyset . 
\end{align*}
\end{enumerate}
\end{proposition}

In case $(3)$ of the proof below, we do an example of a computation as in Lemma~\ref{Graph map in antipaths} for an overlap extension. An example computation for an arrow extension is done in the proof of Lemma~\ref{Graph map in antipaths}, and we refer the reader to Figure~\ref{fig:cohomology-map} for a schematic of such a computation.

\begin{proof}
$(1)$ First note that $\sigma_L$ is inverse, since if it were direct or empty $\Q_\sigma$ would have nontrivial cohomology in degree $-1$, contradicting the fact that it is a projective resolution. 
Therefore, $\sigma_L = \bar{a}_k \cdots \bar{a}_1$ with $a_i \in Q_1$ for $i = 1, \ldots, k$ for some $k \geq 1$. 
Moreover, $\sigma_L$ is the start of $\inverse(v)$, for otherwise $\sigma$ would start in degree $0$ after the removal of a maximal inverse prefix. It follows that $v$ starts with the inverse substring $\bar{a}_k \cdots \bar{a}_2$, whence $\bar{v}$ ends with the direct substring $a_2 \cdots a_k$.

Consider the local subquiver of $Q$, where, without loss of generality, we assume $f_L \in Q_1$,
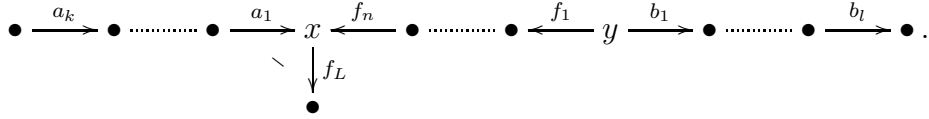
\[
\xymatrix@!R=5px{                              
\xydot \ar[r]^-{a_k} & \xydot \ar@{.}[r] & \xydot \ar[r]^-{a_1}_{}="a"  & x \ar[d]^-{f_L}_{}="f"      & \ar[l]_-{f_n} \xydot & \ar@{.}[l] \xydot & \ar[l]_-{f_1} y \ar[r]^-{b_1} & \xydot \ar@{.}[r] & \xydot \ar[r]^-{b_l} & \xydot \, . \\
                              &                           &                                            & \xydot                      &                               &                           &                                          &                           &                               &
\ar @{--}@/_/ "a";"f"}
\]
If $\tau_R \neq \emptyset$ then $\bar{a}_1 \bar{f_L} = 0$, contradicting the fact that $\sigma_L$ is the start of $\inverse(v)$. Thus, $\tau_R = \emptyset$. 

Since $f$ is not a subletter of $\tau_L$ or vice versa we must have $f_1 \neq b_1$ and $b_1 \bar{f}_1$ is defined as a string. This means that a maximal inverse prefix, whose last (inverse) arrow is $\bar{f}_1$, has been removed from $w$ in the computation of $\tau = \pi(w)$ for otherwise $\tau_R \neq \emptyset$. We claim that $\bar{f}$ is precisely the maximal inverse prefix that has been removed. Clearly, the maximal inverse prefix cannot be a proper substring of $\bar{f}$ for the computation of $\tau = \pi(w)$ in Corollary~\ref{cor:projective-resolution} would require us to compose this with $w$ giving $\tau_R \neq \emptyset$. However, if $\bar{f}$ were a proper substring of the maximal inverse prefix then there would be an arrow $f_{n+1} \in Q_1$ such that $\bar{a}_1 \bar{f}_{n+1} = 0$ giving us a contradiction as above. Therefore, $w$ starts with the substring $\tau_L \bar{f}$.  Applying \cite[Thm. 3.2]{CPS}, Theorem~\ref{thm:cohomology} and a computation as in Lemma~\ref{Graph map in antipaths} shows that that $\Phi(\f)$ gives an arrow extension corresponding to the arrow $a_1$ with middle term $M(u)$, where $u = w a_1 \bar{v}$.

Suppose that $\sigma_R \neq \emptyset$. If $\sigma_R$ is not incident with $\direct(v)$ then by Remark~\ref{rem:observations}\eqref{starts}, $\sigma_L \neq \emptyset$ and is inverse and we write $\sigma_L = \bar{a}_k \cdots \bar{a}_1$, where $a_i \in Q_1$ for $i = 1,\ldots,k$ with $k \geq 1$.

$(2)$ Suppose that $\tau_R = \emptyset$. First observe that, by Corollary~\ref{cor:projective-resolution}, $w$ has a substring of the form $b_{l-1} \cdots b_1 \bar{f}_1 \cdots \bar{f}_n$. 
If $\sigma_R$ is incident with $\direct(v)$ (in which case so is $\sigma_L$ regardless of whether it is empty), then $v$ ends with a substring $f_{n-1} \cdots f_1 f_R$, i.e. $\bar{v}$ starts with a substring $\bar{f_R} \bar{f}_1 \cdots \bar{f}_{n-1}$, by Corollary~\ref{cor:projective-resolution}. Applying \cite[Thm. 3.2]{CPS}, taking cohomology using Theorem~\ref{thm:cohomology} and a calculation as in Lemma~\ref{Graph map in antipaths} then gives us an overlap extension between $M(w)$ and $M(\bar{v})$:
\[
m = \bar{f}_1 \cdots \bar{f}_{n-1} , \quad
A = \emptyset , \quad
B = f_R , \quad
C = f_n \quad \text{and} \quad
D = b_1 .
\]
Now suppose $\sigma_R$ is not incident with $\direct(v)$. By Corollary~\ref{cor:projective-resolution}  $v$ has a substring $\sigma_L f_n \cdots f_1 f_R$, i.e. $\bar{v}$ has a substring $\bar{f}_R \bar{f}_1 \cdots \bar{f}_n \bar{\sigma_L}$. Again applying \cite[Thm. 3.2]{CPS}, Theorem~\ref{thm:cohomology} and a calculation as in Lemma~\ref{Graph map in antipaths} gives us the following overlap extension between $M(w)$ and $M(\bar(v))$:
\[
m = \bar{f}_1 \cdots \bar{f}_n , \quad
A = a_1 , \quad
B = f_R , \quad
C = \emptyset \quad \text{and} \quad
D = b_1.
\]

$(3)$ Suppose that $\tau_R \neq \emptyset$. First we assume $\sigma_R$ is incident with $\direct(v)$, whence $\bar{v}$ starts with the substring $\bar{f_R} \bar{f}_1 \cdots \bar{f}_{n-1}$ as above.
If $\tau_L = \emptyset$, then, by Corollary~\ref{cor:projective-resolution}, $w$ ends with a substring $c_t \cdots c_1 \bar{f}_1 \cdots \bar{f}_n \bar{f_L}$, where $c_i \in Q_1$ for $i = 1, \cdots, t$ and $t \geq 0$. 
In this case the application of \cite[Thm. 3.2]{CPS}, Theorem~\ref{thm:cohomology} and a calculation as in Lemma~\ref{Graph map in antipaths}, which we sketch below, gives us the following overlap extension between $M(w)$ and $M(\bar{v})$:
\[
m = \bar{f}_1 \cdots \bar{f}_{n-1} , \quad
A = \emptyset , \quad
B = f_R , \quad
C = f_n \quad \text{and} \quad
D = c_1.
\]

We now sketch the calculation as in Lemma~\ref{Graph map in antipaths} for this case. The unfolded diagrams of the morphisms occurring in the mapping cone triangle of \cite[Thm. 3.2]{CPS} are:
\[
\xymatrix@!R=5px{
\Q_\tau \ar[d]_-{\g_1} \colon &                                              & \ar@{~}[r]^-{\bar{\gamma}}              & \xydot \ar[d]_-{f_L} \ar[r]^-{f_L f} & \xydot \ar@{=}[d]   & \\
\E_1 \ar[d]_-{\h_1} \colon     & \ar@{=}[d] \ar@{~}[r]^{\beta} & \xydot \ar@{=}[d] \ar[r]^-{\sigma_L} & \xydot \ar@{=}[d] \ar[r]^{f}          & \xydot \ar[d]^-{f_R} & \\
\Q_\sigma \colon                  & \ar@{~}[r]^-{\beta}                 & \xydot \ar[r]_-{\sigma_L}                  & \xydot \ar[r]_-{f f_R}                    & \xydot \ar@{~}[r]_-{\alpha}    &
}
\quad \text{and} \quad
\xymatrix@!R=5px{
\Q_\tau \ar[d]_-{\g_2} \colon &                                & \ar@{=}[d] \ar@{~}[r]^-{\bar{\gamma}} & \xydot \ar@{=}[d] \ar[r]^-{f_L f}         & \xydot \ar[d]^-{f_R}                                & \\
\E_2 \ar[d]_-{\h_2} \colon     &                                & \xydot \ar@{~}[r]^-{\bar{\gamma}}        & \xydot \ar[d]_-{f_L} \ar[r]^{f_L f f_R}  & \xydot \ar@{=}[d]  \ar@{~}[r]^-{\alpha} & \ar@{=}[d] \\
\Q_\sigma \colon                  & \ar@{~}[r]^-{\beta}  & \xydot \ar[r]_-{\sigma_L}                       & \xydot \ar[r]_-{f f_R}                           & \xydot \ar@{~}[r]_-{\alpha}                  &
}.
\]
Figure \ref{fig:case-3-cohomology} below shows the calculation of the induced maps in cohomology: it is clear that they are the canonical maps in the resulting overlap extension.
\begin{figure}
\scalebox{0.6}{
\begin{tikzpicture}
\coordinate[circle,inner sep=7,label=center:{$s(f)$}] (a1);

\coordinate[circle,below right=1cm of a1,label=center:{$\bullet$}] (a1');
\coordinate[circle,below right=1cm of a1',label=center:{$\bullet$}] (a1'');
\coordinate[circle,below right=1cm of a1'',label=center:{$\bullet$}] (a1''');

\coordinate[circle,below left=1cm of a1,label=center:{$\bullet$}] (a2);
\coordinate[circle,below left=1cm of a2,label=center:{$\bullet$}] (a3);
\coordinate[circle,below left=1cm of a3,label=center:{$\bullet$}] (a4);

\coordinate[circle,below left=1cm of a4,label=center:{$\bullet$}] (a5);
\coordinate[circle,below left=1cm of a5,label=center:{$\bullet$}] (a6);
\coordinate[circle,below left=1cm of a6,label=center:{$\quad$}] (a7);
\coordinate[circle, left=1cm of a6,label=center:{$\quad$}] (a8);

\draw[->] (a1) --node [anchor=south west,scale=.9]{$\overline{c}_1$} (a1');
\draw[dotted] (a1') -- (a1'');
\draw[->] (a1'') --node [anchor=south west,scale=.9]{$\overline{c}_t$} (a1''');

\draw[->] (a1) --node [anchor=south,scale=.9]{$f_1$} (a2);
\draw[dotted] (a2) -- (a3);
\draw[->] (a3) --node [anchor=south east,scale=.9]{$f_{n-1}$} (a4);
\draw[->] (a4) --node [anchor=south east,scale=.9]{$f_{n}$} (a5);
\draw[->] (a5) --node [anchor=south east,scale=.9]{$f_{L}$} (a6);

\draw[->,decorate,decoration={snake,amplitude=.4mm,segment length=2mm}] 
(a6)-- node [anchor=south west,scale=.9]{} (a7);
\draw[decorate,decoration={snake,amplitude=.4mm,segment length=2mm}] 
(a6)-- node [anchor=south,scale=.9]{$\overline{\gamma}$} (a8);

\draw [line width=10pt,opacity=0.20,black,line cap=round,rounded corners] (a1'''.center) -- (a1.center) -- (a2.center) -- (a3.center)--(a4.center)--(a5.center);

\draw [line width=10pt,opacity=0.10,black,line cap=round,rounded corners] (a5.center)--(a6.center)--(a8.center);

\coordinate[circle,inner sep=7,below =3cm of a1,label=center:{$s(f)$}] (b1);
\coordinate[circle,below right=1cm of b1,label=center:{$\bullet$}] (b1');
\coordinate[circle,below right=1cm of b1',label=center:{$\bullet$}] (b1'');
\coordinate[circle,below right=1cm of b1'',label=center:{$\bullet$}] (b1''');

\coordinate[circle,below left=1cm of b1,label=center:{$\bullet$}] (b2);
\coordinate[circle,below left=1cm of b2,label=center:{$\bullet$}] (b3);
\coordinate[circle,below left=1cm of b3,label=center:{$\bullet$}] (b4);

\coordinate[circle,below left=1cm of b4,label=center:{$\bullet$}] (b5);
\coordinate[circle,below left=1cm of b5,label=center:{$\bullet$}] (b6);
\coordinate[circle,below left=1cm of b6,label=center:{$\quad$}] (b7);
\coordinate[circle, left=1cm of b5,label=center:{$\quad$}] (b8);

\draw[->] (b1) --node [anchor=south west,scale=.9]{$\overline{c}_1$} (b1');
\draw[dotted] (b1') -- (b1'');
\draw[->] (b1'') --node [anchor=south west,scale=.9]{$\overline{c}_t$} (b1''');

\draw[->] (b1) --node [anchor=south,scale=.9]{$f_1$} (b2);
\draw[dotted] (b2) -- (b3);
\draw[->] (b3) --node [anchor=south east,scale=.9]{$f_{n-1}$} (b4);
\draw[->] (b4) --node [anchor=south east,scale=.9]{$f_{n}$} (b5);
\draw[->] (b5) --node [anchor=north west,scale=.9]{$f_{L}$} (b6);

\draw[->,decorate,decoration={snake,amplitude=.4mm,segment length=2mm}] 
(b6)-- node [anchor=south west,scale=.9]{} (b7);
\draw[->,decorate,decoration={snake,amplitude=.4mm,segment length=2mm}] 
(b5)-- node [anchor=south,scale=.9]{$\sigma_L$} (b8);

\coordinate[circle,above right=-.01cm of b5,label=center:{$\quad$}] (b5');
\draw[thick, dotted] (b5') arc (40:190:.3cm);

\draw [line width=10pt,opacity=0.20,black,line cap=round,rounded corners] (b1'''.center) -- (b1.center) -- (b2.center) -- (b3.center)--(b4.center);

\coordinate[circle,inner sep=7,below =5cm of b1,label=center:{$s(f)$}] (c1);
\coordinate[circle,inner sep=8,above right=1cm of c1,label=center:{$s(f_R)$}] (c1');
\coordinate[circle,right=1.4cm of c1',label=center:{$\quad$}] (c1'');

\coordinate[circle,below left=1cm of c1,label=center:{$\bullet$}] (c2);
\coordinate[circle,below left=1cm of c2,label=center:{$\bullet$}] (c3);
\coordinate[circle,below left=1cm of c3,label=center:{$\bullet$}] (c4);

\coordinate[circle,below left=1cm of c4,label=center:{$\bullet$}] (c5);
\coordinate[circle,below left=1cm of c5,label=center:{$\bullet$}] (c6);
\coordinate[circle,below left=1cm of c6,label=center:{$\quad$}] (c7);
\coordinate[circle, left=1cm of c5,label=center:{$\quad$}] (c8);

\draw[->] (c1') --(c1);
\draw[decorate,decoration={snake,amplitude=.4mm,segment length=2mm}] (c1') --node [anchor=south,scale=.9]{$\alpha$} (c1'');

\draw[->] (c1) --node [anchor=south,scale=.9]{$f_1$} (c2);
\draw[dotted] (c2) -- (c3);
\draw[->] (c3) --node [anchor=south east,scale=.9]{$f_{n-1}$} (c4);
\draw[->] (c4) --node [anchor=south east,scale=.9]{$f_{n}$} (c5);
\draw[->] (c5) --node [anchor=north west,scale=.9]{$f_{L}$} (c6);

\draw[->,decorate,decoration={snake,amplitude=.4mm,segment length=2mm}] 
(c6)-- node [anchor=south west,scale=.9]{} (c7);
\draw[->,decorate,decoration={snake,amplitude=.4mm,segment length=2mm}] 
(c5)-- node [anchor=south,scale=.9]{$\sigma_L$} (c8);

\coordinate[circle,above right=-.01cm of c5,label=center:{$\quad$}] (c5');
\draw[thick, dotted] (c5') arc (40:190:.3cm);

\draw [line width=10pt,opacity=0.20,black,line cap=round,rounded corners] (c1.center) -- (c3.center);

\draw [line width=10pt,opacity=0.10,black,line cap=round,rounded corners] (c1.center)--(c1'.center)--(c1''.center);


\draw[transform canvas={xshift=-1.5pt},color=red] (a1''') -- (b1''');
\draw[transform canvas={xshift=1.5pt},color=red] (a1''') -- (b1''');

\draw[transform canvas={xshift=-1.5pt},color=red] (a1) -- (b1);
\draw[transform canvas={xshift=1.5pt},color=red] (a1) -- (b1);

\draw[transform canvas={xshift=-1.5pt},color=red] (b1) -- (c1);
\draw[transform canvas={xshift=1.5pt},color=red] (b1) -- (c1);

\draw[transform canvas={xshift=-1.5pt},color=red] (a4) -- (b4);
\draw[transform canvas={xshift=1.5pt},color=red] (a4) -- (b4);

\draw[transform canvas={xshift=-1.5pt},color=red] (b4) -- (c4);
\draw[transform canvas={xshift=1.5pt},color=red] (b4) -- (c4);

\draw[transform canvas={xshift=-1.5pt},color=red] (a7) -- (b7);
\draw[transform canvas={xshift=1.5pt},color=red] (a7) -- (b7);

\draw[transform canvas={xshift=-1.5pt},color=red] (b7) -- (c7);
\draw[transform canvas={xshift=1.5pt},color=red] (b7) -- (c7);


\end{tikzpicture}
\qquad \qquad
\begin{tikzpicture}
\coordinate[circle,inner sep=4,label=center:{$\bullet$}] (a1);

\coordinate[circle,below right=1cm of a1,label=center:{$\bullet$}] (a1');
\coordinate[circle,below right=1cm of a1',label=center:{$\bullet$}] (a1'');
\coordinate[circle,below right=1cm of a1'',label=center:{$\bullet$}] (a1''');

\coordinate[circle,below left=1cm of a1,label=center:{$\bullet$}] (a2);
\coordinate[circle,below left=1cm of a2,label=center:{$\bullet$}] (a3);
\coordinate[circle,below left=1cm of a3,label=center:{$\bullet$}] (a4);

\coordinate[circle,below left=1cm of a4,label=center:{$\bullet$}] (a5);
\coordinate[circle,below left=1cm of a5,label=center:{$\bullet$}] (a6);
\coordinate[circle,below left=1cm of a6,label=center:{$\quad$}] (a7);
\coordinate[circle, left=1.4cm of a6,label=center:{$\quad$}] (a8);

\draw[->] (a1) --node [anchor=south west,scale=.9]{$\overline{c}_1$} (a1');
\draw[dotted] (a1') -- (a1'');
\draw[->] (a1'') --node [anchor=south west,scale=.9]{$\overline{c}_t$} (a1''');

\draw[->] (a1) --node [anchor=south,scale=.9]{$f_1$} (a2);
\draw[dotted] (a2) -- (a3);
\draw[->] (a3) --node [anchor=south east,scale=.9]{$f_{n-1}$} (a4);
\draw[->] (a4) --node [anchor=south east,scale=.9]{$f_{n}$} (a5);
\draw[->] (a5) --node [anchor=south east,scale=.9]{$f_{L}$} (a6);

\draw[->,decorate,decoration={snake,amplitude=.4mm,segment length=2mm}] 
(a6)-- node [anchor=south west,scale=.9]{} (a7);
\draw[decorate,decoration={snake,amplitude=.4mm,segment length=2mm}] 
(a6)-- node [anchor=south,scale=.9]{$\overline{\gamma}$} (a8);

\draw [line width=10pt,opacity=0.20,black,line cap=round,rounded corners] (a1'''.center) -- (a1.center) -- (a2.center) -- (a3.center)--(a4.center)--(a5.center)--(a6.center)--(a8.center);

\coordinate[circle,inner sep=4,below =5cm of a1,label=center:{$\bullet$}] (b1);
\coordinate[circle,above right=1cm of b1,label=center:{$\bullet$}] (b1');
\coordinate[circle,right=1.4cm of b1',label=center:{$\quad$}] (b1'');

\coordinate[circle,below left=1cm of b1,label=center:{$\bullet$}] (b2);
\coordinate[circle,below left=1cm of b2,label=center:{$\bullet$}] (b3);
\coordinate[circle,below left=1cm of b3,label=center:{$\bullet$}] (b4);

\coordinate[circle,below left=1cm of b4,label=center:{$\bullet$}] (b5);
\coordinate[circle,below left=1cm of b5,label=center:{$\bullet$}] (b6);
\coordinate[circle,below left=1cm of b6,label=center:{$\quad$}] (b7);
\coordinate[circle, left=1.4cm of b6,label=center:{$\quad$}] (b8);

\draw[->] (b1') --node [anchor=south,scale=.9]{$f_R$} (b1);
\draw[decorate,decoration={snake,amplitude=.4mm,segment length=2mm}] (b1') -- node [anchor=south,scale=.9]{$\alpha$}(b1'');

\draw[->] (b1) --node [anchor=south,scale=.9]{$f_1$} (b2);
\draw[dotted] (b2) -- node [anchor=south,scale=.9]{$\sigma_L$}(b3);
\draw[->] (b3) --node [anchor=south east,scale=.9]{$f_{n-1}$} (b4);
\draw[->] (b4) --node [anchor=south east,scale=.9]{$f_{n}$} (b5);
\draw[->] (b5) --node [anchor=south east,scale=.9]{$f_{L}$} (b6);

\draw[->,decorate,decoration={snake,amplitude=.4mm,segment length=2mm}] 
(b6)-- node [anchor=south west,scale=.9]{} (b7);
\draw[decorate,decoration={snake,amplitude=.4mm,segment length=2mm}] 
(b6)-- node [anchor=south east,scale=.9]{$\overline{\gamma}$} (b8);

\draw [line width=10pt,opacity=0.20,black,line cap=round,rounded corners] (b1'.center) -- (b6.center)--(b8.center);

\draw [line width=10pt,opacity=0.10,black,line cap=round,rounded corners] (b1'.center)--(b1''.center);

\coordinate[circle,inner sep=7,below =5cm of b1,label=center:{$\bullet$}] (c1);
\coordinate[circle,inner sep=4,above right=1cm of c1,label=center:{$\bullet$}] (c1');
\coordinate[circle,right=1.4cm of c1',label=center:{$\quad$}] (c1'');

\coordinate[circle,below left=1cm of c1,label=center:{$\bullet$}] (c2);
\coordinate[circle,below left=1cm of c2,label=center:{$\bullet$}] (c3);
\coordinate[circle,below left=1cm of c3,label=center:{$\bullet$}] (c4);

\coordinate[circle,below left=1cm of c4,label=center:{$\bullet$}] (c5);
\coordinate[circle,below left=1cm of c5,label=center:{$\bullet$}] (c6);
\coordinate[circle,below left=1cm of c6,label=center:{$\quad$}] (c7);
\coordinate[circle, left=1.4cm of c6,label=center:{$\quad$}] (c8);

\draw[->] (c1') --node [anchor=south,scale=.9]{$f_R$} (c1);
\draw[decorate,decoration={snake,amplitude=.4mm,segment length=2mm}] (c1') --node [anchor=south,scale=.9]{$\alpha$} (c1'');

\draw[->] (c1) --node [anchor=south,scale=.9]{$f_1$} (c2);
\draw[dotted] (c2) -- (c3);
\draw[->] (c3) --node [anchor=south east,scale=.9]{$f_{n-1}$} (c4);
\draw[->] (c4) --node [anchor=south east,scale=.9]{$f_{n}$} (c5);
\draw[->] (c5) --node [anchor=south east,scale=.9]{$f_{L}$} (c6);

\draw[->,decorate,decoration={snake,amplitude=.4mm,segment length=2mm}] 
(c6)-- node [anchor=south west,scale=.9]{} (c7);
\draw[->,decorate,decoration={snake,amplitude=.4mm,segment length=2mm}] 
(c6)-- node [anchor=south east,scale=.9]{$\sigma_L$} (c8);

\coordinate[circle,above right=-.01cm of c6,label=center:{$\quad$}] (c6');
\draw[thick, dotted] (c6') arc (40:190:.3cm);

\draw [line width=10pt,opacity=0.20,black,line cap=round,rounded corners] (c1'.center) -- (c5.center);

\draw [line width=10pt,opacity=0.10,black,line cap=round,rounded corners] (c1'.center)--(c1''.center);


\draw[transform canvas={xshift=-1.5pt},color=red] (a1) -- (b1);
\draw[transform canvas={xshift=1.5pt},color=red] (a1) -- (b1);

\draw[transform canvas={xshift=-1.5pt},color=red] (b1) -- (c1);
\draw[transform canvas={xshift=1.5pt},color=red] (b1) -- (c1);

\draw[transform canvas={xshift=-1.5pt},color=red] (b1') -- (c1');
\draw[transform canvas={xshift=1.5pt},color=red] (b1') -- (c1');

\draw[transform canvas={xshift=-1.5pt},color=red] (a4) -- (b4);
\draw[transform canvas={xshift=1.5pt},color=red] (a4) -- (b4);

\draw[transform canvas={xshift=-1.5pt},color=red] (b4) -- (c4);
\draw[transform canvas={xshift=1.5pt},color=red] (b4) -- (c4);

\draw[transform canvas={xshift=-1.5pt},color=red] (b5) -- (c5);
\draw[transform canvas={xshift=1.5pt},color=red] (b5) -- (c5);

\draw[transform canvas={xshift=-1.5pt},color=red] (a7) -- (b7);
\draw[transform canvas={xshift=1.5pt},color=red] (a7) -- (b7);

\draw[transform canvas={xshift=-1.5pt},color=red] (b7) -- (c7);
\draw[transform canvas={xshift=1.5pt},color=red] (b7) -- (c7);


\end{tikzpicture}}
\caption{The diagram of the left shows the calculation of $\g_1 \colon \Q_\tau \to \E_1$ and $\h_1 \colon \E_1 \to \Q_\sigma$, that on the right shows the calculation of $\g_2 \colon \Q_\tau \to \E_2$ and $\h_2 \colon \E_2 \to \Q_\sigma$ in the case (3) of the proof of Proposition~\ref{prop:singleton-single-to-extensions}.} \label{fig:case-3-cohomology} 
\end{figure}
If $\tau_L \neq \emptyset$, then $w$ has a substring $b_{l-1} \cdots b_1 \bar{f}_1 \cdots \bar{f}_n \bar{f_L}$. Applying \cite[Thm. 3.2]{CPS} and Theorem~\ref{thm:cohomology} gives us the following overlap extension between $M(w)$ and $M(\bar{v})$:
\[
m = \bar{f}_1 \cdots \bar{f}_{n-1} , \quad
A = \emptyset , \quad
B = f_R , \quad
C = f_n \quad \text{and} \quad
D = b_1.
\]
Now assume that $\sigma_R$ is not incident with $\direct(v)$, whence $\bar{v}$ has a substring $\bar{f}_R \bar{f}_1 \cdots \bar{f}_n \bar{\sigma_L}$ as above.
Using the calculations of substrings of $w$ for $\tau_L = \emptyset$ and $\tau_L \neq \emptyset$ above respectively, and the application of \cite[Thm. 3.2]{CPS}, Theorem~\ref{thm:cohomology} and a calculation as in Lemma~\ref{Graph map in antipaths} gives us the following overlap extensions between $M(w)$ and $M(\bar{v})$:
\begin{align*}
& m = \bar{f}_1 \cdots \bar{f}_n , \quad
A = a_1 , \quad
B = f_R , \quad
C = \emptyset \quad \text{and} \quad
D = c_1 & & \text{ if } \tau_L = \emptyset ; \\
& m = \bar{f}_1 \cdots \bar{f}_n , \quad
A = a_1 , \quad
B = f_R , \quad
C = \emptyset \quad \text{and} \quad
D = b_1 & & \text{ if } \tau_L \neq \emptyset . \qedhere
\end{align*}
\end{proof}

\section{Surjectivity of $\Phi$ onto overlap and arrow extensions}\label{sec:surjective}

In this section, we use the combinatorics of an overlap or arrow extension to show that the isomorphism $\Phi \colon \Hom_{\KminusL}(\Q_{\pi(v)}, \Sigma \Q_{\pi(w)}) \to \Ext^1_\Lambda(M(v),M(w))$ restricts to a surjection,
\[
\Phi \colon \bigg\{ \, \parbox{0.34\textwidth}{standard basis elements of $\Hom_{\KminusL}(\Q_{\pi(v)}, \Sigma \Q_{\pi(w)})$} \, \bigg\} \twoheadrightarrow \bigg\{ \, \parbox{0.30\textwidth}{overlap and arrow extensions $\eta \in \Ext^1_\Lambda (M(v),M(w))$} \, \bigg\}.
\]

\subsection{Overlap extensions} 

Throughout this section we shall have the following setup.

\begin{setup} \label{setup:overlap}
Let $v$ and $w$ be strings or bands and $\pi(v)$ and $\pi(w)$ be the corresponding homotopy strings or bands of their projective resolutions.

Suppose $0 \neq \eta \in \Ext^1_\Lambda (M(v),M(w))$ is an overlap extension corresponding to the decompositions $v = v_L \bar{B} m A v_R$ and $w = w_L D m \bar{C} w_R$. We consider $m$ and decompose it into  its homotopy letters $m = \mu_l \cdots \mu_1$ with $l \geq 0$. 
When $l = 0$, $m$ is a trivial string, i.e. $m = 1_x$ for some $x \in Q_0$ and we call it a \emph{trivial overlap}. If $l = 1$, we say $m$ is a \emph{direct} or \emph{inverse overlap}. If $l > 1$, we say that $m$ is a \emph{zigzag overlap}.
\end{setup}

\subsubsection{Zigzag overlaps}

We start with the zigzag overlap case.

\begin{lemma} \label{lem:ext-to-graph-map} 
Suppose in Setup~\ref{setup:overlap}, the string $m$ is a zigzag overlap.
Then the map $f \colon M(w) \to M(v)$ associated with this decomposition induces a graph map $\f \colon \Q_{\pi(w)} \to \Q_{\pi(v)}$ of homotopy string or band complexes such that $H^0 (\f) =f$. 
\end{lemma}

\begin{proof}
We first show that the given decomposition induces a graph map $\Q_{\pi(w)} \to \Q_{\pi(v)}$. It is sufficient only to consider the endpoints of the map, as determined by the decomposition. We consider only the right endpoints; the analysis for left endpoints is analogous.

Before breaking the argument up into a case analysis, first note that one of $A$ and $\bar{C}$ must exist (i.e. be nonempty) since $\eta$ is a non-split extension.
By gentleness, if both $A$ and $\bar{C}$ exist we must have $CA = 0$.

\smallskip 
\noindent \textbf{Case:} \textit{$\mu_1$ is a direct homotopy letter.}
\smallskip

By Corollary~\ref{cor:projective-resolution}, the homotopy string or band $\pi(w)$ has the following form:
\[
\pi(w) = 
\begin{cases}
\xymatrix{ \ar@{~}[r] & \xydot \ar[r]^-{\mu_1} & \xydot } 
& \text{if $\bar{C}w_R = \emptyset$ or is removed;} \\
\xymatrix{ \ar@{~}[r] & \xydot \ar[r]^-{\mu_1} & \xydot & \ar[l]_-{\overline{C p}} \xydot \ar@{~}[r] & }
& \text{for some path $p$ in $(Q,I)$ otherwise.}
\end{cases}
\]
Similarly, the homotopy string or band $\pi(v)$ has the following form:
\[
\pi(v) =
\begin{cases}
\xymatrix{ \ar@{~}[r] & \xydot \ar[r]^-{\mu_1} & \xydot & \ar[l]_-{\bar{C}} \xydot \ar@{~}[r] & }
& \text{if $A = \emptyset$;} \\
\xymatrix{ \ar@{~}[r] & \xydot \ar[r]^-{\mu_1 A q} & \xydot \ar@{~}[r] & }
& \text{for some path $q$ in $(Q,I)$ otherwise.}
\end{cases}
\]
Combining these, we get the following unfolded diagrams of graph map right endpoint conditions, showing the claim in this case.
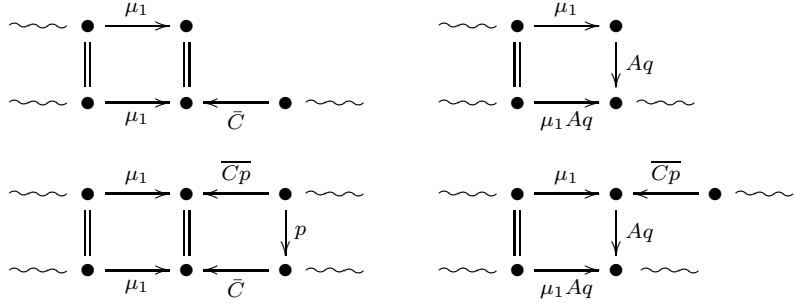
\begin{align*}
\xymatrix@!R=5px{
\ar@{~}[r] & \xydot \ar@{=}[d] \ar[r]^-{\mu_1} & \xydot \ar@{=}[d] &                                                    & \\
\ar@{~}[r] & \xydot \ar[r]_-{\mu_1}                  & \xydot                   & \ar[l]^-{\bar{C}} \xydot \ar@{~}[r] &
} 
& \quad
\xymatrix@!R=5px{
\ar@{~}[r] & \xydot \ar[r]^-{\mu_1}  \ar@{=}[d]  & \xydot \ar[d]^-{A q}   & \\
\ar@{~}[r] & \xydot \ar[r]_-{\mu_1 A q}              & \xydot \ar@{~}[r]       &
} \\
\xymatrix@!R=5px{
\ar@{~}[r] & \xydot \ar[r]^-{\mu_1} \ar@{=}[d] & \xydot \ar@{=}[d] & \ar[l]_-{\overline{C p}} \xydot \ar@{~}[r] \ar[d]^-{p} & \\
\ar@{~}[r] & \xydot \ar[r]_-{\mu_1}                  & \xydot                   & \ar[l]^-{\bar{C}} \xydot \ar@{~}[r]                            &
}
& \quad
\xymatrix@!R=5px{
\ar@{~}[r] & \xydot \ar[r]^-{\mu_1} \ar@{=}[d]  & \xydot \ar[d]^-{A q}  & \ar[l]_-{\overline{C p}} \xydot \ar@{~}[r] & \\
\ar@{~}[r] & \xydot \ar[r]_-{\mu_1 A q}             & \xydot \ar@{~}[r]      &                                                               &
}
\end{align*}

\smallskip
\noindent \textbf{Case:} \textit{$\mu_1$ is an inverse homotopy letter.}
\smallskip

\[
\pi(w) =
\begin{cases}
\xymatrix{ \ar@{~}[r] & \xydot } 
&  \parbox[t]{0.61\textwidth}{if $\bar{C} = \emptyset$ or $\bar{C}w_R$ is inverse, and there is no $a \in Q_1$ with $w \bar{a}$ defined as a string;} \\
\xymatrix{ \ar@{~}[r] & \xydot & \ar[l]_-{\mu_1 \bar{a}} \xydot \ar@{~}[r] &}
& \text{if $\bar{C} = \emptyset$ and there exists $ a \in Q_1$ with $\mu_1 \bar{a}$ a string;} \\
\xymatrix{ \ar@{~}[r] & \xydot & \ar[l]_-{\mu_1 \bar{C} \bar{p}} \xydot \ar@{~}[r] & }
& \text{for some (possibly trivial) path $p$ in $(Q,I)$, otherwise,}
\end{cases}
\]
where the homotopy string in the first case starts with $\mu_2$ if it exists, or a single projective or the start of an antipath otherwise.
Similarly, the homotopy string or band $\pi(v)$ has the following form:
\[
\pi(v) = 
\begin{cases}
\xymatrix{ \ar@{~}[r] & \xydot & \ar[l]_-{\mu_1} \xydot \ar[r]^-{A q} & \xydot \ar@{~}[r] & }
& \text{for some path $q$ in $(Q,I)$ if $A \neq \emptyset$;} \\
\xymatrix{ \ar@{~}[r] & \xydot & \ar[l]_-{\mu_1 \bar{C}} \xydot \ar@{~}[r] & }
& \text{if $A = \emptyset$.}
\end{cases}
\]
We leave it to the reader to match up the various forms of the projective resolutions and confirm that they give rise to graph map right endpoint conditions as above.

Now examining the components of $\f \colon \Q_{\pi(w)} \to \Q_{\pi(v)}$ consisting of identity maps between indecomposable projective modules and following these maps through a calculation of the kind in Lemma~\ref{Graph map in antipaths} shows that the $H^0(\f) = f \colon M(w) \to M(v)$, i.e. $\f$ is indeed induced from $f$.
\end{proof}

Applying Remark~\ref{rem:quasi-graph} we get the following corollary.

\begin{corollary} \label{cor:zigzag}
Keep the setup as in Lemma~\ref{lem:ext-to-graph-map}. The map $f \colon M(w) \to M(v)$ induces a quasi-graph map $\phi \colon \Q_{\pi(v)} \wiggle \Q_{\pi(w)}$ of homotopy string or band complexes, and hence a homotopy family of maps $\Q_{\pi(v)} \to \Sigma \Q_{\pi(w)}$. 
\end{corollary}

Let $\g \colon \Q_{\pi(v)} \to \Sigma \Q_{\pi(w)}$ be a representative of the homotopy family of single or double maps defined by the quasi-graph map $\phi \colon \Q_{\pi(v)} \wiggle \Q_{\pi(w)}$ obtained in Corollary~\ref{cor:zigzag} above. Then, by Proposition~\ref{prop:quasi-overlaps} one obtains $\Phi(\g) = \eta$.

\subsubsection{Direct or inverse overlaps} \label{sec:direct-overlap}

Here we consider the case of Setup~\ref{setup:overlap} in which $m$ is a direct overlap; the case that $m$ is an inverse overlap is analogous. 
As in previous sections $\sigma = \pi(v)$ and $\tau = \pi(w)$. 
Again, we use the combinatorics of the overlap to define a map $\g \colon \Q_{\pi(v)} \to \Sigma \Q_{\pi(w)}$ such that $\Phi(\g) = \eta$. In this case, $\g$ is either a singleton single map or a representative of a homotopy family of maps defined by a quasi-graph map $\phi \colon \Q_{\pi(v)} \wiggle \Q_{\pi(w)}$. 

In the following we do a detailed analysis of the  different types of standard basis maps which are induced by the different possible forms the strings $v$ and $w$ can take. We present the results by grouping the different cases giving rise to the same type of standard basis element in  $\Hom_{\KminusL}(\Q_{\pi(v)}, \Sigma \Q_{\pi(w)})$.

\smallskip
\noindent \textbf{Case:} \textit{$\g \colon \Q_{\pi(v)} \to \Sigma \Q_{\pi(w)}$ is a singleton single map.}
\smallskip

The unfolded diagram of the singleton single map is one of the diagrams below; we explain in which cases they arise.
In each case the precise description of $\tau_R$ is irrelevant, we note only that in each case it is necessarily empty or an inverse homotopy letter not containing $m$ as a substring, or vice versa.
\[
(I) \xymatrix@!R=5px{
\sigma \colon          \ar@{~}[r] & \xydot                          & \ar[l]_-{\bar{q}\bar{B}} \xydot \ar[r]^-{mAp} \ar[d]^-{m} & \xydot \ar@{~}[r]                        & \\
\Sigma \tau \colon  \ar@{~}[r] & \xydot \ar[r]_-{\tau_L}  & \xydot                                                                           & \ar[l]^-{\tau_R} \xydot \ar@{~}[r] &
}
\quad
(II) \xymatrix@!R=5px{
\sigma \colon          \ar@{~}[r] & \xydot                          & \ar@{-}[l] \xydot \ar[r]^-{DmAp} \ar[d]^-{Dm} & \xydot \ar@{~}[r]                        &  \\
\Sigma \tau \colon  \ar@{~}[r] & \xydot \ar[r]_-{\tau_L}  & \xydot                                                                           & \ar[l]^-{\tau_R} \xydot \ar@{~}[r] &
},
\]
where $p$ and $q$ are (possibly trivial) paths in $(Q,I)$.
Diagram (I) occurs precisely when both $A \neq \emptyset$ and $\bar{B} \neq \emptyset$: the pertinent part of the projective resolution of $M(v)$ has this form by Corollary~\ref{cor:projective-resolution}. Now, applying Corollary~\ref{cor:projective-resolution} to $w$ we see that,
\[
\tau_L =
\begin{cases}
dm             & \text{if $D = \emptyset$ but there exists $d \in Q_1$ with $dm$ defined as a string;} \\
q' Dm        & \parbox[t]{0.8\textwidth}{for some (possibly trivial) path $q'$ in $(Q,I)$ if $D \neq \emptyset$ and $w_L$ is not direct or $w_L D$ is direct and there exists $d \in Q_1$ with $dw$ defined as a string;} \\
\emptyset  & \text{otherwise.}
\end{cases}
\]
Diagram (II) occurs in the case that $A \neq \emptyset$ but $\bar{B} = \emptyset$; in this case to avoid $\eta$ being a split extension we must have $D \neq \emptyset$. In this case we have
\[
\tau_L =
\begin{cases}
\emptyset & \text{if $w_L D$ is direct and there exists no $ d \in Q_1$ with $d w$ defined as a string;} \\
q'Dm & \parbox[t]{0.8\textwidth}{for some nontrivial path $q'$ in $(Q,I)$ if the first letter of $w_L$ is not inverse and we are not in the case above.}
\end{cases}
\]
Note that in the case above when the first letter of $w_L$ is inverse, we do not get a singleton single map, hence this case is included in this argument but is treated in the next case below. In each case it is straightforward to verify that the diagram defines a singleton single map.
One now applies Proposition~\ref{prop:singleton-single-to-extensions} to see that $\Phi(\g) = \eta$.

\smallskip
\noindent \textbf{Case:} \textit{$\g \colon \Q_{\pi(v)} \to \Sigma \Q_{\pi(w)}$ is a representative of a homotopy family determined by a quasi-graph map $\phi \colon \Q_{\pi(v)} \wiggle \Q_{\pi(w)}$.}
\smallskip

We actually check that we get a graph map $\f \colon \Q_{\pi(w)} \to \Q_{\pi(v)}$ in the opposite direction and apply Remark~\ref{rem:quasi-graph}.

In the case that $A \neq \emptyset$ but $\bar{B} = \emptyset$, and the first letter of $w_L$ is inverse, i.e. the one case excluded in treating diagram (II) above, then we get the following graph map, in which $p$ is some (possibly trivial) path in $(Q,I)$.
\[
\xymatrix@!R=5px{
\pi(w) \colon \ar@{~}[r] & \xydot         & \ar[l] \xydot \ar[r]^-{Dm} \ar@{=}[d] & \xydot \ar[d]^-{Ap} & \ar[l] \xydot \ar@{~}[r] & \\
\pi(v) \colon \ar@{~}[r] & \xydot \ar[r] & \xydot \ar[r]_-{DmAp}                                    & \xydot \ar@{~}[r]                           &
}
\]

Now suppose $A = \emptyset$, whence $\bar{C} \neq \emptyset$. 
The overlap data gives rise to a graph map with the unfolded diagram,
\[
\xymatrix@!R=5px{
\pi(w) \colon \ar@{~}[r] & \xydot \ar@{-}[r] & \xydot \ar[r]^-{\tau_L} \ar[d]_-{f_L} & \xydot \ar@{=}[d] & \ar[l]_-{\tau_R} \xydot \ar[d]^-{f_R} \ar@{-}[r]^-{\tau_0} & \xydot \ar@{~}[r] & ,\\
\pi(v) \colon                  &  \ar@{~}[r] & \xydot \ar[r]_-{\sigma_L}                            &  \xydot                 & \ar[l]^-{\bar{C}} \xydot \ar@{~}[r]                                    &                           &
}
\]
in which,
\begin{align*}
\sigma_L & = \begin{cases}
m    & \text{if } \bar{B} \neq \emptyset; \\
Dm & \text{if } \bar{B} = \emptyset;
\end{cases} \\
\tau_R & = \begin{cases}
\bar{C}\bar{p} & \text{if } \bar{C}w_R \text{ is not removed when computing } \pi(w); \\
\emptyset        & \text{otherwise;}
\end{cases} \\
\tau_L & = \begin{cases}
\emptyset & \text{if } w_L D m \text{ is removed when computing } \pi(w); \\
qDm         & \text{otherwise,}
\end{cases}
\end{align*}
where $p$ is some (possibly trivial) path in $(Q,I)$ and $q$ is either the (possibly trivial) maximal direct prefix of $w_L$ if $w_L \neq \emptyset$ is not a direct string, and $q = d w_L$ for some $d \in Q_1$ with $d w_L$ defined as a string otherwise; in this second case, $w_L$ is also possibly empty.
The values of $(f_L, f_R)$ in each case are recorded in Table~\ref{table} 
\begin{table}
\begin{center}
\begin{tabular}{ c | c  c }
$f_L$                          & $\sigma_L = m$ & $\sigma_L = Dm$ \\
\hline
$\tau_L = \emptyset$ &  $\emptyset$  & $\emptyset$           \\
$\tau_L = qDm$         &  $qD$          &  $q$                         \\
\end{tabular}
\quad and \quad
\begin{tabular}{ c | c }
$f_R$                                & $\sigma_R = \bar{C}$ \\
\hline
$\tau_R = \bar{C}\bar{p}$ &  $p$                             \\
$\tau_R = \emptyset$        & $\emptyset$               \\
\end{tabular}
\end{center}
\caption{Left: the value of $f_L$ in each case; right: the value of $f_R$ in each case.}
\label{table}
\end{table}

\subsubsection{Trivial overlaps}

We finally turn our attention to trivial overlaps. Suppose $m = 1_x$ for some $x \in Q_0$. In this case, we fix the orientation of our strings and bands by requiring, whenever the relevant arrows exist, that $CB \neq 0$ and $DA \neq 0$. 
We again describe in each case how the combinatorics of the overlap can be used to construct a standard basis map $\g \colon \Q_{\pi(v)} \to \Q_{\pi(w)}$ such that $\Phi(\g) = \eta$.

\smallskip
\noindent \textbf{Case:} \textit{$\g \colon \Q_{\pi(v)} \to \Sigma \Q_{\pi(w)}$ is a graph map supported in one degree.}
\smallskip

This is simply a degeneration of diagram (I) in the singleton single map case of Section~\ref{sec:direct-overlap}, where instead $m = 1_x$ for some vertex $x \in Q_0$, i.e. providing a graph map concentrated in one degree.
Applying Lemma~\ref{lem:graph map single degree} we get $\Phi(\g) = \eta$.

\smallskip
\noindent \textbf{Case:} \textit{$\g \colon \Q_{\pi(v)} \to \Sigma \Q_{\pi(w)}$ is a singleton single map.}
\smallskip

If $A = \emptyset$ and $B \neq \emptyset$, in which case $\bar{C} \neq \emptyset$, then by Corollary~\ref{cor:projective-resolution}, the homotopy string $\pi(v)$ has the form
\[
\xymatrix{\pi(v) \colon \ar@{~}[r] & \xydot & \ar[l]_-{\bar{q} \bar{B} \bar{C}} \xydot & \ar[l]_-{\sigma_R} \xydot & \ar@{~}[l] } \, ,
\]
where $\sigma_R$ may be an empty homotopy letter. 
Similarly, the homotopy string $\pi(w)$ has the form 
\[
\xymatrix{\ar@{~}[r] & \xydot \ar[r]^-{\tau_L} & \xydot & \ar[l]_-{\bar{C} \bar{p}} \xydot \ar@{-}[r]^-{\tau_R} & \xydot \ar@{~}[r] & } 
\quad \text{or} \quad
\xymatrix{\ar@{~}[r] & \xydot \ar[r]^-{\tau_L} & \xydot } \, , 
\]
where $p$ is a (possibly trivial) path in $(Q,I)$, and $\tau_L$ and $\tau_R$ are possibly empty homotopy letters.
The form of $\tau_L$ depends on the form of the substring $w_L D$, but is not relevant for the description of the map.
The second case occurs when $w$ starts with $\bar{C}$ and we fall in case (3) or (4) of Corollary~\ref{cor:projective-resolution}.
In the case that $p$ is nontrivial, we get the unfolded diagram on the left below. In the case that $\pi(w)$ starts with $\tau_L$, we get the unfolded diagram on the right. In both cases we get a singleton single map.
\[  
\xymatrix@!R=4px{
\pi(v) \colon \ar@{~}[r]              & \xydot                         & \ar[l]_-{\bar{q} \bar{B} \bar{C}} \xydot \ar[d]^-{C} & \ar[l]_-{\sigma_R} \xydot                                        & \ar@{~}[l] & \\
\Sigma \pi(w) \colon \ar@{~}[r] & \xydot \ar[r]_-{\tau_L} & \xydot                                                                  & \ar[l]^-{\bar{C} \bar{p}} \xydot \ar@{-}[r]_-{\tau_R} & \xydot \ar@{~}[r] & 
}
\quad \text{or} \quad
\xymatrix@!R=4px{
\pi(v) \colon \ar@{~}[r]        & \xydot                 & \ar[l]_-{\bar{q} \bar{B} \bar{C}} \xydot \ar[d]^-{C} & \ar[l]_-{\sigma_R} \xydot  & \ar@{~}[l] & \\
\Sigma \pi(w) \colon \ar@{~}[r] & \xydot \ar[r]_-{\tau_L} & \xydot                                                                           &            & 
}
\]
The case that $p$ is trivial gives rise to a quasi-graph map, which is dealt with below.
There are obvious dual considerations when $A \neq \emptyset$ and $B = \emptyset$.
Now apply Proposition~\ref{prop:singleton-single-to-extensions}.

\smallskip
\noindent \textbf{Case:} \textit{$\g \colon \Q_{\pi(v)} \to \Sigma \Q_{\pi(w)}$ is a representative of a homotopy family determined by a quasi-graph map $\phi \colon \Q_{\pi(v)} \wiggle \Q_{\pi(w)}$.}
\smallskip

In the case that $A = \emptyset$ but $B \neq \emptyset$ above, in which the path $p$ occurring in the homotopy string $\pi(w)$ is trivial, we must have that $\tau_R \neq \emptyset$ and is direct by Corollary~\ref{cor:projective-resolution}. This gives rise to a graph map $\f \colon \Q_{\pi(w)} \to \Q_{\pi(v)}$ given by the following unfolded diagram.
\[
\xymatrix@!R=5px{
\pi(w) \colon \ar@{~}[r] & \xydot \ar[r]^-{\tau_L} & \xydot \ar[d]_-{Bq} & \ar[l]_-{\bar{C}} \xydot \ar@{=}[d] \ar@{-}[r]^-{\tau_R} & \xydot \ar@{~}[r]              & \\
\pi(v) \colon                 & \ar@{~}[r]                    & \xydot                    & \ar[l]^-{\bar{q} \bar{B} \bar{C}} \xydot                           & \ar[l]^-{\sigma_R} \xydot & \ar@{~}[l] 
}
\]
By Remark~\ref{rem:quasi-graph}, this gives rise to the quasi-graph map $\phi \colon \Q_{\pi(v)} \wiggle \Q_{\pi(w)}$, as claimed. Indeed, one can see that the map given in the unfolded diagram  above is one member of the homotopy family determined by $\phi$. Dual considerations apply for the case $A \neq \emptyset$ and $\B = \emptyset$.

Finally, the case $A = \emptyset$ and $B = \emptyset$ gives rise to a graph map $\f \colon \Q_{\pi(w)} \to \Q_{\pi(v)}$, whence a quasi-graph map $\phi \colon \Q_{\pi(v)} \wiggle \Q_{\pi(w)}$ by Remark~\ref{rem:quasi-graph}. Note that, necessarily, $C \neq \emptyset$ and $D \neq \emptyset$. In this case, by Corollary~\ref{cor:projective-resolution}, $\pi(v)$ has the form,
\[
\xymatrix{\pi(v) \colon \ar@{~}[r] & \xydot \ar[r]^-{\sigma_L} & \xydot \ar[r]^-{D} & x & \ar[l]_-{\bar{C}} \xydot & \ar[l]_-{\sigma_R} \xydot & \ar@{~}[l]} \, ,
\]
in which the homotopy letters $\sigma_L$ and $\sigma_R$ may be empty.
The homotopy string $\pi(w)$ has one of the following four forms
\begin{align*}
& \xymatrix{\ar@{~}[r] & \xydot \ar@{-}[r]^-{\tau_L} & \xydot \ar[r]^-{qD} & x & \ar[l]_-{\bar{C} \bar{p}} \xydot & \ar@{-}[l]_-{\tau_R} \xydot & \ar@{~}[l] } 
\quad &
\xymatrix{x} \hspace{3.7cm} \\
& \xymatrix{ \ar@{~}[r] & \xydot \ar@{-}[r]^-{\tau_L} & \xydot \ar[r]^-{qD} & x }
\quad &
\xymatrix{ x & \ar[l]_-{\bar{C} \bar{p}} \xydot & \ar[l]_-{\tau_R} \xydot & \ar@{~}[l] }
\end{align*}
where $p$ and $q$ are (possibly trivial) paths in $(Q,I)$. Whenever $p$ is trivial $\tau_R \neq \emptyset$ and is direct; whenever $q$ is trivial $\tau_L \neq \emptyset$ and is inverse. The graph map $\f \colon \Q_{\pi(w)} \to \Q_{\pi(v)}$ can be read off from the following unfolded diagram, interpreting $p$ and $q$ as trivial paths (whence isomorphisms) and deleting homotopy letters as appropriate to fit the cases.
\[
\xymatrix@!R=5px{
\pi(w) \colon \ar@{~}[r] & \xydot \ar@{-}[r]^-{\tau_L} & \xydot \ar[d]_-{q} \ar[r]^-{qD} & x \ar@{=}[d] & \ar[l]_-{\bar{C} \bar{p}} \xydot \ar[d]^-{p} & \ar@{-}[l]_-{\tau_R} \xydot & \ar@{~}[l] \\
\pi(v) \colon \ar@{~}[r]  & \xydot \ar[r]^-{\sigma_L}    & \xydot \ar[r]^-{D}                   & x                   & \ar[l]_-{\bar{C}} \xydot                             & \ar[l]_-{\sigma_R} \xydot  & \ar@{~}[l]
}
\]
As above, we apply Proposition~\ref{prop:quasi-overlaps} to get $\Phi(\g) = \eta$.

\subsection{Arrow Extensions}

Let $ v = v_m \cdots v_1$ and $w = w_n \cdots w_1$ where $v_i, w_i \in Q_1 \cup \bar{Q}_1$. 
Suppose that $\eta \in \Ext^1_\Lambda (M(v),M(w))$ is an arrow extension corresponding to an arrow $a \in Q_1$, i.e. $\eta$ corresponds to an extension with $M(u)$ as the middle term where $u = w a v$. 

Since we know $a v$ is defined as a string, then we are in case (1) or (3) in Corollary~\ref{cor:projective-resolution} so that $\pi(v) = \direct(a) \widetilde{v}$, where $\widetilde{v} = v' \inverse(b)$ for some $\bar{b} \in Q_1$ or $\widetilde{v} = v'$ depending on whether we fall into case (1) or (3), respectively. 
We set $\direct(a) = \cdots \theta_2 \theta_1 a$.
Likewise,
\[
\pi(w) =
\begin{cases}
\widetilde{w}\, \inverse(c) & \text{if there exists $c \in Q_1$ such that $w_1 \bar{c}$ is defined as a string;} \\
\widetilde{w}                   & \text{otherwise,}
\end{cases}
\]
where $\widetilde{w}$ is defined in a manner analogous to $\widetilde{v}$, depending on considerations at its end. We write $\inverse(c) = \bar{c} \bar{\phi}_1 \cdots \bar{\phi}_2 \cdots$.

The form of the map $\g \colon \Q_{\pi(v)} \to \Q_{\pi(w)}$ such that $\Phi(\g) = \eta$ depends on whether 
$v$ ends with an inverse or direct letter and
$w$ starts with an inverse or direct letter. We deal with the cases in turn.

\smallskip
\noindent \textbf{Case:} \textit{$w_1 \in Q_1$ and $v_m \in \bar{Q}_1$.}
\smallskip

If $\widetilde{w}\, \inverse(c)$ is defined, then we get the unfolded diagram of a (one-sided) graph map, $\g \colon \Q_{\pi(v)} \to \Sigma \Q_{\pi(w)}$, below, where we have used $\overline{\pi(w)}$ in the diagram.
\[ 
\xymatrix@!R=5px{
\Q_{\pi(v)} \colon \ar[d]_-{\g} & \ar@{~}[r] & \xydot \ar[r]^-{\theta_2} \ar@{=}[d] & \xydot \ar[r]^-{\theta_ 1} \ar@{=}[d] & \xydot \ar[r]^-{a} \ar@{=}[d]  & \xydot   & \ar[l]_-{v_m \cdots v_i}    \ar@{~}[r]   & \\
\Sigma \Q_{\pi(w)} \colon      & \ar@{~}[r] & \xydot \ar[r]_-{\phi_1}                 & \xydot \ar[r]_-{c = \theta_1}                   & \xydot                  & \xydot \ar[l]^-{\bar{w}_1  \cdots \bar{w}_j}                                  \ar@{~}[r]        & 
}
\]
Since  $ w_1 a $ is defined as a string, we have $w_1 a \notin I$,  whence $c = \theta_1$ by gentleness. Continuing, we see that $\phi_i = \theta_{i+1}$ for each $i > 1$.
Applying Lemma~\ref{Graph map in antipaths} one verifies that $\Phi(\g) = \eta$.

If $\widetilde{w}\, \inverse(c)$ is not defined, then we get the following unfolded diagram of a (one-sided) graph map $\g \colon \Q_{\pi(v)} \to \Sigma \Q_{\pi(w)}$ supported in one degree; applying Lemma~\ref{Graph map in antipaths} shows $\Phi(\g) = \eta$.
\[ 
\xymatrix@!R=5px{
\Q_{\pi(v)} \colon  \ar[d]_-{\g}    & \xydot \ar[r]^-{a} \ar@{=}[d]  & \xydot   & \ar[l]_-{v_m \cdots v_i}  \xydot       \ar@{~}[r]   & \\
\Sigma \Q_{\pi(w)} \colon  & \xydot                  & \xydot \ar[l]^-{\bar{w}_1  \cdots \bar{w}_j}                          \ar@{~}[r]        & 
}
\]
Note that since $w_1 a \notin I$ then $\theta_1 = \emptyset$ (i.e. $\direct(a) = a$) because otherwise $\theta_1$ would provide such a $c$ by gentleness of $\Lambda$.

\smallskip
\noindent \textbf{Case:} \textit{$w_1 \in Q_1$ and $v_m \in Q_1$.}
\smallskip

By the same argument as above, we have one of the following unfolded diagram of a (one-sided) graph map, $\g \colon \Q_{\pi(v)} \to \Sigma \Q_{\pi(w)}$, depending on whether $\widetilde{w}\, \inverse(c)$ is defined.
In both cases, one then applies Lemma~\ref{Graph map in antipaths}.
\[
\xymatrix@!R=5px{
\Q_{\pi(v)} \colon \ar[d]_-{\g} & \ar@{~}[r] & \xydot \ar[r]^-{\theta_2} \ar@{=}[d] & \xydot \ar[r]^-{\theta_ 1} \ar@{=}[d] & \xydot \ar[r]^-{av_m \cdots v_i} \ar@{=}[d]  & \xydot       \ar@{~}[r]   & \\
\Sigma \Q_{\pi(w)} \colon      & \ar@{~}[r] & \xydot \ar[r]_-{\phi_1}                 & \xydot \ar[r]_-{c = \theta_1}                   & \xydot                  & \xydot \ar[l]^-{\bar{w}_1  \cdots \bar{w}_j}                               \ar@{~}[r]        & 
}
\quad \text{or} \quad
\xymatrix@!R=5px{
\Q_{\pi(v)} \colon \ar[d]_-{\g} & \xydot \ar[r]^-{av_m \cdots v_i} \ar@{=}[d]  &  \xydot       \ar@{~}[r]                                       & \\
\Sigma \Q_{\pi(w)} \colon      & \xydot                                                         & \xydot \ar[l]^-{\bar{w}_1  \cdots \bar{w}_j}  \ar@{~}[r]        & 
}
\]

\smallskip
\noindent \textbf{Case:} \textit{$w_1 \in \bar{Q}_1$ and $v_m \in \bar{Q}_1$.}
\smallskip

Suppose $\widetilde{w}\, \inverse(c)$ is defined. Since  $ \theta_1 a  \in I$ we have that $\theta_1 \bar{w}_1$ is a string and $c = \theta_1$ is the unique arrow such that $c \bar{w}_1 \notin I$. Continuing we have $\phi_i = \theta_{i+1}$ for $i \geq 1$. This gives the following unfolded diagram of a (one-sided) graph map, $\g \colon \Q_{\pi(v)} \to \Sigma \Q_{\pi(w)}$; now apply Lemma~\ref{Graph map in antipaths} again.
\[ 
\xymatrix@!R=5px{
\Q_{\pi(v)} \colon \ar[d]_-{\g} & \ar@{~}[r] & \xydot \ar[r]^-{\theta_2} \ar@{=}[d] & \xydot \ar[r]^-{\theta_ 1 = c}  \ar@{=}[d]  & \xydot \ar[r]^-{a} \ar[d]^{\bar{w}_1  \cdots \bar{w}_j} & \xydot   & \ar[l]_-{v_m \cdots v_i} \xydot       \ar@{~}[r]   & \\
\Sigma \Q_{\pi(w)} \colon      & \ar@{~}[r] & \xydot \ar[r]_-{\phi_1}                 & \xydot \ar[r]_-{c \bar{w}_1  \cdots \bar{w}_j}                   & \xydot                  & \xydot \ar[l]^-{}                                \ar@{~}[r]        & 
}
\]

If  $\widetilde{w}\, \inverse(c)$ is not defined, then suppose $w_j \cdots w_1$ is the maximal inverse substring starting $w$, in particular, $\widetilde{w}$ starts with $w_{j+1}$ which is either direct or empty. Furthermore, $\theta_1 = \emptyset$ for otherwise $w_1 \bar{\theta}_1$ would be defined as a string and we could take $c = \theta_1$. Hence we get the following unfolded diagram of a singleton single map $\g \colon \Q_{\pi(v)} \to \Sigma \Q_{\pi(w)}$ and we apply Proposition~\ref{prop:singleton-single-to-extensions}.
\[ 
\xymatrix@!R=5px{
\Q_{\pi(v)} \colon \ar[d]_-{\g} &   \xydot \ar[r]^-{a} \ar[d]^{\bar{w}_1  \cdots \bar{w}_j} & \xydot   & \ar[l]_-{v_m \cdots v_i} \xydot       \ar@{~}[r]   & \\
\Sigma \Q_{\pi(w)} \colon      &   \xydot                                                       & \xydot \ar[l]^-{\bar{w}_{j+1} \cdots \bar{w}_k}   \ar@{~}[r]        & 
}
\]

\smallskip
\noindent \textbf{Case:} \textit{$w_1 \in \bar{Q}_1$ and $v_m \in Q_1$.}
\smallskip 

Arguing as above, we get the following unfolded diagram of a (one-sided) graph map or a singleton single map, $\g \colon \Q_{\pi(v)} \to \Sigma \Q_{\pi(w)}$, when $\widetilde{w} \, \inverse(c)$ is defined and when it is not, respectively.
One then applies Lemma~\ref{Graph map in antipaths} or Proposition~\ref{prop:singleton-single-to-extensions}, respectively.
\begin{align*} 
& \xymatrix@!R=5px{
\Q_{\pi(v)} \colon \ar[d]_-{\g} & \ar@{~}[r] & \xydot \ar[r]^-{\theta_2} \ar@{=}[d] & \xydot \ar[r]^-{\theta_ 1 }  \ar@{=}[d]  & \xydot \ar[r]^-{av_m \cdots v_i} \ar[d]^{\bar{w}_1  \cdots \bar{w}_j} & \xydot   & \ar[l]_-{} \xydot     \ar@{~}[r]   & \\
\Sigma \Q_{\pi(w)} \colon      & \ar@{~}[r] & \xydot \ar[r]_-{\phi_1}                 & \xydot \ar[r]_-{c \bar{w}_1  \cdots \bar{w}_j}                   & \xydot                  & \xydot \ar[l]^-{}                              \ar@{~}[r]        & 
}
\quad \text{or} \\
& \xymatrix@!R=5px{
\Q_{\pi(v)} \colon \ar[d]_-{\g} & \xydot \ar[r]^-{av_m \cdots v_i} \ar[d]^{\bar{w}_1  \cdots \bar{w}_j} & \xydot   & \ar[l]_-{} \xydot     \ar@{~}[r]   & \\
\Sigma \Q_{\pi(w)} \colon      & \xydot                  & \xydot \ar[l]^-{\bar{w}_{j+1} \cdots \bar{w}_k}                             \ar@{~}[r]        & 
}
\end{align*}



\begin{thebibliography}{99}

\bibitem{ALP}
 K.~K.~Arnesen, R.~Laking, D~Pauksztello, \textit{Morphisms between indecomposable complexes in the bounded derived category of a gentle algebra}, J. Algebra \textbf{467} (2016), 1--46,
 also \harxiv{1411.7644}. 

\bibitem{ABCP}
I. Assem, T.  Br\"ustle, G. Charbonneau-Jodoin, P.-G. Plamondon, \textit{Gentle algebras arising from surface triangulations} Algebra Number Theory \textbf{4} (2010), no. 2, 201--229,
also \harxiv{0903.3347}/ 

\bibitem{AS}
I.~Assem, A.~Skowro\'nski, \textit{Iterated tilted algebras of type $\widetilde{A}$}, Math. Z. \textbf{195} (1987), 269--290.

\bibitem{BCS}
K.~Baur, R. Coelho Sim\~oes, \textit{A geometric model for the module category of a gentle algebra}, in press IMRN,  \href{https://doi.org/10.1093/imrn/rnz150}{\tt https://doi.org/10.1093/imrn/rnz150},   
also \harxiv{1803.05802}.

\bibitem{BM}
V.~Bekkert, H.~Merklen, \textit{Indecomposables in derived categories of gentle algebras}, Algebr. Represent. Theory \textbf{6} (2003), 285--302.

\bibitem{Bocklandt}
R. Bocklandt, \textit{Noncommutative mirror symmetry for punctured surfaces.} With an appendix by Mohammed Abouzaid. Trans. Amer. Math. Soc. \textbf{368} (2016), no. 1, 429--469,
also \harxiv{1111.3392}

\bibitem{BDMTY}
T. Br\"ustle, G. Douville, K. Mousavand, H. Thomas, E. Y\i ld\i r\i m, \textit{On the Combinatorics of Gentle Algebras}, in press Canadian J. Math. \href{https://doi.org/10.4153/S0008414X19000397}{\tt https://doi.org/10.4153/S0008414X19000397},
also \harxiv{1707.07665}. 

\bibitem{BR}
M.~C.~R.~Butler, C.~M.~Ringel, \textit{Auslander--Reiten sequences with few middle terms and applications to string algebras}, Comm. Algebra \textbf{15}, 145--179.

\bibitem{CPS}
\I.~\Canakci, D.~Pauksztello, S.~Schroll, \textit{Mapping cones in the bounded derived category of a gentle algebra},  J. Algebra \textbf{530} (2019), 163--194,
also \harxiv{1609.09688}.

\bibitem{Addendum}
\.{I}.~\c{C}anak\c{c}i, D.~Pauksztello, S.~Schroll, \textit{Addendum and Erratum: Mapping cones for morphisms involving a band complex in the bounded derived category of a gentle algebra}, \harxiv{2001.06435}.

\bibitem{CS}
\I.~\Canakci, S.~Schroll, \textit{Extensions in Jacobian algebras and cluster categories of marked surfaces},  Adv. Math. \textbf{313} (2017), 1--49,
also \harxiv{1408.2074}.

\bibitem{CB}
W.~W.~Crawley-Boevey, \textit{Maps between representations of zero-relation algebras}, J. Algebra \textbf{126} (1989), 259--263.

\bibitem{Garcia} 
A. Garcia Elsener, \textit{Gentle m-Calabi-Yau tilted algebras}, \harxiv{1701.07968}.

\bibitem{HKK} 
F. Haiden, L. Katzarkov, M. Kontsevich, \textit{Flat surfaces and stability structures}, Publ. Math. Inst. Hautes \'Etudes Sci. \textbf{126} (2017), 247--318,
also \harxiv{1409.8611}. 

\bibitem{Happel}
D.~Happel, ``Triangulated Categories in the Representation Theory of Finite Dimensional Algebras'', London Math. Soc. Lecture Notes Series \textbf{119}, Cambridge University Press (1988).

\bibitem{HK}
R. S. Huerfano, M. Khovanov, \textit{A category for the adjoint representation}, J. Algebra \textbf{246} (2001), no. 2, 514--542,
also \harxiv{math/0002060}. 

\bibitem{HZS} B. Huisgen-Zimmermann, S. O. Smal\o, \textit{The homology of string algebras. I}, J. Reine Angew. Math. \textbf{580} (2005), 1--37.

\bibitem{Kalck}
M. Kalck, \textit{Singularity categories of gentle algebras}, Bull. Lond. Math. Soc. \textbf{47} (2015), no. 1, 65--74,
also \harxiv{1207.6941}. 

\bibitem{Krause}
H.~Krause, \textit{Maps between tree and band modules}, J. Algebra \textbf{137} (1991), 186--194.

\bibitem{LF} 
D. Labardini-Fragoso, \textit{Quivers with potentials associated to triangulated surfaces}, Proc. Lond. Math. Soc. (3) \textbf{98} (2009), no. 3, 797--839,
also \harxiv{0803.1328}. 

\bibitem{LP}
Y. Lekili, A. Polishchuk, \textit{Derived equivalences of gentle algebras via Fukaya categories}, in press Math. Annalen, \href{https://doi.org/10.1007/s00208-019-01894-5}{\tt https://doi.org/10.1007/s00208-019-01894-5},
also \harxiv{1801.06370}.

\bibitem{MCC} T. McConville, \textit{Lattice structure of Grid-Tamari orders}, J. Combin. Theory Ser. A \textbf{148} (2017) 27--56,
also \harxiv{1504.05213}.

\bibitem{OPS} 
 S. Opper, P.-G. Plamondon, S. Schroll, \textit{A geometric model for the derived category of gentle algebras}, \harxiv{1801.09659}. 

\bibitem{PPP} 
Y. Palu, V. Pilaud, P.-G. Plamondon, \textit{Non-kissing complexes and $\tau$ tilting for gentle algebras}, to appear Mem. Amer. Math. Soc., \harxiv{1707.07574}. 

\bibitem{Schroer} 
J. Schr\"oer, \textit{Modules without self-extensions over gentle algebras}, J. Algebra \textbf{216} (1999), no. 1, 178--189.

\bibitem{SS} 
D. Simson,  A. Skowro\'nski, \textit{ Elements of the representation theory of associative algebras. Vol. 3. Representation-infinite tilted algebras}, London Mathematical Society Student Texts, 72. Cambridge University Press, Cambridge, 2007.

\bibitem{Vossieck} D. Vossieck, \textit{The algebras with discrete derived category}, J. Algebra \textbf{243} (2001) 168--176.

\bibitem{WW}
B. Wald, J. Waschb\"usch, \textit{Tame biserial algebras}, J. Algebra \textbf{95} (1985), 480--500.

\bibitem{Zhang}
J. Zhang, \textit{On the indecomposable exceptional modules over gentle algebras}, Comm. Alg. \textbf{42} (2014), 3096--3199. 

\end{thebibliography}
\end{document}